\documentclass[11pt,a4paper]{amsart}

\usepackage{amsfonts,amssymb,amsmath,amsthm}
\usepackage{mathtools,esint}
\usepackage{lmodern}
\usepackage[noadjust]{cite}
\usepackage{enumitem}
\usepackage{xcolor}
\usepackage[plainpages=false,pdfpagelabels,backref=page]{hyperref}
\hypersetup{
 colorlinks=true,
 linkcolor={cyan!90!black},
 citecolor={magenta},
 urlcolor={green!40!black},
 pdftitle={Solvability of the L2 Dirichlet problem for parabolic operators with spatially periodic coefficients},
 pdfauthor={Kaj Nystr\"om},
 pdfsubject={Parabolic measure, reverse Holder estimates, and spatially periodic coefficients}
}

\newtheorem{thm}{Theorem}[section]
\newtheorem{lem}[thm]{Lemma}
\newtheorem{prop}[thm]{Proposition}
\newtheorem{cor}[thm]{Corollary}

\theoremstyle{definition}
\newtheorem{defn}[thm]{Definition}
\newtheorem{rem}[thm]{Remark}
\numberwithin{equation}{section}

\usepackage{tikz}
\usetikzlibrary{arrows.meta}

\newcommand{\R}{\mathbb R}
\newcommand{\cH}{\mathcal H}
\newcommand{\cC}{\mathcal C}

\title[The \(\mathrm L^2\) Dirichlet problem for spatially periodic parabolic operators]
{Solvability of the \(\mathrm L^2\) Dirichlet problem for parabolic operators
with spatially periodic coefficients}
\author{Kaj Nystr\"om}
\address{Department of Mathematics, Uppsala University, Box 480\\
SE-751 06 Uppsala, Sweden}
\email{kaj.nystrom@math.uu.se}
\makeatletter
\@namedef{subjclassname@2020}{\textup{2020} Mathematics Subject Classification}
\makeatother
\date{}

\begin{document}

\begin{abstract}
We prove global reverse H\"older estimates in \(\mathrm L^2\) for the Poisson kernel of
scalar divergence-form parabolic operators in the upper half-space.  The
coefficient matrix is real, symmetric, and periodic in every spatial
variable \(X=(\lambda,x)\), while no periodicity in time is assumed.  If its boundary trace \(A^0=A^0(x,t)\) has a quantitative uniform modulus
of continuity in time and its convergence to \(A^0\) in the transverse
variable is controlled by a square-Dini modulus, then the Poisson kernel
satisfies a uniform \(\mathrm{RH}_2\) estimate at scales below the period.
Full spatial periodicity is then used
to propagate this estimate to every
scale.  The main new ingredient in the large-scale argument is a positive
adjoint height coordinate \(\Phi^*=\lambda+O(1)\).  Boundary comparison with \(\Phi^*\) yields the
required large-scale decay of the Green function and hence the global
\(\mathrm{RH}_2\) estimate.  Consequently, the \(\mathrm L^2\) Dirichlet problem is uniquely
solvable among weak solutions \(u\) satisfying \(N_\ast u\in\mathrm L^2\).  We also
establish, for fully spatially periodic coefficients, a nonsymmetric
\(A_\infty\) theorem under square-Dini convergence to a
transverse-independent boundary trace, with no temporal regularity beyond
measurability.  We explain how the results transfer to time-independent
Lipschitz graph domains under the stated coefficient and
periodicity-compatibility hypotheses, and we derive scale-uniform boundary
estimates for the periodic homogenization families considered below.  The condition
\(D_t^{1/2}A^0\in\mathrm L^\infty_x(\mathrm{BMO}_t)\) provides a natural
sufficient criterion for the temporal continuity required in the
symmetric theorem.
\end{abstract}

\maketitle


\section{Introduction and statement of the main results}
\label{sec:introduction}


A classical theorem of Dahlberg~\cite{Dahlberg} states that harmonic
measure in a Lipschitz domain is absolutely continuous with respect to
surface measure and that the associated Poisson kernel satisfies a
scale-invariant reverse H\"older estimate in \(\mathrm L^2\).
Equivalently, the Dirichlet problem with boundary data in
\(\mathrm L^2\) is solvable with \(\mathrm L^2\)-control of the
non-tangential maximal function.  Jerison and Kenig gave another proof of
Dahlberg's theorem and extended the argument to uniformly elliptic
symmetric operators with smooth coefficients in nonsmooth domains
\cite{JerisonKenig}.  Further developments for divergence-form
operators with coefficients independent of the transverse variable can
be found, for example, in
\cite{AlfonsecaAuscherAxelssonHofmannKim,
HofmannKenigMayborodaPipher}.

Kenig and Shen~\cite{KenigShen}, building on unpublished work of Dahlberg,
weakened the assumption of independence from the transverse variable.  They
considered real
symmetric elliptic operators whose coefficients are periodic in the
transverse variable and satisfy a quantitative regularity condition in that
variable.  Their proof separates estimates below the period from a
large-scale reduction based on transverse, or normal, differences.  In particular, it gives
the \(\mathrm L^2\) Dirichlet estimate under a suitable square-Dini condition in the
transverse direction.  More broadly, the fine properties of elliptic measure
have motivated an extensive literature.  We refer, for example,
to~\cite{KenigKirchheimPipherToro} and the references therein.

The corresponding theory for parabolic measure is less complete.  Fabes and
Salsa~\cite{FabesSalsa} proved a parabolic version of Dahlberg's theorem for
the heat equation in time-independent Lipschitz cylinders.  Related half-space
results can be found in~\cite{CastroNystromSande}.  Parabolic measure associated with the heat equation
in time-dependent Lipschitz-type domains was studied by Lewis and
Silver~\cite{LewisSilver}, Lewis and Murray~\cite{LewisMurray}, and Hofmann
and Lewis~\cite{HofmannLewis}.  These works identified the appropriate parabolic regularity of a
time-dependent graph and demonstrated the importance of quantitative
control of its variation in time.  The related work
\cite{HofmannLewis1}, which treats parabolic operators with singular
drift terms, should also be mentioned in this context.

For time-independent coefficients, a parabolic analogue of the periodic
elliptic argument of Kenig and Shen follows from the structural theorem of Litsg\aa rd
and the author~\cite{LitsgardNystrom}, which shows that a scale-invariant reverse
H\"older estimate for the elliptic Poisson kernel transfers to the parabolic
Poisson kernel of the corresponding time-independent parabolic operator.
The present work is concerned instead with coefficients that genuinely
depend on time.

For operators in the parabolic upper half-space whose coefficients are
independent of the transverse variable, substantial progress has been made
through square-function, layer-potential, and first-order methods.  We refer
to \cite{CastroNystromSande,NystromSquare,NystromL2,AEN,AENDirichlet} and the references
therein.  In particular, Auscher, Egert, and the author~\cite{AENDirichlet}
considered
\[
 \partial_tu-\operatorname{div}_{\lambda,x}
 \bigl(A(x,t)\nabla_{\lambda,x}u\bigr)=0
 \quad\text{in }\mathbb R^{n+2}_+,
\]
where \(A\) is real, bounded, measurable, uniformly elliptic, and not
necessarily symmetric.  They proved that the associated parabolic measure
belongs to \(A_\infty(\mathrm d x\,\mathrm d t)\).  Consequently, its Poisson kernel satisfies a reverse H\"older estimate
with some exponent \(p>1\), and the Dirichlet problem is solvable in
\(\mathrm L^q\), where \(q=p/(p-1)\).  In addition, the first-order
\(\mathrm L^2\) theory developed in~\cite{AEN} provides stability of the
adjoint regularity problem under small \(\mathrm L^\infty\)-perturbations
of \(\lambda\)-independent coefficients.  This is used in our small-scale
argument, while the transverse dependence is restored by the classical
small-Carleson perturbation theorem for parabolic measure recorded in
Proposition~\ref{prop:parabolic-Carleson-perturbation}\textup{(i)} below.

In this paper we study the operator
\begin{equation}\label{eq:operator}
 \cH u
 =\partial_tu-\operatorname{div}_{\lambda,x}
      \bigl(A(\lambda,x,t)\nabla_{\lambda,x}u\bigr),\quad\text{in }\mathbb R^{n+2}_+.
\end{equation}
In our main theorem, which concerns the class
\(B_2(\mathrm d x\,\mathrm d t)\), the coefficient matrix is assumed
real, bounded, measurable, uniformly elliptic, and symmetric.  Its boundary trace \(A^0(x,t)\) is assumed to be uniformly
continuous in time with values in \(\mathrm L^\infty(\mathbb R^n)\), with a
quantitative modulus, and its variation in the transverse variable is controlled
by a square-Dini modulus.  The condition
\[
 D_t^{1/2}A^0
 \in \mathrm L^\infty\bigl(\mathbb R^n;\mathrm{BMO}(\mathbb R)\bigr),
\]
which arises naturally in the first-order framework of~\cite{AEN},
implies the required temporal regularity and therefore remains an
important sufficient hypothesis.  For the large-scale part of the argument we also assume that
\(A\) is one-periodic in every spatial coordinate \(X=(\lambda,x)\).  No
periodicity in time is imposed.

Our main result states that the forward parabolic measure belongs to
\(B_2(\mathrm d x\,\mathrm d t)\) uniformly at every scale.  Equivalently,
its Poisson kernel satisfies a scale-invariant reverse H\"older estimate in
\(\mathrm L^2\).  Since the hypotheses are preserved by transposition and
time reversal, the same conclusion holds for the time-reversed adjoint
operator.  The standard forward-adjoint uniqueness criterion then gives
unique solvability of the \(\mathrm L^2\) Dirichlet problem
\cite[Remark~1.4]{AENDirichlet}.

The proof separates the two scale regimes.  Below the period, the uniform time
modulus permits the boundary trace to be frozen at the central time.
The small \(\mathrm L^\infty\)-perturbation theorem of~\cite{AEN}, applied
after time reversal to the adjoint regularity problem, together with
regularity-Dirichlet duality, gives uniform \(\mathrm L^2\) Dirichlet
solvability and hence, by Theorem~\ref{thm:Bp-Dq}, a uniform
\(B_2\)-estimate for the resulting \(\lambda\)-independent coefficient.
The original transverse
dependence is then restored by the classical small-Carleson \(B_2\)-perturbation
theorem, Proposition~\ref{prop:parabolic-Carleson-perturbation}\textup{(i)}.

Above the period, the direct elliptic transverse-difference argument originating
in the unpublished work of Dahlberg (see Kenig and Shen~\cite{KenigShen}) encounters a
genuinely parabolic obstruction.  A Green function, regarded as a function of
its source variables, solves the adjoint equation.  However, the parabolic
measure in the forward comparison argument represents solutions of the forward
equation.  Symmetry of the spatial coefficient matrix does not remove the
opposite time orientations.

We overcome this obstruction by constructing a positive adjoint height
coordinate \(\Phi^*\) satisfying
\[
 \cH^*\Phi^*=0,
 \qquad
 \Phi^*=0\quad\text{on }\{\lambda=0\},
 \qquad
 \bigl|\Phi^*(\lambda,x,t)-\lambda\bigr|\leq C.
\]
It is at this point that periodicity in the tangential spatial variables
first becomes essential in the large-scale argument.  Indeed, a nonautonomous adjoint corrector problem on the
spatial torus is used to construct \(\Phi^*\).  Once
\(\Phi^*\) is available, its subsequent use requires only its positivity,
its vanishing boundary trace, and the comparison
\[
 \Phi^*(\lambda,x,t)\simeq\lambda,
 \qquad \lambda\geq1.
\]
Boundary comparison between localized Green functions and \(\Phi^*\) then gives the
unit-height Green-function decay needed in the large-scale summation.

The construction and use of \(\Phi^*\) constitute the principal new
ingredient in the large-scale argument.  It allows us to implement a
parabolic analogue of the Green-function strategy in Dahlberg's
unpublished elliptic proof, which is reproduced in the appendix
of~\cite{KenigShen}.

The small-scale theorem does not require periodicity.  Transverse periodicity
alone still implies that the difference
\[
 u(\lambda+1,x,t)-u(\lambda,x,t)
\]
is a solution, but the present argument does not establish the global theorem without
periodicity in the tangential spatial variables.  The corresponding Neumann
and regularity problems are not addressed here.

\subsection{The coefficients and the parabolic geometry}\label{subsec:setting}

Throughout, \(n\geq1\), and
\[
 A:\R^{n+2}_+\longrightarrow\mathbb R^{(n+1)\times(n+1)},
\]
where
\[
 \R^{n+2}_+=\{(\lambda,x,t):\lambda>0,\ x\in\R^n,\ t\in\R\},
 \qquad \partial\R^{n+2}_+=\R^n\times\R.
\]
We consider the scalar operator \(\cH\) introduced in \eqref{eq:operator}.  The matrix \(A\) is assumed real and measurable, and there are constants
\(0<\mu\leq \Lambda<\infty\) such that, for almost every
\((\lambda,x,t)\) and every \(\xi\in\R^{n+1}\),
\begin{equation}\label{eq:ellipticity}
 \mu |\xi|^2\leq A(\lambda,x,t)\xi\cdot\xi,
 \qquad |A(\lambda,x,t)\xi|\leq \Lambda|\xi|.
\end{equation}

A function \(u\) is a weak solution of \(\cH u=0\) in an open set
\(\Omega\subset\R^{n+2}_+\) if
\[
 u\in\mathrm L^2_{\mathrm{loc}}(\Omega),
 \qquad
 \nabla_{\lambda,x}u\in
 \mathrm L^2_{\mathrm{loc}}(\Omega;\mathbb R^{n+1}),
\]
and
\begin{equation}\label{eq:weak-solution}
 \iiint_{\Omega}
 A\nabla_{\lambda,x}u\cdot\nabla_{\lambda,x}\varphi
 -u\,\partial_t\varphi
 \,\mathrm d \lambda\,\mathrm d x\,\mathrm d t=0
\end{equation}
for every \(\varphi\in C_0^\infty(\Omega)\).  The adjoint operator is
\[
 \cH^*=-\partial_t-\operatorname{div}_{\lambda,x}
       (A^{\mathsf T}\nabla_{\lambda,x}),
\]
with the analogous weak formulation.

For the transverse-difference discussion it is enough to assume
\begin{equation}\label{eq:periodicity}
 A(\lambda+1,x,t)=A(\lambda,x,t)
 \quad\text{for almost every }(\lambda,x,t).
\end{equation}
The large-scale theorem proved below uses the stronger assumption of full
spatial periodicity.  Whenever \eqref{eq:periodicity} is assumed, we identify
\(A\) with its one-periodic extension in \(\lambda\) to all
\(X=(\lambda,x)\in\mathbb R^{n+1}\).
Full spatial periodicity then means that this extension is real, bounded,
measurable, uniformly elliptic, and satisfies
\begin{equation}\label{eq:full-periodicity}
 A(X+k,t)=A(X,t)
 \quad\text{for almost every }(X,t)
 \text{ and every }k\in\mathbb Z^{n+1}.
\end{equation}
In particular, \eqref{eq:full-periodicity} implies
\eqref{eq:periodicity}.  Here ``full'' refers to all spatial variables,
including the transverse variable \(\lambda\), and period one is only a
normalization.  No periodicity in \(t\) is required.

We use the boundary parabolic quasi-distance
\[
 d_p((x,t),(y,s))=|x-y|+|t-s|^{1/2}.
\]
For \((x,t)\in\R^n\times\R\) and \(r>0\), set
\begin{equation}\label{eq:boundary-cubes}
 \begin{aligned}
  Q_r(x)&=B(x,r),
  &I_r(t)&=(t-r^2,t+r^2),\\
  \Delta_r(x,t)&=Q_r(x)\times I_r(t),
  &\ell(\Delta_r)&=r.
 \end{aligned}
\end{equation}
If \(c>0\), then \(cQ_r(x)=B(x,cr)\),
\(c^2I_r(t)=(t-c^2r^2,t+c^2r^2)\), and
\[
 c\Delta_r(x,t)
 :=cQ_r(x)\times c^2I_r(t)=\Delta_{cr}(x,t).
\]
Thus \(|\Delta_r(x,t)|=2|B(0,1)|r^{n+2}\), and the homogeneous
dimension of the boundary space is \(n+2\).  We also use the oriented
halves
\begin{equation}\label{eq:oriented-boundary-boxes}
 \Delta_r^-(x,t)=Q_r(x)\times(t-r^2,t),
 \qquad
 \Delta_r^+(x,t)=Q_r(x)\times(t,t+r^2),
\end{equation}
with the convention
\[
 c\Delta_r^\pm(x,t):=\Delta_{cr}^\pm(x,t).
\]
We write
\[
 T_r(x,t)=(0,r)\times\Delta_r(x,t),
 \qquad
 T_r^\pm(x,t)=(0,r)\times\Delta_r^\pm(x,t).
\]
The symmetric cube \(\Delta_r\) is the default boundary cube.  The
superscripts \(+\) and \(-\) are retained whenever the time orientation
matters.  In particular, a symmetric space-time box is an ambient
region and is not, by itself, the incoming parabolic boundary of a
forward problem.

\subsection{Parabolic measure and reverse H\"older estimates}\label{subsec:parabolic-measure}

For any coefficient matrix \(C\) and admissible pole \(P\), write
\[
 \cH_C
 :=
 \partial_t-\operatorname{div}_{\lambda,x}
 \bigl(C\nabla_{\lambda,x}\bigr),
 \qquad
 \omega_C^P:=\omega_{\cH_C}^P.
\]
Whenever \(\omega_C^P\ll\mathrm d x\,\mathrm d t\), denote its Poisson
kernel by
\[
 k_C^P
 :=
 \frac{\mathrm d\omega_C^P}{\mathrm d x\,\mathrm d t}.
\]
For the principal coefficient matrix \(A\), so that
\(\cH=\cH_A\), we abbreviate
\[
 \omega^P:=\omega_A^P,
 \qquad
 k^P:=k_A^P.
\]
\begin{defn}\label{defn:admissible-pole}
Fix structural constants \(0<c_a<1<C_a\).  Given
\(\Delta_r=\Delta_r(x_0,t_0)\) and \(M>4\), we call
\(P=(\lambda_P,x_P,t_P)\) an \(M\)-admissible pole if
\begin{equation}\label{eq:admissible-pole}
 \lambda_P\geq Mr,\qquad
 |x_P-x_0|\leq C_a\lambda_P,\qquad
 c_a\lambda_P^2\leq t_P-t_0\leq C_a\lambda_P^2.
\end{equation}
\end{defn}
\begin{rem}\label{rem:admissible-pole}
In Definition~\ref{defn:admissible-pole}, the constant \(M\) is chosen sufficiently large relative to all fixed
geometric parameters used below.  In particular, we require
\(c_aM^2>64\).  Thus \(P\) lies quantitatively in the parabolic future
of \(\Delta_r\), and every fixed enlargement of \(T_r(x_0,t_0)\) used
below is separated from \(P\).
The same terminology will be used for poles joined to this region by a
uniformly bounded, correctly oriented Harnack chain.
\end{rem}

\begin{defn}\label{defn:local-global-rh2}
We say that the local \(\mathrm{RH}_2\) estimate holds up to scale one if there are
\(M>4\) and \(C_0<\infty\) such that, for every
\((x_0,t_0)\in\R^n\times\R\), every \(0<r\leq1\), and every
\(M\)-admissible pole \(P\) for \(\Delta_r(x_0,t_0)\), one has
\(\omega_{\cH}^P\ll\mathrm d x\,\mathrm d t\), and its Poisson kernel
satisfies
\begin{equation}\label{eq:local-rh}
 \left(
 \frac{1}{|\Delta_r|}
 \iint_{\Delta_r(x_0,t_0)}(k^P)^2\,\mathrm d x\,\mathrm d t
 \right)^{1/2}
 \leq
 \frac{C_0}{|\Delta_r|}
 \iint_{\Delta_r(x_0,t_0)}k^P\,\mathrm d x\,\mathrm d t.
\end{equation}
The global estimate is the same assertion, including absolute
continuity, for every \(r>0\).
\end{defn}

Lemma~\ref{lem:cube-conventions} below shows that the formulations with
symmetric cubes \(\Delta_r\) and with the oriented halves
\(\Delta_r^-\) are quantitatively equivalent.  We use symmetric cubes in
the statements and backward cubes in the large-scale argument when their
causal orientation is useful.

\subsection{The small-scale regularity assumptions}\label{subsec:small-assumptions}

For the small-scale theorem we assume in addition that
\(A=A^{\mathsf T}\) and that \(A\) has a bounded, symmetric, uniformly
elliptic boundary trace \(A^0=A^0(x,t)\).  We assume that the map
\[
 t\longmapsto A^0(\cdot,t)\in\mathrm L^\infty(\R^n)
\]
has a uniformly continuous representative.  For this representative, set
\begin{equation}\label{eq:boundary-time-modulus}
 \varpi(\rho)
 :=
 \sup_{\substack{s,t\in\R\\ |t-s|^{1/2}\leq\rho}}
 \bigl\|A^0(\cdot,t)-A^0(\cdot,s)\bigr\|_{\mathrm L^\infty(\R^n)},
 \qquad \rho>0,
\end{equation}
and assume that
\begin{equation}\label{eq:boundary-time-continuity}
 \lim_{\rho\downarrow0}\varpi(\rho)=0.
\end{equation}
The function \(\varpi\) is nondecreasing and bounded by a constant depending
only on \(\Lambda\).  The representative may be chosen so that
\(A^0(\cdot,t)\) is symmetric and has the same ellipticity bounds for every
\(t\in\R\).

A useful sufficient condition for
\eqref{eq:boundary-time-continuity} is the half-order BMO condition
\begin{equation}\label{eq:boundary-fiber-bmo}
 K_0:=
 \operatorname*{ess\,sup}_{x\in\R^n}
 \bigl\|D_t^{1/2}A^0(x,\cdot)\bigr\|_{\mathrm{BMO}(\R_t)}
 <\infty.
\end{equation}
Here \(D_t^{1/2}\) denotes the half-order derivative in time, defined by
the Fourier multiplier \(|\tau|^{1/2}\), and \(H_t\) denotes the Hilbert
transform in the time variable, with its sign chosen so that \(\partial_t=D_t^{1/2}H_tD_t^{1/2}\).  Lemma~\ref{lem:temporal-freezing} below shows that
\eqref{eq:boundary-fiber-bmo} implies
\begin{equation}\label{eq:bmo-time-modulus}
 \varpi(\rho)\leq CK_0\rho,
 \qquad \rho>0.
\end{equation}
Thus \eqref{eq:boundary-fiber-bmo} remains an important sufficient
hypothesis, but it is not required in the main theorem.  Notice that it is an
\(\mathrm L^\infty_x(\mathrm{BMO}_t)\)-condition and not merely a joint
parabolic BMO condition in \((x,t)\).

The oscillation in the transverse variable is measured by
\begin{equation}\label{eq:normal-modulus}
 \eta(\rho)
 :=
 \operatorname*{ess\,sup}_{0<\lambda\leq\rho}
 \bigl\|
 A(\lambda,\cdot,\cdot)-A^0
 \bigr\|_{\mathrm L^\infty(\R^n\times\R)}.
\end{equation}
We impose the square-Dini condition
\begin{equation}\label{eq:normal-square-Dini}
 \mathfrak D(1)
 :=
 \left(
  \int_0^1\eta(\rho)^2\,\frac{\mathrm d \rho}{\rho}
 \right)^{1/2}
 <\infty.
\end{equation}
More generally, for \(0<R\leq1\), define
\begin{equation}\label{eq:Dini-tail-zero}
 \mathfrak D(R)
 :=
 \left(
  \int_0^R\eta(\rho)^2\,\frac{\mathrm d \rho}{\rho}
 \right)^{1/2}.
\end{equation}
The absolute continuity of the integral in
\eqref{eq:normal-square-Dini} implies
\[
 \lim_{R\downarrow0}\mathfrak D(R)=0.
\]
Moreover, since \(\eta\) is nondecreasing, for
\(0<\rho\leq1/2\) we have
\[
 \eta(\rho)^2\log 2
 \leq
 \int_\rho^{2\rho}\eta(s)^2\,\frac{\mathrm d s}{s}
 \leq
 \mathfrak D(2\rho)^2.
\]
Consequently,
\[
 \eta(\rho)
 \leq
 \frac{\mathfrak D(2\rho)}{\sqrt{\log 2}}
 \longrightarrow0
 \qquad\text{as }\rho\downarrow0.
\]

Since \(\eta(\rho)\to0\) as \(\rho\downarrow0\), we have
\begin{equation}\label{eq:uniform-essential-trace}
 \lim_{\rho\downarrow0}
 \operatorname*{ess\,sup}_{0<\lambda\leq\rho}
 \bigl\|
 A(\lambda,\cdot,\cdot)-A^0
 \bigr\|_{\mathrm L^\infty(\R^n\times\R)}
 =0.
\end{equation}
Equivalently, for every \(\varepsilon>0\), there is
\(\rho_\varepsilon>0\) such that
\[
 \bigl\|
 A(\lambda,\cdot,\cdot)-A^0
 \bigr\|_{\mathrm L^\infty(\R^n\times\R)}
 <\varepsilon
\]
for almost every \(0<\lambda<\rho_\varepsilon\).  We define the
coefficient on the boundary by \(A(0,x,t)=A^0(x,t)\).  In this essential
one-sided sense, \(A^0\) is the uniform boundary trace of \(A\).

\subsection{Statement of the main results}\label{subsec:main-results}

The following is the main result of the paper.

\begin{thm}\label{thm:global}
Let \(A\) be real, bounded, measurable, symmetric, and uniformly elliptic,
and assume the full spatial periodicity \eqref{eq:full-periodicity}.
Suppose that there is a bounded, symmetric, uniformly elliptic boundary
trace \(A^0=A^0(x,t)\) satisfying the temporal-continuity hypotheses
\eqref{eq:boundary-time-modulus}-\eqref{eq:boundary-time-continuity} and
that \(A\) converges to \(A^0\) according to the transverse square-Dini
hypothesis \eqref{eq:normal-square-Dini}.  Then there are \(M>4\) and
\(C<\infty\) such that, for every
\((x_0,t_0)\in\R^n\times\R\), every \(r>0\), and every
\(M\)-admissible pole \(P\) for \(\Delta_r=\Delta_r(x_0,t_0)\), one has
\(\omega_{\cH}^P\ll\mathrm d x\,\mathrm d t\), and the forward
parabolic Poisson kernel satisfies
\[
 \left(
 \frac{1}{|\Delta_r|}
 \iint_{\Delta_r}(k^P)^2\,\mathrm d x\,\mathrm d t
 \right)^{1/2}
 \leq
 \frac{C}{|\Delta_r|}
 \iint_{\Delta_r}k^P\,\mathrm d x\,\mathrm d t
\]
uniformly in the center, \(r\), and \(P\).  The constants depend only on
\(n,\mu,\Lambda\), the moduli \(\varpi\) and \(\mathfrak D\), and the fixed
geometric parameters.
\end{thm}

\begin{cor}
\label{cor:global-D2}
Under the hypotheses of Theorem~\ref{thm:global}, parabolic measure
belongs to \(B_2(\mathrm d x\,\mathrm d t)\), and the Dirichlet problem
\(D_2\) is uniquely solvable in \(\R^{n+2}_+\).
More precisely, the solution is unique among weak solutions \(u\) for
which \(N_\ast u\in\mathrm L^2(\R^n\times\R)\).
\end{cor}

The proof of Theorem~\ref{thm:global} is separated into two results.  The first is a large-scale
reduction which requires no symmetry or quantitative coefficient regularity.

\begin{thm}
\label{thm:large-scale}
Assume \eqref{eq:ellipticity} and \eqref{eq:full-periodicity}.  If the local
reverse H\"older estimate \eqref{eq:local-rh} holds uniformly for
\(0<r\leq1\), then it holds uniformly for every \(r>0\).  The global
constant depends only on \(n,\mu,\Lambda,M,C_0\) and the fixed geometric
parameters.  Since all spatial periods have been normalized
to one, no additional period parameter occurs.
\end{thm}

The second result supplies the estimate below the period scale.  It uses
symmetry and the regularity assumptions in
Subsection~\ref{subsec:small-assumptions}, but no periodicity.

\begin{thm}
\label{thm:parabolic-KS}
Assume \eqref{eq:ellipticity}, \(A=A^{\mathsf T}\), and the boundary-trace,
temporal-continuity, and transverse square-Dini hypotheses of
Subsection~\ref{subsec:small-assumptions}.  Then there are \(M>4\) and
\(C<\infty\) such that, for every
\((x_0,t_0)\in\R^n\times\R\), every \(0<r\leq1\), and every
\(M\)-admissible pole \(P\) for \(\Delta_r=\Delta_r(x_0,t_0)\), one has
\(\omega_{\cH}^P\ll\mathrm d x\,\mathrm d t\), and
\begin{equation}\label{eq:KS-small-rh}
 \left(
 \frac{1}{|\Delta_r|}
 \iint_{\Delta_r}(k^P)^2\,\mathrm d x\,\mathrm d t
 \right)^{1/2}
 \leq
 \frac{C}{|\Delta_r|}
 \iint_{\Delta_r}k^P\,\mathrm d x\,\mathrm d t.
\end{equation}
The constant depends only on \(n,\mu,\Lambda\), the moduli \(\varpi\) and
\(\mathfrak D\), and the fixed geometric parameters.  It is independent of
the center and radius of \(\Delta_r\).
\end{thm}

\subsection{Discussion of the proofs}\label{subsec:outline} We describe the main points of the argument.  The proof decomposes into the large-scale estimate and the small-scale estimate.

{\it The large-scale estimate.} Periodicity in the transverse direction implies that
\(u(\lambda+1,x,t)-u(\lambda,x,t)\) is again a solution and gives one
additional inverse power of the scale in the interior.  This observation suggests a direct parabolic analogue of Dahlberg's
elliptic comparison scheme, but the forward-adjoint mismatch described
in Remark~\ref{rem:normal-periodicity-obstruction} prevents that argument
from closing.  Under full spatial periodicity, we overcome this
obstruction by solving a nonautonomous adjoint cell problem on the
spatial torus and using the resulting height coordinate to close the
boundary summation.  The resulting function
\(\Phi^*\) vanishes on \(\{\lambda=0\}\), is positive in the half-space,
and satisfies \(\Phi^*=\lambda+O(1)\).   Comparing the Green function (in the adjoint variables) with \(\Phi^*\) yields \(v_2(1,y,s)=G_{\cH}^{\Omega_2}(P_2;(1,y,s))
\lesssim r^{-(n+2)}\) throughout the comparison region.  This improves the general bound
\(v_2(1,y,s)\lesssim r^{-(n+1)}\) by the required factor \(r^{-1}\),   and this
bound closes the parabolic version of the discrete Dahlberg sum and propagates the unit-scale
reverse H\"older estimate to every \(r>1\).

{\it The small-scale estimate.} At a boundary box of radius \(r\), the
 definition of \(\varpi\) shows that the boundary trace differs by
at most \(\varpi(2Lr)\) from the time-independent matrix
\(A^0(x,t_0)\).  A time cutoff produces a globally defined, \(\lambda\)-independent
coefficient \(B_r\) whose \(\mathrm L^\infty\)-distance from the frozen
matrix \(A^0(x,t_0)\) is at most \(\varpi(2Lr)\), and hence is small for
all sufficiently small \(r\).  The
\(\mathrm L^2\) regularity theory for the frozen matrix, the small
\(\mathrm L^\infty\)-perturbation theorem of~\cite{AEN}, applied after
time reversal, and regularity-Dirichlet duality give uniform
\(\mathrm L^2\) Dirichlet solvability and hence, by
Theorem~\ref{thm:Bp-Dq}, a uniform \(B_2\)-estimate for the parabolic
measure associated with \(B_r\).
A transverse cutoff then produces a coefficient \(\widetilde A_r\)
agreeing with \(A\) in the relevant box.  Its discrepancy from \(B_r\)
has Carleson norm bounded by \(\mathfrak D(Cr)\), where \(Cr\leq1\).  The small-Carleson
perturbation theorem, Proposition~\ref{prop:parabolic-Carleson-perturbation}\textup{(i)}, preserves \(B_2\), and
a local change-of-operator argument yields \(\mathrm{RH}_2\) for all
sufficiently small \(r\).  A finite covering gives the remaining range
\(r\leq1\).

Finally, Appendix~\ref{sec:smooth-coefficient-reduction} records a
boundary-compatible smooth-approximation argument.  It shows that, in
proving the main theorem and constructing the complete corrector, the
finite-slab solutions, and the adjoint height coordinate, one may first
assume that \(A\) is smooth.  The approximating matrices preserve the
ellipticity, spatial periodicity, symmetry, temporal modulus, and
square-Dini control uniformly.  Smoothness is used only qualitatively: all estimates must be independent of derivatives of the approximating
coefficients.  Weak energy compactness and convergence of parabolic
measures then remove the smoothness assumption.
\subsection{Organization of the paper}

Section~\ref{sec:technical-tools} records the potential-theoretic tools
used throughout.  Section~\ref{sec:normal-comparison} establishes the
transverse-difference and two-comparison estimates.  The obstruction to
a direct argument under transverse periodicity alone is explained in
Remark~\ref{rem:normal-periodicity-obstruction}.
Section~\ref{sec:height-coordinate} constructs the adjoint height
coordinate and proves the resulting Green-function summation estimate.
Sections~\ref{sec:large-scale-proof} and~\ref{sec:parabolic-KS} prove the
large- and small-scale theorems, respectively.  Section~\ref{sec:completion}
establishes a nonsymmetric \(A_\infty\) extension, describes the
transfer to time-independent Lipschitz graph domains, and formulates the
abstract height-coordinate condition underlying the large-scale
argument.  For nonsymmetric coefficients, the argument appears to yield
a new result even in the elliptic setting; see
Remark~\ref{rem:ellipticresult}.  The final subsection records
consequences for scale-uniform boundary estimates for the families
\(A(X/\varepsilon,t)\) and, when \(A\) is also time-periodic,
\(A(X/\varepsilon,t/\varepsilon^\alpha)\).

Appendix~\ref{sec:smooth-coefficient-reduction} proves the
smooth-approximation principle used to remove auxiliary coefficient
regularity from the corrector, slab, and height-coordinate constructions,
as well as from the proof of the main theorem.
Appendix~\ref{Homoderv} gives a detailed derivation of the effective
operator for the spatially oscillating family \(A(X/\varepsilon,t)\).
In Section~\ref{sec:completion}, the space-time periodic families
\(A(X/\varepsilon,t/\varepsilon^\alpha)\) are discussed only at the
level of their cell problems and effective matrices, without detailed
proofs. Appendix~\ref{app:comparison-geometry} illustrates
the geometry of the two-comparison construction used in
Section~\ref{sec:normal-comparison}.

\section{Technical tools}\label{sec:technical-tools}

This section records the potential-theoretic results used below and fixes
their causal orientation.  Unless otherwise stated, the coefficients are
real, bounded, measurable, and uniformly elliptic.  The constants depend
only on \(n,\mu,\Lambda\) and the geometric parameters displayed in the
statements.  We refer to
\cite{Aronson,FabesGarofaloSalsa,FabesSafonov,FabesSafonovYuan,Nystrom}
for the underlying theory.  We also record the localization and
Carleson-perturbation statements used in the small-scale argument.

\medskip
\noindent
{\bf Convention on constants.} Unless otherwise stated,
\(c\) and \(C\) denote positive constants which may change from line to
line and depend only on the dimension, ellipticity, and the quantitative
constants explicitly appearing in the hypotheses.  We write \(A\lesssim
B\) if \(A\leq CB\), and \(A\simeq B\) if both \(A\lesssim B\) and
\(B\lesssim A\).

\subsection{Parabolic geometry and causal orientation}
\label{subsec:parabolic-geometry-tools}

Write \(X=(\lambda,x)\in\R^{n+1}_+\) and \(Z=(X,t)\).  In addition to
the boundary cubes in \eqref{eq:boundary-cubes}, define
\begin{equation}\label{eq:standard-corkscrews}
 A_r^\pm(x_0,t_0)=(4r,x_0,t_0\pm16r^2),
 \qquad
 A_r(x_0,t_0)=(4r,x_0,t_0).
\end{equation}
For \(\kappa\geq1\), let
\begin{align}
 \mathcal P_{\kappa,r}^+(x_0,t_0)
 &:=\left\{(\lambda,x,t)\in\R^{n+2}_+:
  |x-x_0|+\lambda\leq\kappa(t-t_0)^{1/2},\quad
  t-t_0\geq64r^2\right\},
 \label{eq:forward-pole-region}\\
 \mathcal P_{\kappa,r}^-(x_0,t_0)
 &:=\left\{(\lambda,x,t)\in\R^{n+2}_+:
  |x-x_0|+\lambda\leq\kappa(t_0-t)^{1/2},\quad
  t_0-t\geq64r^2\right\}.
 \label{eq:adjoint-pole-region}
\end{align}
The first is a forward pole-comparison region with fixed aperture and causal
separation.  The second is its adjoint counterpart.  The numerical
constants in \eqref{eq:standard-corkscrews}-\eqref{eq:adjoint-pole-region}
are fixed normalizations.  Replacing them by other fixed constants changes
only the comparison constants.

\begin{lem}\label{lem:universal-subcube-pole}
Let \(\Delta_0=\Delta_{r_0}(x_0,t_0)\), fix \(M>4\), and set
\(\theta=(c_a+C_a)/2\), where \(c_a,C_a\) are the constants in Definition~\ref{defn:admissible-pole}.  If \(K\geq M\) is sufficiently large depending
only on \(c_a,C_a,M\), then
\[
 P_{\Delta_0}^{\mathrm{far}}
 =\bigl(Kr_0,x_0,t_0+\theta K^2r_0^2\bigr)
\]
is \(M\)-admissible for every boundary cube
\(\Delta_\rho(y,s)\subset\Delta_0\).  Moreover,
\(P_{\Delta_0}^{\mathrm{far}}\) and \(A_{4r_0}^+(x_0,t_0)\) lie in one
common forward pole region
\(\mathcal P_{\kappa_0,r_0}^+(x_0,t_0)\), where \(\kappa_0\) is fixed.
\end{lem}

\begin{proof}
For such a subcube,
\[
 \rho\leq r_0,
 \qquad |y-x_0|\leq r_0,
 \qquad |s-t_0|\leq r_0^2.
\]
Thus the height and spatial requirements in
\eqref{eq:admissible-pole} follow once \(K\geq M\) and
\(C_aK\geq1\).  Moreover,
\[
 \theta K^2r_0^2-r_0^2
 \leq t(P_{\Delta_0}^{\mathrm{far}})-s
 \leq \theta K^2r_0^2+r_0^2.
\]
Taking \(K\) so that
\(K^{-2}\leq(C_a-c_a)/2\) proves the two temporal inequalities.  We may
also require \(\theta K^2\geq64\).  Then
\[
 \lambda(P_{\Delta_0}^{\mathrm{far}})
 =\theta^{-1/2}
 \bigl(t(P_{\Delta_0}^{\mathrm{far}})-t_0\bigr)^{1/2},
\]
so the asserted common pole region follows from
\eqref{eq:forward-pole-region}, after increasing \(\kappa_0\) by a fixed
amount to include \(A_{4r_0}^+(x_0,t_0)\).
\end{proof}

If \(D\subset\R^{n+1}\) and \(a<b\), the forward and adjoint parabolic
boundaries of \(D\times(a,b)\) are, respectively,
\begin{align}
 \partial_p(D\times(a,b))
 &=\bigl(\overline D\times\{a\}\bigr)
   \cup\bigl(\partial D\times(a,b)\bigr),
 \label{eq:forward-parabolic-boundary}\\
 \partial_p^*(D\times(a,b))
 &=\bigl(\overline D\times\{b\}\bigr)
   \cup\bigl(\partial D\times(a,b)\bigr).
 \label{eq:adjoint-parabolic-boundary}
\end{align}
Thus a forward problem takes data on the initial time face and the lateral
boundary.  An adjoint problem takes data on the terminal time face and the
lateral boundary.  Causality further restricts the forward representing
measure to earlier boundary times and the adjoint representing measure to
later boundary times.

\subsection{Local estimates and Harnack inequalities}
\label{subsec:local-parabolic-tools}

For \(Z_0=(X_0,t_0)\) and \(r>0\), set
\[
 \mathcal Q_r^-(Z_0)=B(X_0,r)\times(t_0-r^2,t_0),
 \qquad
 \mathcal Q_r^+(Z_0)=B(X_0,r)\times(t_0,t_0+r^2).
\]

\begin{lem}
\label{lem:local-parabolic-estimates}
Let \(u\) be a weak solution of \(\cH u=0\) in
\(\mathcal Q_r^-(Z_0)\).  Then
\begin{align}
 \iiint_{\mathcal Q_{r/2}^-(Z_0)}
 |\nabla_Xu|^2\,\mathrm d \lambda\,\mathrm d x\,\mathrm d t
 &\leq
 \frac{C}{r^2}
 \iiint_{\mathcal Q_r^-(Z_0)}
 |u|^2\,\mathrm d \lambda\,\mathrm d x\,\mathrm d t,
 \label{eq:caccioppoli-standard}\\
 \|u\|_{\mathrm L^\infty(\mathcal Q_{r/2}^-(Z_0))}^2
 &\leq
 \frac{C}{r^{n+3}}
 \iiint_{\mathcal Q_r^-(Z_0)}
 |u|^2\,\mathrm d \lambda\,\mathrm d x\,\mathrm d t.
 \label{eq:local-boundedness-standard}
\end{align}
Solutions also satisfy the scale-invariant interior H\"older estimate.  The
same estimates hold in flat half-cylinders when \(u\) has zero trace on the
enlarged flat boundary portion.  More precisely, there are
\(\beta\in(0,1)\) and \(C<\infty\), depending only on
\(n,\mu,\Lambda\), such that, if
\[
 \mathcal Q_{r,+}^-(x_0,t_0)
 =
 \bigl(B((0,x_0),r)\cap\{\lambda>0\}\bigr)
 \times(t_0-r^2,t_0),
\]
and \(u\) is a solution there with zero Sobolev trace on
\(\{\lambda=0\}\cap\partial\mathcal Q_{r,+}^-(x_0,t_0)\), then
\[
 |u(\lambda,x,t)|
 \leq C\left(\frac{\lambda}{r}\right)^\beta
 \left(
 \fint_{\mathcal Q_{r,+}^-(x_0,t_0)}|u|^2
 \,\mathrm d\lambda\,\mathrm d x\,\mathrm d t
 \right)^{1/2}
\]
in \(\mathcal Q_{r/2,+}^-(x_0,t_0)\).  For
\(\cH^*u=0\), replace past cylinders by future cylinders.
This boundary decay follows by iteration of
\cite[Corollary~4.2]{SafonovYuan}, together with local boundedness, in
the time-dependent divergence-form setting.
\end{lem}

\begin{lem}
\label{lem:parabolic-harnack}
Let \(u\geq0\) solve \(\cH u=0\).  Suppose that two interior
subcylinders \(\mathcal Q^-\) and \(\mathcal Q^+\), whose radii are
comparable to \(r\), are joined by a forward Harnack chain contained in
the domain.  If \(\mathcal Q^-\) occurs before \(\mathcal Q^+\), then
\begin{equation}\label{eq:forward-harnack}
 \sup_{\mathcal Q^-}u
 \leq C\inf_{\mathcal Q^+}u.
\end{equation}
If \(u\geq0\) solves \(\cH^*u=0\), the causal orientation is reversed and
\begin{equation}\label{eq:adjoint-harnack}
 \sup_{\mathcal Q^+}u
 \leq C\inf_{\mathcal Q^-}u.
\end{equation}
The constant \(C\) depends only on \(n,\mu,\Lambda\) and the quantitative
geometry and length of the Harnack chain.
\end{lem}

\begin{rem}
Lemma~\ref{lem:parabolic-harnack} contains only the ordinary
time-oriented Harnack inequality.  No structural reverse estimate
between \(u(A_r^+)\) and \(u(A_r^-)\) is asserted for an arbitrary
nonnegative solution.
\end{rem}
\begin{rem}
The preceding estimates are invariant under translation and parabolic
dilation.  Their boundary versions are used only on lateral spatial
boundaries, never across an initial or terminal time face.
\end{rem}
\subsection{Parabolic measure and Green functions}
\label{subsec:green-package}

\begin{lem}
\label{lem:parabolic-measure-representation}
For each \(P=(\lambda,x,t)\in\R^{n+2}_+\), there is a positive regular
Borel measure \(\omega_{\cH}^P\) on \(\R^n\times\R\).  If
\(f\in C_0(\R^n\times\R)\) and \(u\) is the bounded continuous solution
with boundary data \(f\), then
\begin{equation}\label{eq:parabolic-measure-representation}
 u(P)=
 \iint_{\R^n\times\R}
 f(y,s)\,\mathrm d \omega_{\cH}^P(y,s).
\end{equation}
Moreover,
\begin{equation}\label{eq:measure-causal-support}
 \omega_{\cH}^P\bigl(\mathbb R^n\times[t,\infty)\bigr)=0,
 \qquad
 \operatorname{supp}\omega_{\cH}^P
 \subset\{(y,s):s\leq t\}.
\end{equation}
The adjoint measure \(\omega_{\cH^*}^P\) satisfies the analogous
representation and
\[
 \omega_{\cH^*}^P\bigl(\mathbb R^n\times(-\infty,t]\bigr)=0,
 \qquad
 \operatorname{supp}\omega_{\cH^*}^P
 \subset\{(y,s):s\geq t\}.
\]
\end{lem}

\begin{rem}
In the half-space, the representing measures are probability measures:
\begin{equation}\label{eq:parabolic-measure-total-mass}
 \omega_{\cH}^P(\mathbb R^n\times\mathbb R)=1,
 \qquad
 \omega_{\cH^*}^P(\mathbb R^n\times\mathbb R)=1.
\end{equation}
Indeed, the Riesz representation theorem and the maximum principle
initially give only that the total masses are at most one.  To obtain equality, let
\(z_P=(x_P,t_P)\) be the boundary projection of
\(P=(\lambda_P,x_P,t_P)\), and choose an increasing family
\(f_R\in C_0^\infty(\mathbb R^n\times\mathbb R)\) such that
\[
 0\leq f_R\leq1,\qquad
 f_R=1\quad\text{on }\Delta_R(z_P),\qquad
 \operatorname{supp}f_R\subset\Delta_{2R}(z_P).
\]
Then \(f_R(y,s)\uparrow1\) for every
\((y,s)\in\mathbb R^n\times\mathbb R\) as \(R\to\infty\).  If \(u_R\)
denotes the corresponding solution, the boundary H\"older
estimate applied to \(1-u_R\) gives
\[
 0\leq1-u_R(P)
 \leq C\left(\frac{\lambda_P}{R}\right)^\beta.
\]
Hence \(u_R(P)\to1\), and monotone convergence gives
\eqref{eq:parabolic-measure-total-mass}.  The adjoint conclusion follows
by reversing the time orientation.  In a bounded time-oriented cylinder,
the corresponding measures also
have total mass one and are carried by the appropriate parabolic
boundaries in \eqref{eq:forward-parabolic-boundary} and
\eqref{eq:adjoint-parabolic-boundary}.  In particular, their restrictions
to any selected portion of those boundaries have mass at most one.
\end{rem}

\begin{lem}
\label{lem:green-reciprocity}
Let \(\Omega\) be the half-space or a time-oriented cylinder.  Then the
forward and adjoint Dirichlet Green functions in \(\Omega\) exist.  We
use the convention that \((Y,s)\) is the pole of
\(G_{\cH}^{\Omega}(X,t;Y,s)\).  Away from the diagonal, causality gives
\begin{align}
 G_{\cH}^{\Omega}(X,t;Y,s)&=0
 \qquad\text{when }t\leq s,
 \label{eq:green-forward-causality}\\
 G_{\cH^*}^{\Omega}(Y,s;X,t)&=0
 \qquad\text{when }s\geq t.
 \label{eq:green-adjoint-causality}
\end{align}
Moreover, the reciprocity identity
\begin{equation}\label{eq:green-reciprocity}
 G_{\cH}^{\Omega}(X,t;Y,s)
 =
 G_{\cH^*}^{\Omega}(Y,s;X,t)
\end{equation}
holds.  Both kernels in \eqref{eq:green-reciprocity} are taken in
\(\Omega\), with the mutually adjoint Dirichlet conditions on
\(\partial_p\Omega\) and \(\partial_p^*\Omega\).
\end{lem}

\begin{rem}\label{rem:reciprocity-time-dependent}
The reciprocity identity \eqref{eq:green-reciprocity} is initially
understood as an almost-everywhere equality.  After choosing the locally
H\"older continuous representatives of the Green functions, it holds
pointwise away from the diagonal.  For fixed \((Y,s)\), the function
\[
 (X,t)\longmapsto G_{\cH}^{\Omega}(X,t;Y,s)
\]
is a forward solution away from \((Y,s)\).  For fixed \((X,t)\), the
function
\[
 (Y,s)\longmapsto G_{\cH}^{\Omega}(X,t;Y,s)
\]
is an adjoint solution away from \((X,t)\).  The word \emph{pole} is also
used below for the point at which parabolic
measure is evaluated.  Thus, in \(G_{\cH}^{\Omega}(P;Z)\), the point
\(P\) is the parabolic-measure pole.  When this Green function is regarded
as a function of \(Z\), reciprocity shows that \(P\) is its singularity
for the adjoint equation.  Reciprocity remains valid for time-dependent
coefficients and does not
require \(A\) to be symmetric.  Even when \(A=A^{\mathsf T}\), the
operators \(\cH\) and \(\cH^*\) differ because adjunction reverses the
sign of the time derivative.  Consequently, one must use the
forward-adjoint identity \eqref{eq:green-reciprocity}, rather than a
symmetry identity for the same Green function with its two space-time
arguments interchanged.
\end{rem}

\begin{lem}
\label{lem:green-upper}
After extending \(A\) to a real uniformly elliptic matrix on
\(\R^{n+1}\times\R\), the whole-space fundamental solution vanishes when
\(t\leq s\).  When \(t>s\), it satisfies
\begin{equation}\label{eq:aronson-upper}
 0\leq\Gamma_{\cH}(X,t;Y,s)
 \leq
 \frac{C}{(t-s)^{(n+1)/2}}
 \exp\!\left(-\frac{|X-Y|^2}{C(t-s)}\right).
\end{equation}
The Dirichlet Green function in any domain under consideration satisfies
\begin{equation}\label{eq:green-dominated}
 0\leq G_{\cH}^{\Omega}(X,t;Y,s)
 \leq\Gamma_{\cH}(X,t;Y,s).
\end{equation}
Consequently, whenever \((X,t)\neq(Y,s)\),
\begin{equation}\label{eq:green-parabolic-upper}
 G_{\cH}^{\Omega}(X,t;Y,s)
 \leq
 \frac{C\mathbf 1_{\{t>s\}}}
 {\bigl(|X-Y|+|t-s|^{1/2}\bigr)^{n+1}}.
\end{equation}
The adjoint estimates hold with the time orientation reversed.
\end{lem}

\subsection{Boundary comparisons  and doubling}
\label{subsec:boundary-comparison-tools}

If \(\Omega=D\times(a,b)\) is a time-oriented cylinder, we denote by
\(\omega_{\cH,\Omega}^P\) the forward representing measure on
\(\partial_p\Omega\), and by \(\omega_{\cH^*,\Omega}^P\) the adjoint
representing measure on \(\partial_p^*\Omega\), both evaluated at \(P\).

The comparison results below are local.  In an ambient time-oriented cylinder,
each fixed enlargement below is required to lie in the cylinder and to
remain quantitatively separated from the inappropriate time face.  If a
comparison function is a Green function, its singularity must lie outside
that enlargement and have the stated causal separation.

\begin{lem}
\label{lem:boundary-comparison}
There exist structural constants \(c_0,c_1\geq1\) with the following
property.  Let \(u\) and \(v\) be positive solutions of
\(\cH u=\cH v=0\) in \(T_{2r}(x_0,t_0)\).  Assume that they are
continuous up to the flat boundary and vanish there on
\(2\Delta_r(x_0,t_0)\).  Put \(\rho_0=r/c_0\),
\(\rho_1=\rho_0/c_1\), and define the reference balance by
\begin{equation}\label{eq:forward-balance-parameters}
 \mathcal B_{\rho_0}(u,v)
 =\max\left\{
 \frac{u(A_{\rho_0}^+(x_0,t_0))}
      {u(A_{\rho_0}^-(x_0,t_0))},
 \frac{v(A_{\rho_0}^+(x_0,t_0))}
      {v(A_{\rho_0}^-(x_0,t_0))}
 \right\}.
\end{equation}
If \(\widetilde z\in\Delta_{\rho_1}(x_0,t_0)\),
\(0<\rho<\rho_1\), and \(Z\in T_{\rho/c_1}(\widetilde z)\), then
\begin{equation}\label{eq:normalized-boundary-comparison}
 C^{-1}
 \frac{u(A_\rho(\widetilde z))}
      {v(A_\rho(\widetilde z))}
 \leq
 \frac{u(Z)}{v(Z)}
 \leq
 C
 \frac{u(A_\rho(\widetilde z))}
      {v(A_\rho(\widetilde z))}.
\end{equation}
Here
\begin{equation}\label{eq:boundary-comparison-dependence}
 C=C\bigl(n,\mu,\Lambda,\mathcal B_{\rho_0}(u,v)\bigr).
\end{equation}
\end{lem}

\begin{rem}
For positive adjoint solutions, Lemma~\ref{lem:boundary-comparison} holds
with the time directions reversed and with
\begin{equation}\label{eq:adjoint-balance-parameters}
 \mathcal B_{\rho_0}^*(u,v)
 =\max\left\{
 \frac{u(A_{\rho_0}^-(x_0,t_0))}
      {u(A_{\rho_0}^+(x_0,t_0))},
 \frac{v(A_{\rho_0}^-(x_0,t_0))}
      {v(A_{\rho_0}^+(x_0,t_0))}
 \right\}.
\end{equation}
The dependence on the reference balance is essential and agrees with the
normalized form of boundary comparison in
\cite[Theorem~4.2]{LitsgardNystrom}.  That theorem is cited here to clarify
the dependence of the constant.  In the time-dependent divergence-form
setting, the local quotient comparison used here follows from
\cite[Theorem~2.4]{SafonovYuan}, after translation, parabolic scaling,
time reversal for the adjoint statement, and localization inside the
ambient cylinder.  The underlying backward Harnack estimates are recorded
in \cite[Theorem~2.3]{SafonovYuan}; see also
\cite{FabesSafonov,FabesSafonovYuan,Nystrom}.
\end{rem}

\begin{rem}\label{rem:domain-convention}
Whenever a result below is stated in a time-oriented cylinder at scale
\(r\) above \(\Delta_r(x_0,t_0)\), we use the following convention.
The domain \(\Omega\) is either \(\R^{n+2}_+\) or
\(\Omega=D\times(a,b)\), where \(D\subset\R^{n+1}_+\) is a bounded,
connected Lipschitz domain.  The domains \(D\) belong to a fixed
quantitative class of Lipschitz uniform domains: their Lipschitz
characters, interior corkscrew constants, and Harnack-chain constants are
uniformly controlled.  No upper bound for \(\operatorname{diam}(D)/r\)
is imposed.  There is a fixed constant \(K_\Omega\geq64\), chosen sufficiently large
relative to all fixed dilation and aperture constants occurring in the
result under consideration, such that
\[
 D\cap
 \bigl((-K_\Omega r,K_\Omega r)
       \times B(x_0,K_\Omega r)\bigr)
 =
 (0,K_\Omega r)\times B(x_0,K_\Omega r).
\]
We require the fixed enlargement of every comparison region used below
to be contained in the corresponding concentric half-box with
\(K_\Omega/2\) in place of \(K_\Omega\).  Consequently, the relevant
part of \(\partial D\) is the flat boundary \(\{\lambda=0\}\), and
these comparison regions are separated by at least a fixed multiple of
\(r\) from the remaining spatial boundary.
We also assume that
\[
 a<t_0-K_\Omega^2r^2<t_0+K_\Omega^2r^2<b,
\]
with
\[
 (t_0-K_\Omega^2r^2)-a\gtrsim r^2,
 \qquad
 b-(t_0+K_\Omega^2r^2)\gtrsim r^2.
\]
Every pole \(P=(X(P),t(P))\) occurring in a cylinder statement is
assumed, in addition to the displayed aperture condition, to satisfy
\[
 \operatorname{dist}\!\left(
 X(P),
 \partial D\setminus
 \bigl(\{0\}\times B(x_0,K_\Omega r)\bigr)
 \right)
 \geq c_\Omega r
\]
and
\[
 t(P)-a\geq c_\Omega r^2,
 \qquad
 b-t(P)\geq c_\Omega r^2,
\]
where \(c_\Omega>0\) is fixed.  Whenever a Harnack comparison between a pole and a scale-\(r\)
reference region is used, the required time-oriented Harnack chain has
uniformly bounded length, and fixed enlargements of all its members
remain in \(\Omega\) and have spatial and temporal clearance comparable
to \(r\) and \(r^2\), respectively.  All constants implicit in this convention are fixed independently of
\(r\), \(x_0\), and \(t_0\).  They are included in what is called the
fixed cylinder geometry.
\end{rem}

\begin{lem}
\label{lem:green-boundary-time-comparison}
Let \(\Delta_r=\Delta_r(x_0,t_0)\), and let \(\Omega\) be either the
half-space or a time-oriented cylinder satisfying the convention in
Remark~\ref{rem:domain-convention} relative to \(\Delta_r\).  If
\(P\in\Omega\cap\mathcal P_{\kappa,r}^+(x_0,t_0)\), then
\begin{equation}\label{eq:green-boundary-time-comparison}
 G_{\cH}^{\Omega}(P;A_r^-(x_0,t_0))
 \simeq
 G_{\cH}^{\Omega}(P;A_r^+(x_0,t_0)).
\end{equation}
The comparison constant depends only on \(n,\mu,\Lambda,\kappa\) and
the fixed cylinder geometry.  If
\(P\in\Omega\cap\mathcal P_{\kappa,r}^-(x_0,t_0)\), the
corresponding estimate for \(G_{\cH^*}^{\Omega}\) holds with the time
orientations reversed.
\end{lem}

\begin{rem}
The constant in \eqref{eq:green-boundary-time-comparison} depends only on
\(n,\mu,\Lambda\), the fixed comparison geometry, and the quantitative
aperture and causal separation of \(P\).  In a truncated cylinder it also
depends on the clearance from the artificial time faces.  This is the
boundary backward Harnack estimate specialized to a pole-separated Green
function.  See
\cite{FabesGarofaloSalsa,FabesSafonov,FabesSafonovYuan}.  When both
functions in Lemma~\ref{lem:boundary-comparison} are
pole-separated Green functions, Lemma~\ref{lem:green-boundary-time-comparison}
bounds the balance in
\eqref{eq:forward-balance-parameters} or
\eqref{eq:adjoint-balance-parameters}.  The constant in
\eqref{eq:normalized-boundary-comparison} is then structural.  At the
large scales where the adjoint height coordinate \(\Phi^*\) is used, the
same conclusion holds when one comparison function is \(\Phi^*\).  Indeed, as we will see, after
increasing the large-scale threshold so that \(\rho_0\geq1\),
\[
 \Phi^*(A_{\rho_0}^\pm)\simeq\rho_0
\]
by Lemma~\ref{lem:height-coordinate} below.
\end{rem}

\begin{lem}
\label{lem:cfms}
Let \(\Delta_r=\Delta_r(x_0,t_0)\), and let \(\Omega\) be either the
half-space or a time-oriented cylinder satisfying the convention in
Remark~\ref{rem:domain-convention} relative to \(\Delta_r\).   If
\(P\in\Omega\cap\mathcal P_{\kappa,r}^+(x_0,t_0)\), then
\begin{equation}\label{eq:cfms-forward-raw}
 C^{-1}r^{n+1}G_{\cH}^{\Omega}(P;A_r^+(x_0,t_0))
 \leq\omega_{\cH,\Omega}^P(\Delta_{r/2}(x_0,t_0))
 \leq Cr^{n+1}G_{\cH}^{\Omega}(P;A_r^-(x_0,t_0)).
\end{equation}
If \(P\in\Omega\cap\mathcal P_{\kappa,r}^-(x_0,t_0)\), then
\begin{equation}\label{eq:cfms-adjoint-raw}
 C^{-1}r^{n+1}G_{\cH^*}^{\Omega}(P;A_r^-(x_0,t_0))
 \leq\omega_{\cH^*,\Omega}^P(\Delta_{r/2}(x_0,t_0))
 \leq Cr^{n+1}G_{\cH^*}^{\Omega}(P;A_r^+(x_0,t_0)).
\end{equation}
\end{lem}

\begin{rem}
By reciprocity,
\[
 G_{\cH^*}^{\Omega}(P;A_r^\pm)
 =G_{\cH}^{\Omega}(A_r^\pm;P).
\]
Thus \eqref{eq:cfms-adjoint-raw} may be written entirely in terms of the
forward Green function.  Different fixed corkscrew normalizations give
equivalent estimates by
Lemma~\ref{lem:green-boundary-time-comparison} and ordinary Harnack
chains.  The constants depend only on \(n,\mu,\Lambda,\kappa\) and the
fixed comparison geometry.  We omit \(\Omega\) from the notation in the
half-space.
\end{rem}

\begin{lem}
\label{lem:parabolic-measure-doubling}
In the half-space, let \(\Delta_0=\Delta_{r_0}(x_0,t_0)\) and
\(P_0=A_{4r_0}^+(x_0,t_0)\).  Then
\begin{equation}\label{eq:parabolic-measure-doubling}
 \omega_{\cH}^{P_0}(2\Delta)
 \leq C\omega_{\cH}^{P_0}(\Delta)
\end{equation}
whenever \(\Delta\) is a boundary parabolic cube contained in
\(4\Delta_0\).  The adjoint statement holds with a
past-oriented pole.
\end{lem}

\begin{lem}
\label{lem:parabolic-measure-nondegeneracy}
There is a structural constant \(c_{\mathrm{nd}}>0\) such that
\begin{equation}\label{eq:parabolic-measure-nondegeneracy}
 \omega_{\cH}^{A_{4r}^+(x_0,t_0)}
       (\Delta_r(x_0,t_0))\geq c_{\mathrm{nd}}
\end{equation}
for every \(r>0\).  The adjoint statement holds with a past-oriented
corkscrew.
\end{lem}

\begin{rem}
Doubling and nondegeneracy are separate conclusions.  The
nondegeneracy estimate in
Lemma~\ref{lem:parabolic-measure-nondegeneracy} follows directly from
boundary H\"older continuity and the forward Harnack inequality.  Indeed,
choose \(f\in C_0(\Delta_r(x_0,t_0))\) such that
\[
 0\leq f\leq1,
 \qquad
 f=1\quad\text{on }\Delta_{r/2}(x_0,t_0),
\]
and let \(u\) be the solution with boundary data \(f\).  The function
\(1-u\) is nonnegative and vanishes on
\(\Delta_{r/2}(x_0,t_0)\).  Boundary H\"older continuity therefore gives
\[
 1-u(A_{\varepsilon r}^+(x_0,t_0))
 \leq C\varepsilon^\beta
\]
for a fixed sufficiently small \(\varepsilon>0\).  Hence
\(u(A_{\varepsilon r}^+)\geq1/2\).  A forward Harnack chain from
\(A_{\varepsilon r}^+\) to \(A_{4r}^+\) then gives
\[
 \omega_{\cH}^{A_{4r}^+}(\Delta_r)
 \geq u(A_{4r}^+)
 \geq c.
\]
The factor \(4\) may be replaced by any fixed factor that preserves the
causal separation and the required Harnack-chain geometry.  The lower
bound then depends on that factor and is not uniform for arbitrarily
distant poles.  See
\cite{FabesGarofaloSalsa,FabesSafonovYuan,Nystrom}.
\end{rem}

\begin{rem}\label{rem:green-function-harnack}
Fix \(P=(X,t)\).  As a function of \(Z=(Y,s)\),
\(G_{\cH}^{\Omega}(P;Z)\) solves the adjoint equation away from \(P\).
Consequently, if \(Z_1\) and \(Z_2\) are connected by an adjoint Harnack
chain that stays quantitatively separated from \(P\) and from the
parabolic boundary, then the comparison in
\eqref{eq:adjoint-harnack} applies to the two Green-function values.  In
this adjoint-variable description, \(P\) is the singularity rather than
the second-variable pole in our forward Green-function convention.  If
the second point is fixed and the first Green variable is allowed to
vary, the forward comparison \eqref{eq:forward-harnack} applies instead.
\end{rem}

\begin{lem}
\label{lem:green-function-scale}
Let \(\Omega\) be either the half-space or a time-oriented cylinder in
the fixed quantitative class of Remark~\ref{rem:domain-convention}.
Let \(P=(X,t),Z=(Y,s)\in\Omega\) satisfy
\[
 cr^2\leq t-s\leq Cr^2,
 \qquad |X-Y|\lesssim r.
\]
Assume that there is a point \(P_1=(X_1,t_1)\) such that
\[
 cr^2\leq t_1-s\leq Cr^2,
 \qquad |X_1-Y|\lesssim r,
\]
and a standard time-oriented cylinder
\(\mathcal C_r=B(X_c,\gamma r)\times(\tau_c,\tau_c+\sigma r^2)\), where
\(\gamma,\sigma\) are bounded above and below by fixed positive constants.
Assume that \(Z,P_1\in\mathcal C_r\), that their spatial components are at
distance at least \(cr\) from \(\partial B(X_c,\gamma r)\), that
\(s-\tau_c\geq cr^2\) and
\(\tau_c+\sigma r^2-t_1\geq cr^2\), and that a fixed enlargement of
\(\mathcal C_r\) is contained in \(\Omega\).  Assume further that
\(P_1\) can be joined to \(P\) by a forward Harnack chain of uniformly
bounded length, with \(P_1\) occurring before \(P\).  Every cylinder in
this chain is assumed to have radius comparable to \(r\), and its fixed
enlargement is contained in \(\Omega\), remains at parabolic distance
comparable to \(r\) from \(Z\), and is quantitatively separated from all
spatial and temporal boundary faces.  Then
\begin{equation}\label{eq:green-scale}
 G_{\cH}^{\Omega}(P;Z)\simeq r^{-(n+1)}.
\end{equation}
The comparison constants depend only on \(n,\mu,\Lambda\) and the fixed
geometric constants in the preceding assumptions.  The adjoint statement
holds with the causal direction reversed.
\end{lem}

\begin{proof}
The upper estimate follows from Lemma~\ref{lem:green-upper}.  For the
lower estimate, the interior Green-function lower bound in
\(\mathcal C_r\), together with domain monotonicity, gives
\[
 G_{\cH}^{\Omega}(P_1;Z)
 \geq
 G_{\cH}^{\mathcal C_r}(P_1;Z)
 \gtrsim r^{-(n+1)}.
\]
See~\cite{Aronson,FabesGarofaloSalsa,Nystrom}.  Now fix \(Z\) and regard
\[
 Q\longmapsto G_{\cH}^{\Omega}(Q;Z)
\]
as a forward solution.  By hypothesis, every member of the Harnack chain
from \(P_1\) to \(P\) is quantitatively separated from the singularity
\(Z\).  The forward Harnack inequality may therefore be applied
successively along the chain, and it yields
\[
 G_{\cH}^{\Omega}(P_1;Z)
 \lesssim
 G_{\cH}^{\Omega}(P;Z).
\]
Combining the last two estimates proves the lower bound in
\eqref{eq:green-scale}.
\end{proof}

The exponent in \eqref{eq:green-scale} reflects the \(n+1\) spatial
variables of the operator.  At time separation comparable to \(r^2\),
the fundamental solution has size \(r^{-(n+1)}\).  Without the boundary
clearance in Lemma~\ref{lem:green-function-scale}, only the upper bound is
used.

\begin{lem}
\label{lem:change-of-pole}
Let \(\Delta_r=\Delta_r(x_0,t_0)\), and let \(\Omega\) be either the
half-space or a time-oriented cylinder satisfying the convention in
Remark~\ref{rem:domain-convention} relative to \(\Delta_r\).  Set
\(P_r^0=A_{4r}^+(x_0,t_0)\), and assume that
\(P,P_r^0\in\Omega\cap\mathcal P_{\kappa,r}^+(x_0,t_0)\).  Then, for
every Borel set \(E\subset\Delta_r\),
\begin{equation}\label{eq:change-of-pole}
 \frac{\omega_{\cH,\Omega}^P(E)}
      {\omega_{\cH,\Omega}^P(\Delta_r)}
 \simeq
 \frac{\omega_{\cH,\Omega}^{P_r^0}(E)}
      {\omega_{\cH,\Omega}^{P_r^0}(\Delta_r)}.
\end{equation}
The comparison constant depends only on
\(n,\mu,\Lambda,\kappa\) and the fixed cylinder geometry.  The
corresponding adjoint statement holds with
\(P_r^0=A_{4r}^-(x_0,t_0)\),
\(\mathcal P_{\kappa,r}^-\), and
\(\omega_{\cH^*,\Omega}\) in place of
\(\mathcal P_{\kappa,r}^+\) and
\(\omega_{\cH,\Omega}\), respectively.
\end{lem}

\begin{proof}
Set
\[
 g_P(Z)=G_{\cH}^{\Omega}(P;Z),
 \qquad
 g_0(Z)=G_{\cH}^{\Omega}(P_r^0;Z).
\]
Both functions are positive adjoint solutions in the common comparison
box and vanish on its flat boundary.  Their singularities lie outside
that box.  Lemma~\ref{lem:green-boundary-time-comparison} controls their
reference balance parameters.  Hence adjoint boundary comparison applies
with a constant depending only on the structural data and the fixed pole
geometry.  Combining the comparison at the corkscrews above every
subcube of \(\Delta_r\) with Lemma~\ref{lem:cfms} proves
\eqref{eq:change-of-pole} first for subcubes.  Differentiation with respect
to the parabolic Vitali basis gives mutually bounded Radon-Nikodym
derivatives for the normalized restrictions.  Inner and outer regularity
extend the estimate to all Borel sets.
\end{proof}

\begin{rem}
The comparison constant depends on the fixed aperture and causal
separation in the admissibility condition.  It does not depend on the
distance from \(P\) to \(\Delta_r\), which may be much larger than \(r\).
\end{rem}

Lemmas~\ref{lem:parabolic-measure-doubling}
and~\ref{lem:change-of-pole} imply boundary doubling for every admissible
forward pole.  The doubling constant depends only on the structural
constants and the fixed aperture in the admissibility condition.  Moreover,
combining
Lemmas~\ref{lem:parabolic-measure-doubling},~\ref{lem:green-boundary-time-comparison},
and~\ref{lem:cfms} gives the
abbreviated half-space form
\begin{equation}\label{eq:cfms}
 \frac{\omega_{\cH}^P(\Delta_r)}{|\Delta_r|}
 \simeq\frac{G_{\cH}(P;A_r^\pm)}{r}.
\end{equation}
The estimate holds for either choice of sign.  Here \(P\) and the reference
point have the stated forward causal
separation.  The same formula holds in a time-oriented cylinder, with
\(\omega_{\cH,\Omega}^P\) and \(G_{\cH}^{\Omega}\), whenever the fixed
enlargement required in Lemma~\ref{lem:cfms} remains inside the cylinder.
The adjoint formula has the time directions reversed.

\subsection{Localization and change of operator}
\label{subsec:local-change-tools}

Fix \(\kappa>16\) and set
\begin{equation}\label{eq:localization-cylinder}
 \mathcal U_{\kappa r}(x_0,t_0)
 =(0,\kappa r)\times B(x_0,\kappa r)
 \times(t_0-\kappa^2r^2,t_0+\kappa^2r^2).
\end{equation}
Set \(Z_r=A_{4r}^+(x_0,t_0)\).  Since \(\kappa>16\), the point \(Z_r\)
belongs to \(\mathcal U_{\kappa r}(x_0,t_0)\) and lies outside a fixed
smaller comparison box above \(2\Delta_r(x_0,t_0)\).

\begin{lem}
\label{lem:localization-parabolic-measure}
Fix \(\kappa_P\geq1\).  Let
\(P\in\mathcal P_{\kappa_P,r}^+(x_0,t_0)\) and assume that
\(t(P)>t_0+\kappa^2r^2\), so that \(P\) lies beyond the terminal time
face of \(\mathcal U_{\kappa r}(x_0,t_0)\).  Let
\(\omega_{\mathcal U}^{Z_r}\) be parabolic measure for the same operator
but in \(\mathcal U_{\kappa r}(x_0,t_0)\).  Then, for every Borel set
\(E\subset\Delta_r(x_0,t_0)\),
\begin{equation}\label{eq:localization-parabolic-measure}
 \frac{\omega_{\cH}^P(E)}
      {\omega_{\cH}^P(\Delta_r)}
 \simeq
 \frac{\omega_{\mathcal U}^{Z_r}(E)}
      {\omega_{\mathcal U}^{Z_r}(\Delta_r)}.
\end{equation}
The comparison constant depends only on
\(n,\mu,\Lambda,\kappa\), and \(\kappa_P\).
\end{lem}

\begin{proof}
The global Green function
\[
 g_P(Y,s)=G_{\cH}(P;Y,s)
\]
is a positive adjoint solution in the localization cylinder and vanishes
on its flat boundary portion.  Its adjoint singularity lies at \(P\),
outside the cylinder.  Define
\[
 g_{Z_r}(Y,s)
 :=G_{\cH}^{\mathcal U_{\kappa r}(x_0,t_0)}(Z_r;Y,s).
\]
As a function of \((Y,s)\), this is also an adjoint solution.  Its
adjoint-variable singularity is \(Z_r\), which lies outside the fixed
smaller box above \(\Delta_r\).
Lemma~\ref{lem:green-boundary-time-comparison} controls the reference
balance parameters of \(g_P\) and \(g_{Z_r}\).  The adjoint boundary
comparison principle therefore applies with a structural constant and
compares their normalized values at corkscrews above every subcube
\(\Delta_\rho(y,s)\subset2\Delta_r\).  Lemma~\ref{lem:cfms}
therefore gives the analogue of
\eqref{eq:localization-parabolic-measure} first for such subcubes.  A
parabolic Vitali covering argument and the inner and outer regularity of
the measures extend the comparison to every Borel subset of \(\Delta_r\).
\end{proof}

\begin{rem}
The parameter \(\kappa_P\) is the aperture of the forward pole region
\(\mathcal P_{\kappa_P,r}^+(x_0,t_0)\).  Thus it controls the spatial
and transverse location of \(P\) relative to its time separation from
\((x_0,t_0)\).  It is distinct from \(\kappa\), which determines the
size of the localization cylinder
\(\mathcal U_{\kappa r}(x_0,t_0)\).  The comparison constant in
\eqref{eq:localization-parabolic-measure} depends on both \(\kappa\)
and \(\kappa_P\).  Since
\(Z_r=A_{4r}^+(x_0,t_0)=(16r,x_0,t_0+256r^2)\), its distances from the
upper transverse face and the terminal time face
of \(\mathcal U_{\kappa r}(x_0,t_0)\) are \((\kappa-16)r\) and
 \((\kappa^2-256)r^2\), respectively.  Consequently, the comparison constant may deteriorate as
\(\kappa\downarrow16\).
\end{rem}

\begin{cor}
\label{cor:local-change-operator}
Fix \(\kappa>16\) and \(\kappa_P\geq1\).  For \(i\in\{1,2\}\), let
\[
 \cH_i
 =
 \partial_t-\operatorname{div}_{\lambda,x}
 \bigl(A_i\nabla_{\lambda,x}\bigr),
\]
where \(A_i\) is real, bounded, and uniformly elliptic with common
ellipticity constants.  Suppose that
\[
 A_1=A_2
 \quad\text{almost everywhere in }
 \mathcal U_{\kappa r}(x_0,t_0).
\]
Let
\[
 P_i\in\mathcal P_{\kappa_P,r}^+(x_0,t_0),
 \qquad
 t(P_i)>t_0+\kappa^2r^2,
 \qquad i=1,2.
\]
Then, for every Borel set \(E\subset\Delta_r(x_0,t_0)\),
\begin{equation}\label{eq:local-change-package}
 \frac{\omega_{\cH_1}^{P_1}(E)}
      {\omega_{\cH_1}^{P_1}(\Delta_r)}
 \simeq
 \frac{\omega_{\cH_2}^{P_2}(E)}
      {\omega_{\cH_2}^{P_2}(\Delta_r)}.
\end{equation}
The comparison constant depends only on \(n\), the common ellipticity
constants, \(\kappa\), and \(\kappa_P\).
\end{cor}

\begin{proof}
Set \(\mathcal U=\mathcal U_{\kappa r}(x_0,t_0)\), \(Z_r=A_{4r}^+(x_0,t_0)\), and denote by \(\omega_{i,\mathcal U}^{Z_r}\) the parabolic measure for
\(\cH_i\) in \(\mathcal U\), evaluated at \(Z_r\).
Lemma~\ref{lem:localization-parabolic-measure} gives
\[
 \frac{\omega_{\cH_i}^{P_i}(E)}
      {\omega_{\cH_i}^{P_i}(\Delta_r)}
 \simeq
 \frac{\omega_{i,\mathcal U}^{Z_r}(E)}
      {\omega_{i,\mathcal U}^{Z_r}(\Delta_r)},
 \qquad i=1,2.
\]
Since \(A_1=A_2\) almost everywhere in \(\mathcal U\), the two local
Dirichlet problems have the same weak solutions for the same boundary
data.  Uniqueness of the local Dirichlet problem therefore implies
\[
 \omega_{1,\mathcal U}^{Z_r}
 =
 \omega_{2,\mathcal U}^{Z_r}.
\]
Combining these identities with the preceding comparisons proves
\eqref{eq:local-change-package}.
\end{proof}

\subsection{Reverse H\"older classes and the Dirichlet problem}
\label{subsec:RH-Dirichlet-tools}

\begin{defn}\label{defn:parabolic-Ainfty}
Let \(\Delta_0=\Delta_{r_0}(x_0,t_0)\) and
\(P_0=A_{4r_0}^+(x_0,t_0)\).  We say that
\(\omega_{\cH}^{P_0}\in A_\infty(\Delta_0,\mathrm d x\,\mathrm d t)\)
if, for every \(\varepsilon>0\), there is
\(\delta=\delta(\varepsilon)>0\) such that
\begin{equation}\label{eq:Ainfty-definition}
 \frac{\omega_{\cH}^{P_0}(E)}
      {\omega_{\cH}^{P_0}(\Delta)}<\delta
 \quad\Longrightarrow\quad
 \frac{|E|}{|\Delta|}<\varepsilon
\end{equation}
whenever \(\Delta\subset\Delta_0\) and \(E\subset\Delta\) is Borel.  We
write \(\omega_{\cH}\in A_\infty(\mathrm d x\,\mathrm d t)\) when this
holds for every \(\Delta_0\), with uniform constants.
\end{defn}

\begin{rem}
The direction of the implication in \eqref{eq:Ainfty-definition} is the
convention used in the parabolic-measure literature cited here.  Together
with local doubling, it is quantitatively equivalent to the usual
weight-theoretic formulations of \(A_\infty\).  See
\cite{CoifmanFefferman,Nystrom}.
\end{rem}

\begin{defn}\label{defn:parabolic-Bp}
Let \(\Delta_0=\Delta_{r_0}(x_0,t_0)\) and
\(P_0=A_{4r_0}^+(x_0,t_0)\).  For \(1<p<\infty\), we say that
\(\omega_{\cH}\in B_p(\mathrm d x\,\mathrm d t)\), equivalently
\(\omega_{\cH}\in \mathrm{RH}_p(\mathrm d x\,\mathrm d t)\), if
\(\omega_{\cH}^{P_0}\ll\mathrm d x\,\mathrm d t\) and its Poisson kernel
\[
 k^{P_0}(x,t)
 =\frac{\mathrm d \omega_{\cH}^{P_0}}
       {\mathrm d x\,\mathrm d t}(x,t)
\]
satisfies
\begin{equation}\label{eq:Bp-definition}
 \left(
 \frac1{|\Delta|}
 \iint_\Delta (k^{P_0}(x,t))^p\,\mathrm d x\,\mathrm d t
 \right)^{1/p}
 \leq
 \frac{C}{|\Delta|}
 \iint_\Delta k^{P_0}(x,t)\,\mathrm d x\,\mathrm d t
\end{equation}
whenever \(\Delta\subset\Delta_0\).  The constant is independent of
\(x_0,t_0,r_0\), and \(\Delta\).
\end{defn}

\begin{lem}
\label{lem:cube-conventions}
Fix \(1<p<\infty\).  The following assertions are quantitatively
equivalent, after a fixed adjustment of the pole-admissibility constants:
the reverse H\"older estimate of exponent \(p\) holds uniformly for every
center \((x,t)\in\R^n\times\R\), every \(r>0\), and every admissible pole
for the symmetric cube \(\Delta_r(x,t)\); or the corresponding estimate
holds with the same quantifiers for the backward cube
\(\Delta_r^-(x,t)\).  The same equivalence holds with all radii restricted
to \(0<r\leq r_0\), provided the harmless fixed enlargements occurring in
the covering argument are also allowed.
\end{lem}

\begin{proof}
Each symmetric cube is, up to a set of Lebesgue measure zero, the union
of two backward cubes obtained by translating the time coordinate of the
center.
Conversely, every backward cube can be covered by a fixed number of
symmetric cubes of comparable radius contained in a fixed enlargement
of the backward cube.

Let \(E\) denote the cube on which the estimate is to be proved and let
\(\{E_j\}_{j=1}^{N_0}\) be the corresponding finite cover by cubes of
the other type.  The number \(N_0\) is structural,
\(|E_j|\simeq|E|\), and \(\bigcup_jE_j\) is contained in a fixed
enlargement \(E^\sharp\) of \(E\).  After adjusting the admissibility
constants and applying Lemma~\ref{lem:change-of-pole}, the assumed
reverse H\"older estimate may be applied to \(k^P\) on every \(E_j\),
with a uniform constant.  Hence
\begin{align*}
 \iint_E(k^P)^p\,\mathrm d x\,\mathrm d t
 \leq
 \sum_{j=1}^{N_0}
 \iint_{E_j}(k^P)^p\,\mathrm d x\,\mathrm d t&\lesssim
 |E|^{1-p}
 \sum_{j=1}^{N_0}
 \bigl(\omega_{\cH}^P(E_j)\bigr)^p\\
 &\leq
 |E|^{1-p}
 \biggl(
 \sum_{j=1}^{N_0}\omega_{\cH}^P(E_j)
 \biggr)^p\\
 &\lesssim
 |E|^{1-p}
 \bigl(\omega_{\cH}^P(E^\sharp)\bigr)^p\lesssim
 |E|^{1-p}
 \bigl(\omega_{\cH}^P(E)\bigr)^p.
\end{align*}
If \(E\) is a backward cube, choose a symmetric cube
\(E_0\subset E\) of comparable radius.  Since
\(E^\sharp\subset CE_0\), finitely many applications of
Lemma~\ref{lem:parabolic-measure-doubling}, together with the same
change-of-pole adjustment, give
\[
 \omega_{\cH}^P(E^\sharp)
 \lesssim
 \omega_{\cH}^P(E_0)
 \leq
 \omega_{\cH}^P(E).
\]
If \(E\) is symmetric, the same conclusion follows directly from
boundary doubling.  This justifies the last estimate above.  Taking the
\(p\)-th root and dividing by \(|E|^{1/p}\) gives the required reverse
H\"older estimate on \(E\).  The same argument applies in the opposite
direction.
\end{proof}

\begin{prop}
\label{prop:Ainfty-union-Bp}
Under the hypotheses of
Lemma~\ref{lem:parabolic-measure-doubling},
\begin{equation}\label{eq:Ainfty-union-Bp}
 A_\infty(\mathrm d x\,\mathrm d t)
 =\bigcup_{p>1}B_p(\mathrm d x\,\mathrm d t).
\end{equation}
In particular, parabolic measure belongs to
\(A_\infty(\mathrm d x\,\mathrm d t)\) if and only if it belongs to
\(B_p(\mathrm d x\,\mathrm d t)\) for some \(p>1\).  The
\(A_\infty\) property implies mutual absolute continuity of parabolic
measure and Lebesgue measure.
\end{prop}

\begin{proof}
Parabolic measure is doubling by
Lemma~\ref{lem:parabolic-measure-doubling}, and Lebesgue measure is
doubling by parabolic scaling.  The assertion is therefore
the standard reverse H\"older characterization and self-improvement of
\(A_\infty\) weights.  We refer to~\cite{CoifmanFefferman}.
\end{proof}

For a function \(F\) in the half-space, define
\begin{equation}\label{eq:NTmaxDef}
 N_\ast F(x,t)
 =\sup_{\lambda>0}
 \operatorname*{ess\,sup}_{\substack{
  \lambda/2<\lambda'<\lambda\\
  |y-x|<\lambda,\ |s-t|<\lambda^2}}
 |F(\lambda',y,s)|.
\end{equation}
Given \(\eta>0\), the parabolic cone with vertex \((x_0,t_0)\) is
\begin{equation}\label{eq:parabolic-cone}
 \Gamma^\eta(x_0,t_0)
 =\left\{(\lambda,x,t):
 d_p((x,t),(x_0,t_0))<\eta\lambda\right\}.
\end{equation}

\begin{defn}
\label{defn:Dq}
Let \(1<q<\infty\).  We say that the Dirichlet problem \(D_q\) for
\(\cH\) is solvable if, for every
\(f\in\mathrm L^q(\R^n\times\R)\), there exists a weak solution \(u\)
such that
\begin{align}
 \cH u&=0 &&\text{in }\R^{n+2}_+,
 \label{eq:Dq-equation}\\
 \|N_\ast u\|_{\mathrm L^q(\R^n\times\R)}
 &\leq C\|f\|_{\mathrm L^q(\R^n\times\R)},
 \label{eq:Dq-maximal}\\
 u(\lambda,\cdot,\cdot)&\longrightarrow f
 &&\text{in }\mathrm L^q(\R^n\times\R)
   \text{ as }\lambda\downarrow0.
 \label{eq:Dq-trace}
\end{align}
The solution is also required to converge to \(f\) almost everywhere as
the point approaches the boundary through
\(\Gamma^\eta(x_0,t_0)\), for one (equivalently, every) fixed aperture
\(\eta>0\).  We say that \(D_q\) is uniquely solvable if
this solution is unique among all weak solutions for which
\(N_\ast u\in\mathrm L^q(\R^n\times\R)\).
\end{defn}

\begin{thm}
\label{thm:Bp-Dq}
Let \(1<p<\infty\) and \(q=p'=p/(p-1)\).  Then
\begin{equation}\label{eq:Bp-Dq-equivalence}
 \omega_{\cH}\in B_p(\mathrm d x\,\mathrm d t)
 \quad\Longleftrightarrow\quad
 D_q\text{ is solvable}.
\end{equation}
The reverse H\"older constant and the Dirichlet solvability constant
determine one another, with additional dependence only on
\(p,n,\mu,\Lambda\) and the fixed geometric parameters.
\end{thm}

\begin{rem}
Theorem~\ref{thm:Bp-Dq} concerns solvability, not uniqueness.  If the
forward measure for \(\cH\) and the adjoint measure for \(\cH^*\) belong
to the corresponding \(B_p\) classes, then \(D_q\) is uniquely solvable.
See~\cite[Remark~1.4]{AENDirichlet}.  Nystr\"om proves the equivalence in
\cite[Theorem~6.2]{Nystrom} under a symmetry hypothesis.  The
potential-theoretic ingredients do not require symmetry, as observed in
\cite[Section~1.3]{AENDirichlet}.
\end{rem}

\subsection{Perturbations of parabolic measure}
\label{subsec:parabolic-measure-perturbations}

For coefficient matrices \(B_0=B_0(x,t)\) and
\(B_1=B_1(\lambda,x,t)\), define
\[
 \mathcal E_{B_1,B_0}(\lambda,x,t)
 :=
 \operatorname*{ess\,sup}_{(\lambda',y,s)\in W(\lambda,x,t)}
 |B_1(\lambda',y,s)-B_0(y,s)|,
\]
where \(W(\lambda,x,t)\) is a fixed parabolic Whitney region on which
\(\lambda'\simeq\lambda\), \(|y-x|\lesssim\lambda\), and
\(|s-t|\lesssim\lambda^2\).  Set
\begin{equation}\label{eq:carleson-discrepancy}
 \|B_1-B_0\|_*^2
 :=
 \sup_\Delta\frac1{|\Delta|}
 \iiint_{(0,\ell(\Delta))\times\Delta}
 \bigl(\mathcal E_{B_1,B_0}(\lambda,x,t)\bigr)^2
 \,\frac{\mathrm d\lambda\,\mathrm d x\,\mathrm d t}{\lambda},
\end{equation}
where \(\Delta\) ranges over boundary parabolic cubes.

\begin{prop}\label{prop:parabolic-Carleson-perturbation}
Let \(\cH_j=\partial_t-\operatorname{div}_{\lambda,x}
(B_j\nabla_{\lambda,x})\), \(j\in\{0,1\}\), have real, bounded,
uniformly elliptic coefficients in \(\R^{n+2}_+\), and assume that
\(B_0=B_0(x,t)\) is independent of \(\lambda\).
\begin{enumerate}[label=\textup{(\roman*)},leftmargin=2.5em]
\item If \(\omega_{\cH_0}\in B_2(\mathrm d x\,\mathrm d t)\), then there is
\(\varepsilon_*>0\), depending only on the structural constants and the
\(B_2\)-characteristic of \(\omega_{\cH_0}\), such that
\(\|B_1-B_0\|_*\leq\varepsilon_*\) implies
\(\omega_{\cH_1}\in B_2(\mathrm d x\,\mathrm d t)\), with quantitative
control.
\item If \(\omega_{\cH_0}\in A_\infty(\mathrm d x\,\mathrm d t)\) and
\(\|B_1-B_0\|_*<\infty\), then
\(\omega_{\cH_1}\in A_\infty(\mathrm d x\,\mathrm d t)\).  The
quantitative constants depend on the structural constants, the
\(A_\infty\)-characteristic of \(\omega_{\cH_0}\), and the discrepancy
norm.
\end{enumerate}
No symmetry assumption is required.
\end{prop}

\begin{proof} The precise small- and finite-Carleson Dirichlet perturbation conclusions used here are
recorded in \cite[Proposition~2.25]{Ulmer}, applied with \(p=2\) in
\textup{(i)} and specialized to the upper half-space.  For the original
argument in the symmetric setting, see
\cite[Theorems~6.4-6.6]{Nystrom}: Theorem~6.4 gives the
finite-Carleson \(A_\infty\) conclusion, while the small-norm version
of the argument proving Theorem~6.5, whose perturbative core is
Theorem~6.6, gives preservation of \(B_p\).   The
Carleson quantity in those references is \(\|B_1-B_0\|_*^2\), here we have
merely reparametrized the smallness threshold.
\end{proof}

\section{Transverse differences and boundary comparisons}
\label{sec:normal-comparison}

In this section we prove the estimates which reduce the large-scale
argument to a unit-height bound for an adjoint Green function.  We begin
with the elementary consequence of periodicity in the transverse variable.
This first estimate is diagnostic and is discussed further in
Remark~\ref{rem:normal-periodicity-obstruction}.  The proof of
Theorem~\ref{thm:large-scale} begins with the two-comparison estimate in
Subsection~\ref{subsec:two-comparison}.

\subsection{The transverse-difference estimate}\label{subsec:normal-difference}

For an interior point \(Z_0=(\lambda_0,x_0,t_0)\), write
\begin{equation}\label{eq:ambient-parabolic-cylinder}
 \cC_r(Z_0)=
 \bigl((\lambda_0-r,\lambda_0+r)\cap(0,\infty)\bigr)
 \times B(x_0,r)\times(t_0-r^2,t_0+r^2).
\end{equation}
This is only an ambient cylinder.  When local boundedness is applied to a
forward solution, the estimate uses a subcylinder lying in the past of the
evaluation point.  For an adjoint solution, it uses a subcylinder lying in
the future.  Both are contained in \eqref{eq:ambient-parabolic-cylinder}.
All boundary versions below concern the spatial boundary
\(\{\lambda=0\}\) or a lateral spatial face, not a terminal time face.

Define
\begin{equation}\label{eq:normal-difference}
 \delta_\lambda u(\lambda,x,t)
 =u(\lambda+1,x,t)-u(\lambda,x,t).
\end{equation}

\begin{lem}\label{lem:commute}
If \(u\) is a weak solution of \(\cH u=0\), then \(\delta_\lambda u\) is a weak
solution wherever both terms in \eqref{eq:normal-difference} are defined.  The same is true
for the adjoint equation.
\end{lem}

\begin{proof}
Let \(\tau_1u(\lambda,x,t)=u(\lambda+1,x,t)\).  By the change of variable
\(\rho=\lambda+1\) in the weak formulation and the periodicity
\eqref{eq:periodicity},
\[
 \cH(\tau_1u)=\tau_1(\cH u)=0.
\]
Subtracting the equation for \(u\) proves the assertion.  Since
\(A^{\mathsf T}\) has the same periodicity, the adjoint assertion follows
identically.
\end{proof}

\begin{lem}\label{lem:difference}
Let \(r\geq8\), and suppose that \(u\) solves either the forward or the
adjoint equation in an interior parabolic cylinder \(\cC_{2r}(Z_0)\).
Assume that the translated cylinders needed to define \(\delta_\lambda u\) are
contained in \(\cC_{2r}(Z_0)\).  Then
\begin{equation}\label{eq:difference-estimate}
 |\delta_\lambda u(Z_0)|
 \leq \frac Cr
 \left(
 \frac{1}{|\cC_{2r}(Z_0)|}
 \iiint_{\cC_{2r}(Z_0)}|u|^2
 \,\mathrm d \lambda\,\mathrm d x\,\mathrm d t
 \right)^{1/2}.
\end{equation}
The same estimate holds up to a flat or lateral boundary portion on which
both \(u\) and \(\tau_1u\) have zero trace.
\end{lem}

\begin{proof}
By Lemma~\ref{lem:commute}, \(\delta_\lambda u\) is a solution in
\(\cC_r(Z_0)\).  Interior local boundedness therefore gives
\begin{equation}\label{eq:moser-difference}
 |\delta_\lambda u(Z_0)|^2
 \leq \frac{C}{|\cC_r(Z_0)|}
 \iiint_{\cC_r(Z_0)}|\delta_\lambda u|^2
 \,\mathrm d \lambda\,\mathrm d x\,\mathrm d t.
\end{equation}
The fundamental theorem of calculus and Cauchy-Schwarz imply
\begin{align}
 |\delta_\lambda u(\lambda,x,t)|^2
 &=\left|\int_0^1\partial_\lambda
           u(\lambda+s,x,t)\,\mathrm d s\right|^2 \leq\int_0^1|\partial_\lambda
           u(\lambda+s,x,t)|^2\,\mathrm d s.        \label{eq:FTC-difference}
\end{align}
After integration and a change of variables,
\begin{equation}\label{eq:integrated-difference}
 \iiint_{\cC_r(Z_0)}|\delta_\lambda u|^2
 \,\mathrm d \lambda\,\mathrm d x\,\mathrm d t
 \leq C\iiint_{\cC_{r+2}(Z_0)}|\partial_\lambda u|^2
 \,\mathrm d \lambda\,\mathrm d x\,\mathrm d t.
\end{equation}
Caccioppoli's inequality, with a cutoff equal to one on
\(\cC_{r+2}(Z_0)\) and supported in \(\cC_{2r}(Z_0)\), gives
\begin{equation}\label{eq:caccioppoli-difference}
 \iiint_{\cC_{r+2}(Z_0)}|\partial_\lambda u|^2
 \,\mathrm d \lambda\,\mathrm d x\,\mathrm d t
 \leq \frac C{r^2}
 \iiint_{\cC_{2r}(Z_0)}|u|^2
 \,\mathrm d \lambda\,\mathrm d x\,\mathrm d t.
\end{equation}
Since \(|\cC_r(Z_0)|\simeq r^{n+3}\), inserting
\eqref{eq:integrated-difference}-\eqref{eq:caccioppoli-difference} into
\eqref{eq:moser-difference} proves \eqref{eq:difference-estimate}.  If the
cylinder meets a boundary portion on which both \(u\) and \(\tau_1u\)
have zero trace, the same proof uses boundary local boundedness in place of interior
local boundedness and the boundary Caccioppoli inequality in place of
\eqref{eq:caccioppoli-difference}.
\end{proof}

\begin{rem}
If \(G\) is a Green function whose singularity and the cylinder in
Lemma~\ref{lem:difference} have the correct causal order and are separated
by parabolic distance comparable to \(r\), then
Lemma~\ref{lem:green-upper} gives \(G\lesssim r^{-(n+1)}\) throughout a
fixed enlargement of the cylinder.  Lemma~\ref{lem:difference} therefore
gives
\begin{equation}\label{eq:difference-Green-scale}
 |\delta_\lambda G|\lesssim r^{-1}r^{-(n+1)}=r^{-(n+2)}.
\end{equation}
This one-power improvement is the quantitative core of Dahlberg's
large-scale argument.
\end{rem}

The next estimate combines two adjoint boundary comparisons.  The time ordering is included in the statement of the setup because it is essential in the parabolic argument.

\subsection{The two-comparison estimate}\label{subsec:two-comparison}

Let \(r\geq20\) and put \(z_0=(x_0,t_0)\).  Choose points
\(z_k=(x_k,t_k)\) such that the unit backward cubes
\[
 \Delta_k=\Delta_1^-(z_k)
 =B(x_k,1)\times(t_k-1,t_k)
\]
cover \(\Delta_r^-(z_0)\), are contained in a fixed enlargement of
\(\Delta_r^-(z_0)\), and have bounded overlap after a fixed dilation.
For each \(k\), set
\begin{equation}\label{eq:unit-corkscrews}
 A_k=A_1^+(z_k)=(4,x_k,t_k+16).
\end{equation}
Thus \(A_k\) is the future-oriented unit corkscrew associated with
\(\Delta_k\).  Let \(P\) be the \(M\)-admissible pole associated with
the Poisson kernel under consideration.

Choose two reference points \(P_1\) and \(P_2\) at spatial scale \(r\),
so that
\[
 \lambda(P_i)\simeq r,
 \qquad
 |x(P_i)-x_0|\lesssim r,
 \qquad i=1,2.
\]
Their time coordinates are chosen so that every \(A_k\) lies
quantitatively in the common causal past of \(P\), \(P_1\), and \(P_2\).
More precisely,
\begin{equation}\label{eq:time-order}
 \begin{aligned}
  &\min\{t(P),t(P_1)\}-\sup_k t(A_k)\gtrsim r^2,
  \qquad t(P_1)<t(P_2),\\
  &t(P_1)-t_0\simeq r^2,
  \qquad t(P_2)-t(P_1)\simeq r^2,
  \qquad t(P)-t_0\gtrsim M^2r^2.
 \end{aligned}
\end{equation}
The fixed geometric constants and \(M\)
are chosen so that, for some fixed \(c_0>0\),
\eqref{eq:time-order} yields
\[
 \sup_k t(A_k)+c_0r^2
 \leq
 t(P_1)
 \leq
 t(P_2)-c_0r^2,
 \qquad
 t(P_2)
 \leq
 t(P)-c_0r^2.
\]
Thus every \(A_k\) lies quantitatively in the common causal past of
\(P_1\), \(P_2\), and \(P\).  The spatial positions, transverse heights, and time separations are fixed
so that all the Harnack chains used below have uniformly controlled
geometry.

Choose two time-oriented boundary cylinders
\[
 \Omega_i=D_i\times(a_i,b_i),
 \qquad i=1,2,
\]
whose spatial diameters are comparable to \(r\) and whose time lengths
are comparable to \(r^2\).  We require
\(\Omega_1\subset\Omega_2\), with the artificial spatial and temporal
faces of \(\Omega_1\) quantitatively separated from the corresponding
faces of \(\Omega_2\).  The cylinders may share part of their flat
boundary on \(\{\lambda=0\}\).  We impose the following additional
conditions.

\begin{enumerate}[label=\textup{(\roman*)},leftmargin=2.5em]
\item
The flat boundary portion of each cylinder contains a fixed enlargement
of \(\Delta_r^-(z_0)\).
\item
The point \(P_1\) belongs to \(\Omega_1\), whereas
\(P_2\in\Omega_2\setminus\overline{\Omega_1}\).  The exclusion of
\(P_2\) from \(\Omega_1\) is due to its spatial position.  The time
coordinates of both \(P_1\) and \(P_2\) belong to both time intervals.
\item
The time intervals are chosen so that
\[
 a_i<t_0-2r^2<t_0-r^2<t(P_1)<t(P_2)<b_i,
\]
and
\[
 (t_0-r^2)-a_i\simeq r^2,
 \qquad
 b_i-t(P_2)\simeq r^2,
\]
for \(i=1,2\).  Thus every comparison region used below is quantitatively separated from
the artificial initial and terminal time faces.
\end{enumerate}
For example, choose fixed constants \(C_2\gg C_1\gg1\), depending only on
the comparison geometry, and take
\[
 D_1=(0,C_1r)\times B(x_0,C_1r),
 \qquad
 D_2=(0,C_2r)\times B(x_0,C_2r).
\]
Taking \(x(P_1)=x(P_2)=x_0\), choose the transverse coordinates in fixed
interior subintervals so that \(P_1\) has distance comparable to \(r\)
from the spatial faces of \(D_1\), while \(P_2\)
has distance comparable to \(r\) from both \(\overline{D_1}\) and the
spatial faces of \(D_2\).  The time intervals are chosen consistently
with \eqref{eq:time-order} and the separation conditions above.

Whenever a forward representation in
\(\Omega_i=D_i\times(a_i,b_i)\) is used, the representing measure is
supported on the incoming parabolic boundary
\[
 \partial_p^{\mathrm{in}}\Omega_i
 =
 \bigl(\overline D_i\times\{a_i\}\bigr)
 \cup
 \bigl(\partial D_i\times(a_i,b_i)\bigr).
\]
In particular, the terminal face \(D_i\times\{b_i\}\) does not belong to
the forward parabolic boundary.

We now define
\begin{equation}\label{eq:two-adjoint-green}
 v_i(Z)=G_{\cH}^{\Omega_i}(P_i;Z),
 \qquad i=1,2.
\end{equation}
Writing \(Z=(Y,s)\), Green-function reciprocity and causality give
\[
 v_i(Y,s)=0
 \qquad\text{when }s\geq t(P_i)
 \text{ and }(Y,s)\neq P_i.
\]
In the causal past of \(P_i\), the function \(v_i\) is positive and
satisfies
\[
 \cH^*v_i=0
 \qquad\text{in }\Omega_i\cap\{s<t(P_i)\}.
\]
All boundary comparisons involving \(v_i\) below are carried out in
subregions of this causal past that remain quantitatively separated from
\(P_i\).

Let \(\omega_{\cH,\Omega_1}^{P_1}\) denote the forward parabolic
measure for \(\cH\) in \(\Omega_1\), evaluated at \(P_1\).  For each
\(k\), put
\[
 \widehat z_k=(x_k,t_k-\tfrac12),
 \qquad
 \widehat\Delta_k=\Delta_{1/4}(\widehat z_k).
\]
Then
\[
 2\widehat\Delta_k=\Delta_{1/2}(\widehat z_k)
 \subset\Delta_k.
\]
Let
\[
 \widetilde A_k=A_{1/2}^+(\widehat z_k).
\]
The time-ordering assumptions and \(r\geq20\) imply, with a fixed
\(\kappa_u\) independent of \(r\) and \(k\), that
\[
 P_1\in
 \Omega_1\cap\mathcal P_{\kappa_u,1/2}^+(\widehat z_k).
\]
Moreover, the fixed enlargement required in
Lemma~\ref{lem:cfms} is contained in \(\Omega_1\) and is
quantitatively separated from its artificial spatial and temporal
faces.  The left-hand inequality in
\eqref{eq:cfms-forward-raw}, applied at scale \(1/2\), therefore gives
\begin{equation}\label{eq:v1-cfms-intermediate}
 v_1(\widetilde A_k)
 =G_{\cH}^{\Omega_1}(P_1;\widetilde A_k)
 \lesssim
 \omega_{\cH,\Omega_1}^{P_1}(\widehat\Delta_k).
\end{equation}
By construction,
\[
 t(\widetilde A_k)=t_k+\frac72
 <t_k+16=t(A_k).
\]
The points \(\widetilde A_k\) and \(A_k\) are joined by a uniformly
bounded chain of unit-scale cylinders contained in
\(\Omega_1\cap\{t<t(P_1)\}\).  The chain remains quantitatively
separated from \(P_1\) and from the artificial boundary faces.
Since \(v_1\) is an adjoint solution in this region, the correctly
oriented adjoint Harnack inequality gives
\[
 v_1(A_k)\lesssim v_1(\widetilde A_k).
\]
Combining this estimate with \eqref{eq:v1-cfms-intermediate} and using
\(\widehat\Delta_k\subset\Delta_k\), we obtain
\begin{equation}\label{eq:parabolic-v1-omega}
 v_1(A_k)
 \lesssim
 \omega_{\cH,\Omega_1}^{P_1}(\Delta_k).
\end{equation}

If \((y,s)\in\Delta_k\), then
\[
 16\leq t(A_k)-s\leq17,
 \qquad
 |x_k-y|\leq1.
\]
Consequently, \(A_k\) and \((1,y,s)\) are joined by a uniformly bounded
chain of unit-scale cylinders, with \(A_k\) occurring later in time.
The entire chain is contained in
\(\Omega_2\cap\{t<t(P_2)\}\) and remains quantitatively separated from
\(P_2\) and from the artificial boundary faces.  Applying the adjoint
Harnack inequality in its reversed time orientation gives
\begin{equation}\label{eq:parabolic-v2-boundary}
 v_2(A_k)
 \lesssim
 v_2(1,y,s),
 \qquad (y,s)\in\Delta_k.
\end{equation}
All implicit constants are independent of \(r\), \(k\), and
\((y,s)\in\Delta_k\).

Put
\begin{align}\label{Greena}
 w(Z)&=G_{\cH}(P;Z).
\end{align}
On every pole-free comparison cylinder used below, \(w\), \(v_1\), and
\(v_2\) are positive adjoint solutions and vanish continuously on the
common flat boundary portion.

Fix the interior reference point
\[
 A_r^*=(r,x_0,t_0-2r^2).
\]
The constants defining \(\Omega_1\) are chosen so that
\(A_r^*\in\Omega_1\).  In particular,
\begin{equation}\label{eq:adjoint-reference-point}
 \lambda(A_r^*)=r,
 \qquad x(A_r^*)=x_0,
 \qquad t(A_r^*)=t_0-2r^2.
\end{equation}
Let \(c_{\mathrm{bc},0}\) and \(c_{\mathrm{bc},1}\) be the structural
constants in Lemma~\ref{lem:boundary-comparison}, with
\(c_{\mathrm{bc},0}\) enlarged, if necessary, so that
\(c_{\mathrm{bc},0}\geq4\), and put
\[
 z_r^c=(x_0,t_0-r^2).
\]
Since the points \(A_k\) are unit-scale corkscrew points associated with
cubes contained in a fixed enlargement of \(\Delta_r^-(z_0)\), they
satisfy
\[
 \lambda(A_k)\lesssim1,
 \qquad
 |x(A_k)-x_0|\lesssim r,
 \qquad
 |t(A_k)-(t_0-r^2)|\lesssim r^2.
\]
Together with \eqref{eq:adjoint-reference-point}, these estimates imply
that there is a fixed constant \(L_0\geq1\), independent of \(r\) and
\(k\), such that
\begin{equation}\label{eq:common-reduced-region}
 \{A_k:k\}\cup\{A_r^*\}
 \subset T_{L_0r}(z_r^c).
\end{equation}

Set
\[
 \rho=L_0c_{\mathrm{bc},1}r,
 \qquad
 R=2c_{\mathrm{bc},0}c_{\mathrm{bc},1}\rho
   =2L_0c_{\mathrm{bc},0}c_{\mathrm{bc},1}^2r.
\]
When Lemma~\ref{lem:boundary-comparison} is applied with scale \(R\),
its two auxiliary scales are
\[
 \rho_0=\frac{R}{c_{\mathrm{bc},0}},
 \qquad
 \rho_1=\frac{R}{c_{\mathrm{bc},0}c_{\mathrm{bc},1}}
        =2\rho.
\]
In particular, \(\rho<\rho_1\), and
\[
 T_{\rho/c_{\mathrm{bc},1}}(z_r^c)
 =T_{L_0r}(z_r^c).
\]
Define
\[
 \mathcal Q_r=T_{2R}(z_r^c),
 \qquad
 \mathcal Q_r^{\mathrm{red}}
 =T_{\rho/c_{\mathrm{bc},1}}(z_r^c),
 \qquad
 A_r^{\mathrm{ref}}=A_\rho(z_r^c).
\]
It follows from \eqref{eq:common-reduced-region} that
\begin{equation}\label{eq:common-comparison-points}
 A_k,A_r^*\in\mathcal Q_r^{\mathrm{red}}
 \qquad\text{for every }k.
\end{equation}
The geometry of the construction is summarized in
Figure~\ref{fig:two-comparison-geometry} in
Appendix~\ref{app:comparison-geometry}.

The fixed constants in the preceding choices of
\(\Omega_1,\Omega_2,P_1,P_2\), and \(M\) are taken sufficiently large
so that \(T_{4R}(z_r^c)\subset\Omega_1\) and this enlarged cylinder is
quantitatively separated from the artificial spatial and temporal
faces.  We also require
\[
 \sup\{t:(X,t)\in T_{4R}(z_r^c)\}
 \leq
 \min\{t(P),t(P_1),t(P_2)\}-cR^2
\]
for a fixed \(c>0\).  Thus \(T_{4R}(z_r^c)\) lies in the strict common
causal past of the three points \(P_1,P_2,P\).  Moreover, \(P\), \(P_1\), and \(P_2\)
are chosen so that Lemma~\ref{lem:green-boundary-time-comparison}
applies relative to
\(\Delta_{\rho_0}(z_r^c)\), in the half-space for \(P\), in
\(\Omega_1\) for \(P_1\), and in \(\Omega_2\) for \(P_2\).  Consequently, \(w\), \(v_1\), and \(v_2\) are strictly positive adjoint
solutions in \(\mathcal Q_r\), and they vanish continuously on its
common flat boundary portion.

For \(i=1,2\), we also impose the scale-\(r\) Green-function geometry
\begin{equation}\label{eq:reference-green-geometry}
 t(P_i)-t(A_r^*)\simeq r^2,
 \qquad
 |X(P_i)-X(A_r^*)|\lesssim r.
\end{equation}
More precisely, choose a point \(Q_i\) in the causal future of \(A_r^*\)
and in the causal past of \(P_i\)
such that \(A_r^*\) and \(Q_i\) lie in a scale-\(r\) interior cylinder
whose fixed enlargement is contained in \(\Omega_i\), with
\[
 t(Q_i)-t(A_r^*)\simeq r^2,
 \qquad
 |X(Q_i)-X(A_r^*)|\simeq r.
\]
A neighborhood of \(Q_i\) is connected to \(P_i\) by a forward
scale-\(r\) Harnack chain in the first Green variable.  Every cylinder
in this chain has parabolic distance between \(cr\) and \(Cr\) from
\(A_r^*\), and remains quantitatively separated from the parabolic
boundary of \(\Omega_i\).

The constants in this construction are selected in the following order.
First fix the constants in Lemma~\ref{lem:boundary-comparison}, the fixed
enlargement factors, and the structural constants \(L_0\) and \(L_h\)
used here and in Lemma~\ref{lem:unit-height-decay}.  This determines the
ratios of all auxiliary radii to \(r\).
The constants defining \(\Omega_1\), \(\Omega_2\), \(P_1\), \(P_2\),
and \(M\) are then chosen sufficiently large to satisfy the containment,
boundary-clearance, and causal-separation conditions above.  All these
quantities are fixed multiples of \(r\) or \(r^2\), so the requirements
are compatible and the resulting constants are independent of \(r\).

\begin{lem}
\label{lem:two-comparison}
Under the preceding geometric assumptions, \(w\) in \eqref{Greena} satisfies
\begin{equation}\label{eq:correct-comparison}
 \sum_k w(A_k)^2
 \lesssim
 r^{2(n+1)}w(A_r^*)^2
 \sum_k v_1(A_k)v_2(A_k).
\end{equation}
\end{lem}

\begin{proof}
Fix \(i\in\{1,2\}\).  The functions \(w\) and \(v_i\) are positive
solutions of the adjoint equation in \(\mathcal Q_r\), and they vanish
continuously on the same flat boundary portion.  Their singularities
lie in the strict future of \(T_{4R}(z_r^c)\).  At the balance scale \(\rho_0=R/c_{\mathrm{bc},0}\),
Lemma~\ref{lem:green-boundary-time-comparison}, applied separately to
\(w\) and \(v_i\) in their respective ambient domains, gives
\[
 \mathcal B_{\rho_0}^*(w,v_i)\leq C.
\]
The constant in the adjoint form of
Lemma~\ref{lem:boundary-comparison} is therefore uniform in \(r\), \(k\),
and \(i\).  Apply that lemma with scale \(R\),
\(\widetilde z=z_r^c\), and the interior scale \(\rho\).  Since
\(\rho<\rho_1\), it gives
\begin{equation}\label{eq:common-normalized-comparison}
 \frac{w(Z)}{v_i(Z)}
 \simeq
 \frac{w(A_r^{\mathrm{ref}})}
      {v_i(A_r^{\mathrm{ref}})}
 \qquad
 \text{for every }Z\in\mathcal Q_r^{\mathrm{red}}.
\end{equation}
By \eqref{eq:common-comparison-points}, both \(A_k\) and \(A_r^*\)
belong to this same reduced comparison region.  Taking successively
\(Z=A_k\) and \(Z=A_r^*\) in
\eqref{eq:common-normalized-comparison} yields
\begin{equation}\label{eq:two-comparison-one-factor}
 \frac{w(A_k)}{v_i(A_k)}
 \lesssim
 \frac{w(A_r^{\mathrm{ref}})}
      {v_i(A_r^{\mathrm{ref}})}
 \lesssim
 \frac{w(A_r^*)}{v_i(A_r^*)},
 \qquad i=1,2.
\end{equation}
Multiplying the two inequalities in
\eqref{eq:two-comparison-one-factor} gives
\begin{equation}\label{eq:two-comparison-pointwise}
 w(A_k)^2
 \lesssim
 w(A_r^*)^2
 \frac{v_1(A_k)v_2(A_k)}
      {v_1(A_r^*)v_2(A_r^*)}.
\end{equation}

It remains to estimate the denominator.  For \(i=1,2\), define
\[
 U_i(Q)=G_{\cH}^{\Omega_i}(Q;A_r^*).
\]
As a function of \(Q\), the function \(U_i\) is a forward solution away
from its second-variable pole \(A_r^*\).  The interior Green-function
lower estimate in the scale-\(r\) cylinder containing \(Q_i\) and
\(A_r^*\) gives
\[
 U_i(Q_i)\gtrsim r^{-(n+1)}.
\]
The Harnack chain from \(Q_i\) to \(P_i\) remains separated from
\(A_r^*\).  Forward Harnack along this chain therefore gives
\[
 U_i(Q_i)\lesssim U_i(P_i).
\]
Since
\[
 U_i(P_i)
 =G_{\cH}^{\Omega_i}(P_i;A_r^*)
 =v_i(A_r^*),
\]
we conclude that
\begin{equation}\label{eq:reference-green-lower}
 v_i(A_r^*)\gtrsim r^{-(n+1)},
 \qquad i=1,2.
\end{equation}
Substituting \eqref{eq:reference-green-lower} into
\eqref{eq:two-comparison-pointwise} gives
\[
 w(A_k)^2
 \lesssim
 r^{2(n+1)}w(A_r^*)^2v_1(A_k)v_2(A_k).
\]
Summing over \(k\) proves \eqref{eq:correct-comparison}.
\end{proof}

\begin{rem}
The implicit constant in \eqref{eq:correct-comparison} depends only on
\(n,\mu,\Lambda\), the fixed pole-admissibility parameters, and the fixed
comparison geometry.  In particular, it is independent of \(r\) and
\(k\).  The role of
Lemma~\ref{lem:green-boundary-time-comparison} is to control uniformly
the reference balances
\(\mathcal B_{\rho_0}^*(w,v_i)\), \(i=1,2\), occurring in normalized
boundary comparison.  Such uniform control need not hold for arbitrary
positive adjoint solutions.  It holds here because \(w\), \(v_1\), and
\(v_2\) are Green functions whose singularities are quantitatively
separated from the comparison region in the proper time direction.
Only the lower bounds in \eqref{eq:reference-green-lower} are used in
the proof.  The corresponding upper bounds follow directly from
Lemma~\ref{lem:green-upper}.
\end{rem}

Using \eqref{eq:parabolic-v1-omega},
\eqref{eq:parabolic-v2-boundary}, and the bounded overlap of the cubes
\(\Delta_k\), we obtain
\begin{align}
 \sum_k v_1(A_k)v_2(A_k)
 \lesssim
 \sum_k v_2(A_k)
 \omega_{\cH,\Omega_1}^{P_1}(\Delta_k)
 &\lesssim
 \sum_k
 \iint_{\Delta_k}v_2(1,y,s)\,
 \mathrm d\omega_{\cH,\Omega_1}^{P_1}(y,s)\notag\\
 &\lesssim
 \iint_{C\Delta_r^-(z_0)}v_2(1,y,s)\,
 \mathrm d\omega_{\cH,\Omega_1}^{P_1}(y,s).
 \label{eq:correct-sum-to-integral}
\end{align}
Here \(C\Delta_r^-(z_0)\) is a fixed enlargement containing all the
cubes occurring in the covering.  It therefore remains to prove
\begin{equation}\label{eq:forward-adjoint-defect}
 \iint_{C\Delta_r^-(z_0)}v_2(1,y,s)\,
 \mathrm d\omega_{\cH,\Omega_1}^{P_1}(y,s)
 \lesssim r^{-(n+2)},
 \qquad r\geq20.
\end{equation}
The scale-\(r\) Green-function estimate alone gives only $v_2(1,y,s)
 \lesssim r^{-(n+1)}$ for $(y,s)\in C\Delta_r^-(z_0)$,
which is one power of \(r\) too weak.  The key point is therefore to
gain an additional factor \(r^{-1}\) from the vanishing of \(v_2\) on
the flat boundary.  In Section~\ref{sec:height-coordinate}, we construct
a positive adjoint height coordinate and use boundary comparison to
prove the stronger pointwise estimate
\eqref{eq:unit-height-decay}.  The comparison supplies precisely this
additional factor, since the height coordinate is of order one at
height \(1\) and of order \(r\) at height \(r\).  Since the restriction
of \(\omega_{\cH,\Omega_1}^{P_1}\) to
\(C\Delta_r^-(z_0)\) has mass at most one,
\eqref{eq:unit-height-decay} implies
\eqref{eq:forward-adjoint-defect}.

\begin{rem}\label{rem:normal-periodicity-obstruction}
For genuinely time-dependent coefficients, transverse periodicity alone
does not justify a direct parabolic analogue of Dahlberg's representation
argument.  To see the
obstruction, put \(e_\lambda=(1,0,\ldots,0)\) and
\[
 F=\delta_\lambda v_2,
 \qquad
 F(Z)=v_2(Z+e_\lambda)-v_2(Z).
\]
The two terms have singularities at $\widetilde P_2:=P_2-e_\lambda$ and $P_2$, respectively.  Since \(P_2\) is spatially separated from
\(\overline{\Omega_1}\) by a distance comparable to \(r\) and
\(r\geq20\), both singularities lie outside
\(\overline{\Omega_1}\).  Consequently, on every subregion of
\(\Omega_1\) on which \(Z+e_\lambda\in\Omega_2\), transverse
periodicity gives \(\cH^*F=0\).  On the relevant flat boundary patch, \(\operatorname{Tr}_0F
 =F(0,\cdot)
 =v_2(1,\cdot)\), since \(v_2\) has zero trace on
\(\{\lambda=0\}\).  On the other hand,
\(\omega_{\cH,\Omega_1}^{P_1}\) represents solutions of the forward
equation \(\cH U=0\) from their data on the incoming parabolic boundary.
Consequently, even if the geometry were enlarged so that \(F\) were
defined and solved the adjoint
equation throughout \(\Omega_1\), one could not
assert
\begin{equation}\label{eq:false-representation}
 F(P_1)
 =
 \int_{\partial_p^{\mathrm{in}}\Omega_1}
 \operatorname{Tr}F\,
 \mathrm d\omega_{\cH,\Omega_1}^{P_1}.
\end{equation}
The right-hand side of \eqref{eq:false-representation} is a forward
representation formula, whereas \(F\) solves the adjoint equation.  An
adjoint representation, when applicable, uses adjoint parabolic measure
and the outgoing parabolic boundary.  Pointwise symmetry
\(A=A^{\mathsf T}\) does not by itself remove this distinction when
\(A\) depends on time, because the time derivatives in \(\cH\) and
\(\cH^*\) have opposite signs.  Green-function reciprocity and
reverse-time Harnack inequalities likewise do not turn an adjoint
solution into a forward solution.  The following scaling computation is
included only as heuristic motivation.  A scale-adapted trace-duality
argument makes the resulting quantitative loss explicit.  Let
\[
 \mathcal R_{2r}\subset\mathcal R_{3r}\subset\mathcal R_{4r}
\]
be fixed nested, pole-free boundary cylinders at scale \(r\), with flat
boundary portions containing the required enlargements of
\(C\Delta_r^-(z_0)\).  After localization at scale \(r\), the local
\(\mathrm{RH}_2\) estimate and the conormal-trace bound for \(v_1\),
together with the parabolic Dirichlet trace estimate for \(F\), would
imply \eqref{eq:forward-adjoint-defect} if one could prove
\begin{equation}\label{eq:missing-difference-energy}
 \|\nabla_{\lambda,x}F\|_{\mathrm L^2(\mathcal R_{2r})}
 +r^{-1}\|F\|_{\mathrm L^2(\mathcal R_{2r})}
 \lesssim r^{-(n+3)/2}.
\end{equation}
Indeed, the corresponding conormal and Dirichlet trace bounds would have
sizes \(r^{-(n+1)/2}\) and \(r^{-(n+3)/2}\), respectively, whose product
is \(r^{-(n+2)}\).  The zeroth-order term in \eqref{eq:missing-difference-energy} has the
required size.  The fundamental theorem of calculus in the transverse
variable, boundary Caccioppoli for \(v_2\), and the scale-\(r\)
Green-function upper bound give
\begin{align*}
 r^{-1}\|F\|_{\mathrm L^2(\mathcal R_{2r})}
 &\lesssim
 r^{-1}\|\partial_\lambda v_2\|_{\mathrm L^2(\mathcal R_{3r})}\lesssim
 r^{-2}\|v_2\|_{\mathrm L^2(\mathcal R_{4r})}
 \lesssim r^{-(n+3)/2}.
\end{align*}
By contrast, the estimates available from the transverse-difference
argument give only
\[
 \|\nabla_{\lambda,x}F\|_{\mathrm L^2(\mathcal R_{2r})}
 \lesssim r^{-(n+1)/2}.
\]
The boundary Caccioppoli estimate that would supply the additional
factor \(r^{-1}\) requires zero boundary trace and therefore cannot be
applied to \(F\), since
\[
 \operatorname{Tr}_0F=v_2(1,\cdot)\not\equiv0.
\]
Using the available gradient estimate in the trace-duality argument
gives only \(r^{-(n+1)}\), rather than the required \(r^{-(n+2)}\).  Thus this particular trace-duality approach remains one large-scale
power short under transverse periodicity alone.  This observation does
not exclude a different argument under that hypothesis.  The adjoint
height coordinate constructed in
Section~\ref{sec:height-coordinate}, using periodicity in all spatial
variables but no periodicity in time, proves
\eqref{eq:forward-adjoint-defect} directly and bypasses, rather than
establishes, \eqref{eq:missing-difference-energy}.
\end{rem}

\begin{rem}
A version of Lemma~\ref{lem:difference} appears in
\cite[Lemma~2.11]{CastroStromqvist}, which we cite only for this lemma.
Its subsequent large-scale argument uses time independence and symmetry
to rearrange the Green-function factors so that one has the form
required by forward parabolic measure while the other remains in the
forward Green variable.  This permits the transverse difference to be
represented by forward parabolic measure, but the mechanism is
unavailable for time-dependent coefficients and does not resolve the
forward-adjoint mismatch considered here.  Moreover, in the notation of their paper, the proof of
Theorem~2.12 applies Lemmas~2.4 and~2.5 to the solution \(u\) introduced
in equation~\textup{(2.25)} without first reducing to \(u\geq0\),
although both lemmas require nonnegative solutions.  We therefore do
not use that large-scale argument.  As noted in the introduction, the
corresponding conclusion follows instead by combining the elliptic
theorem of Kenig and Shen~\cite{KenigShen} with the structural theorem
of Litsg\aa rd and the author~\cite{LitsgardNystrom}.
\end{rem}

\section{The spatially periodic adjoint height coordinate}\label{sec:height-coordinate}

The forward-adjoint distinction prevents one from representing the
boundary datum in \eqref{eq:forward-adjoint-defect} by forward parabolic
measure.  Periodicity in all the spatial variables
\(Y=(\lambda,x)\) instead provides a positive adjoint comparison
function.  No periodicity in the time variable is needed in this
construction.

We write \(Y=(\lambda,x)\in\mathbb R^{n+1}\),
\(e_\lambda=(1,0,\ldots,0)\), and
\[
 \cH^*=-\partial_t-\operatorname{div}_{Y}
       (A^{\mathsf T}(Y,t)\nabla_Y).
\]
We identify the torus \(\mathbb T_Y^{n+1}
 =
 \mathbb R^{n+1}/\mathbb Z^{n+1}\) with the unit cell \([0,1)^{n+1}\),
whose opposite faces are identified.
It is a compact connected manifold without boundary.  Consequently,
periodic integration by parts produces no spatial boundary terms, and
Poincar\'e's inequality holds on the subspace of functions with zero
spatial mean.  Only the spatial factor is compact.  The product
\(\mathbb T_Y^{n+1}\times\mathbb R_t\) remains unbounded in time.

Let
\[
 \mathrm H^1_{\mathrm{per}}(\mathbb T_Y^{n+1})
 =
 \biggl\{
 f\in\mathrm H^1_{\mathrm{loc}}(\mathbb R^{n+1}):
 f(Y+k)=f(Y)\ \forall\, k\in\mathbb Z^{n+1},
 \int_{[0,1)^{n+1}}
 \bigl(|f|^2+|\nabla_Yf|^2\bigr)\,\mathrm dY<\infty
 \biggr\},
\]
where equality under translations is understood almost everywhere.  The
space \(\mathrm H^{-1}_{\mathrm{per}}(\mathbb T_Y^{n+1})\) denotes the dual
of \(\mathrm H^1_{\mathrm{per}}(\mathbb T_Y^{n+1})\).  Also,
\[
 \mathrm L^2_0(\mathbb T_Y^{n+1})
 =
 \biggl\{
 f\in\mathrm L^2(\mathbb T_Y^{n+1}):
 \int_{\mathbb T^{n+1}}f(Y)\,\mathrm dY=0
 \biggr\}.
\]
In what follows, a complete solution means a solution defined for every
\(t\in\mathbb R\).

\begin{lem}\label{lem:adjoint-periodic-corrector}
Assume \eqref{eq:ellipticity} and \eqref{eq:full-periodicity}, where
\eqref{eq:full-periodicity} concerns the spatial variables
\(Y=(\lambda,x)\) only.  Then there exists a unique bounded complete
weak solution
\[
 \chi^*\in
 \mathrm L^2_{\mathrm{loc}}
 \bigl(\mathbb R_t;
 \mathrm H^1_{\mathrm{per}}(\mathbb T_Y^{n+1})\bigr)
 \cap
 \mathrm C
 \bigl(\mathbb R_t;
 \mathrm L^2_0(\mathbb T_Y^{n+1})\bigr),\quad\mbox{with}\quad\partial_t\chi^*\in
 \mathrm L^2_{\mathrm{loc}}
 \bigl(\mathbb R_t;
 \mathrm H^{-1}_{\mathrm{per}}(\mathbb T_Y^{n+1})\bigr),
\]
which satisfies
\begin{equation}\label{eq:adjoint-cell}
 -\partial_t\chi^*
 -\operatorname{div}_Y
 \left(
 A^{\mathsf T}(Y,t)
 \bigl(e_\lambda+\nabla_Y\chi^*\bigr)
 \right)
 =0
 \quad\text{on }\mathbb T_Y^{n+1}\times\mathbb R_t
\end{equation}
and is normalized by
\begin{equation}\label{eq:corrector-normalization}
 \int_{\mathbb T^{n+1}}
 \chi^*(Y,t)\,\mathrm dY
 =0
 \quad\text{for every }t\in\mathbb R.
\end{equation}
Uniqueness is understood in the class of complete weak
solutions satisfying \eqref{eq:corrector-normalization} and
\begin{equation}\label{eq:corrector-energy}
 \sup_{t\in\mathbb R}
 \|\chi^*(\cdot,t)\|_{\mathrm L^2(\mathbb T^{n+1})}
 +
 \sup_{a\in\mathbb R}
 \left(
 \int_a^{a+1}
 \left(
 \|\nabla_Y\chi^*(\cdot,t)\|_2^2
 +
 \|\partial_t\chi^*(\cdot,t)\|_{\mathrm H^{-1}_{\mathrm{per}}}^2
 \right)
 \,\mathrm d t
 \right)^{1/2}
 \leq C.
\end{equation}
Moreover,
\begin{equation}\label{eq:corrector-infty}
 \|\chi^*\|_{\mathrm L^\infty(
 \mathbb T_Y^{n+1}\times\mathbb R_t)}
 \leq M_\chi,
\end{equation}
where \(C\) and \(M_\chi\geq1\) depend only on
\(n,\mu,\Lambda\).  Identifying \(\chi^*\) with its
\(\mathbb Z^{n+1}\)-periodic lift to
\(\mathbb R_Y^{n+1}\times\mathbb R_t\), define
\begin{equation}\label{eq:affine-corrector}
 p^*(\lambda,x,t)
 =\lambda+\chi^*(\lambda,x,t).
\end{equation}
Then
\begin{equation}\label{eq:pstar-properties}
 \cH^*p^*=0
 \quad\text{in }\mathbb R^{n+2},
 \qquad
 p^*(\lambda+1,x,t)
 =p^*(\lambda,x,t)+1.
\end{equation}
Moreover, \(p^*\) is one-periodic in every component of \(x\).
\end{lem}

\begin{proof}
The weak formulation of \eqref{eq:adjoint-cell} is
\begin{equation}\label{eq:adjoint-cell-weak}
 \iint_{\mathbb T_Y^{n+1}\times\mathbb R}
 \left[
 \chi^*\partial_t\varphi
 +
 A^{\mathsf T}
 \bigl(e_\lambda+\nabla_Y\chi^*\bigr)
 \cdot\nabla_Y\varphi
 \right]
 \,\mathrm dY\,\mathrm dt
 =0
\end{equation}
for every smooth test function \(\varphi\) which is periodic in \(Y\)
and compactly supported in time.  If \(\chi^*\) is such a weak solution
and
\[
 m(t)
 =
 \int_{\mathbb T^{n+1}}\chi^*(Y,t)\,\mathrm dY,
\]
then testing with functions depending only on \(t\) gives \(m'=0\) in
\(\mathcal D'(\mathbb R)\).  Thus \(m\) is constant almost everywhere.
Since adding a constant does not change the equation, we seek a
solution in the mean-zero class.

Reverse time by setting
\[
 \tau=-t,
 \qquad
 \widetilde A(Y,\tau)=A^{\mathsf T}(Y,-\tau),
 \qquad
 w(Y,\tau)=\chi^*(Y,-\tau).
\]
The function \(w\) must solve
\begin{equation}\label{eq:forward-cell-reversed}
 \partial_\tau w
 -\operatorname{div}_Y
 \bigl(\widetilde A\nabla_Yw\bigr)
 =
 \operatorname{div}_Y
 \bigl(\widetilde A e_\lambda\bigr)
 \quad\text{on }\mathbb T_Y^{n+1}\times\mathbb R_\tau.
\end{equation}
Fix \(T>0\) and consider the Gelfand triple
\[
 V\hookrightarrow H\cong H'\hookrightarrow V',
 \qquad
 H=\mathrm L^2_0(\mathbb T_Y^{n+1}),
 \qquad
 V=\mathrm H^1_{\mathrm{per}}(\mathbb T_Y^{n+1})\cap H.
\]
Boundedness and ellipticity of \(\widetilde A\), together with the
periodic Poincar\'e inequality on \(V\), show that the associated
bilinear forms are measurable in time, uniformly bounded, and
coercive on \(V\).  Moreover,
\(\operatorname{div}_Y(\widetilde A e_\lambda)\) belongs uniformly to
\(V'\).  Hence, for each \(\widehat T>-T\), the standard variational
theory for nonautonomous coercive forms gives a unique energy solution
\(w_{T,\widehat T}\) of \eqref{eq:forward-cell-reversed} on
\((-T,\widehat T)\), with \(w_{T,\widehat T}(\cdot,-T)=0\).
Uniqueness implies that these solutions agree on overlapping time
intervals.  They therefore determine a unique energy solution \(w_T\)
on \((-T,\infty)\), with \(w_T(\cdot,-T)=0\).  Since
\(w_T\in\mathrm C([-T,\infty);H)\), its spatial mean is zero at every
time.  The solution is defined on the spatial torus, so its lift to
\(\mathbb R_Y^{n+1}\) is periodic in \(Y\).

The variational construction initially gives
\[
 \iint_{\mathbb T_Y^{n+1}\times(-T,\infty)}
 \left[
 -w_T\partial_\tau\varphi_0
 +
 \widetilde A
 \bigl(\nabla_Yw_T+e_\lambda\bigr)
 \cdot\nabla_Y\varphi_0
 \right]
 \,\mathrm dY\,\mathrm d\tau
 =0
\]
for every smooth test function \(\varphi_0\) which has zero spatial
mean, is periodic in \(Y\), and is compactly supported in time.  To
recover the full weak formulation, let \(\varphi\) be any smooth
function which is periodic in \(Y\) and compactly supported in
\((-T,\infty)\), and set
\[
 \overline\varphi(\tau)
 =
 \int_{\mathbb T^{n+1}}\varphi(Y,\tau)\,\mathrm dY,
 \qquad
 \varphi_0(Y,\tau)
 =
 \varphi(Y,\tau)-\overline\varphi(\tau).
\]
Then \(\varphi_0(\cdot,\tau)\in V\) and
\(\nabla_Y\varphi_0=\nabla_Y\varphi\).  Moreover, since \(w_T\) has
zero spatial mean,
\[
 \iint_{\mathbb T_Y^{n+1}\times(-T,\infty)}
 w_T\,\partial_\tau
 \bigl(\varphi-\varphi_0\bigr)
 \,\mathrm dY\,\mathrm d\tau
 =
 \int_{-T}^{\infty}
 \overline\varphi'(\tau)
 \left(
 \int_{\mathbb T^{n+1}}w_T(Y,\tau)\,\mathrm dY
 \right)
 \,\mathrm d\tau
 =0.
\]
Thus the variational identity for \(\varphi_0\) is identical to the
one for \(\varphi\).  Consequently, \(w_T\) satisfies
\eqref{eq:forward-cell-reversed} against every smooth periodic test
function compactly supported in time.

 Testing the equation with \(w_T\), and then using
ellipticity, Young's inequality, and the Poincar\'e inequality for
mean-zero functions on the unit torus,
\[
 \|w_T(\cdot,\tau)\|_{\mathrm L^2(\mathbb T^{n+1})}
 \leq C_{\mathrm P}
 \|\nabla_Yw_T(\cdot,\tau)\|_{\mathrm L^2(\mathbb T^{n+1})},
\]
gives \begin{equation}\label{eq:corrector-dissipation}
 \frac{\mathrm d}{\mathrm d \tau}
 \|w_T(\cdot,\tau)\|_2^2
 +c_0\|w_T(\cdot,\tau)\|_2^2
 +c_0\|\nabla_Yw_T(\cdot,\tau)\|_2^2
 \leq C_0
\end{equation}
for almost every \(\tau>-T\).  The constants \(c_0>0\) and
\(C_0<\infty\) depend only on \(n,\mu,\Lambda\).  It follows that
\begin{equation}\label{eq:corrector-pullback-bounds}
 \sup_{T>0}\sup_{\tau>-T}
 \|w_T(\cdot,\tau)\|_2
 +
 \sup_{T>0}\sup_{a>-T}
 \left(
 \int_a^{a+1}
 \|\nabla_Yw_T(\cdot,\tau)\|_2^2
 \,\mathrm d \tau
 \right)^{1/2}
 \leq C.
\end{equation}

Let \(S>T\).  On \((-T,\infty)\), the difference
\(z_{S,T}=w_S-w_T\) solves the homogeneous equation and has zero spatial
mean.  Its energy inequality and Poincar\'e's inequality give exponential
decay from time \(-T\).  Since \(w_T(\cdot,-T)=0\) and
\(\|w_S(\cdot,-T)\|_2\leq C\) by
\eqref{eq:corrector-pullback-bounds}, it follows that
\begin{equation}\label{eq:corrector-pullback-contraction}
 \|w_S(\cdot,\tau)-w_T(\cdot,\tau)\|_2
 \leq
 C\exp\bigl(-c(\tau+T)\bigr),
 \qquad \tau\geq-T.
\end{equation}
For every compact interval \(I=[a,b]\), the homogeneous energy
inequality also gives, once \(T>-a\),
\[
 \mu\int_a^b
 \|\nabla_Y(w_S-w_T)(\cdot,\tau)\|_2^2\,\mathrm d\tau
 \leq
 \frac12\|(w_S-w_T)(\cdot,a)\|_2^2
 \longrightarrow0
\]
as \(S>T\to\infty\), by
\eqref{eq:corrector-pullback-contraction}.  Hence the whole family is
Cauchy in the local energy space, not merely along a subsequence.
Consequently, \(w_T\) converges as \(T\to\infty\), locally in time in
\(\mathrm C(\mathbb R;\mathrm L^2_0)\) and strongly in
\(\mathrm L^2_{\mathrm{loc}}(\mathbb R;\mathrm H^1_{\mathrm{per}})\),
to a complete solution \(w\) of
\eqref{eq:forward-cell-reversed}.  The equation gives
\[
 \|\partial_\tau w\|_{\mathrm H^{-1}_{\mathrm{per}}}
 \leq
 C\bigl(1+\|\nabla_Yw\|_2\bigr).
\]
Thus \(w\) satisfies the time-uniform estimates corresponding to
\eqref{eq:corrector-energy}.

Uniqueness holds in the class of complete mean-zero solutions satisfying
these time-uniform energy bounds.  Indeed, if \(w_1\) and \(w_2\) are two
such solutions, then, for \(s<\tau\),
\[
 \|w_1(\cdot,\tau)-w_2(\cdot,\tau)\|_2
 \leq
 e^{-c(\tau-s)}
 \|w_1(\cdot,s)-w_2(\cdot,s)\|_2.
\]
Since
\[
 \sup_{s\in\mathbb R}
 \|w_1(\cdot,s)-w_2(\cdot,s)\|_2<\infty,
\]
letting \(s\to-\infty\) shows that
\(w_1(\cdot,\tau)=w_2(\cdot,\tau)\).  Hence \(w_1=w_2\).
Reversing time now gives \(\chi^*\) and proves
\eqref{eq:corrector-energy}.

Identify \(A\) and \(\chi^*\) with their \(\mathbb Z^{n+1}\)-periodic
lifts in \(Y\) to \(\mathbb R_Y^{n+1}\times\mathbb R_t\), and define \(p^*\)
by \eqref{eq:affine-corrector}.  Since
\[
 \partial_tp^*=\partial_t\chi^*,
 \qquad
 \nabla_Yp^*=e_\lambda+\nabla_Y\chi^*,
\]
equation \eqref{eq:adjoint-cell} gives \(\cH^*p^*=0\).
Periodicity of \(\chi^*\) in \(\lambda\) gives the affine periodicity
in \eqref{eq:pstar-properties}.  Moreover, it follows that \(p^*\) is one-periodic in every component of \(x\).

It remains to obtain the pointwise bound in \eqref{eq:corrector-infty}.  For
\(j,\ell\in\mathbb Z\), let
\[
 \mathcal Q_{j,\ell}
 =
 (j,j+1)\times(0,1)^n\times(\ell,\ell+1),
\]
and let \(\widetilde{\mathcal Q}_{j,\ell}\) be a fixed enlargement.
The function \(p_j^*=p^*-j\) is a homogeneous adjoint solution.  On
\(\widetilde{\mathcal Q}_{j,\ell}\), one has
\[
 p_j^*=(\lambda-j)+\chi^*.
\]
The first term is uniformly bounded, and
\eqref{eq:corrector-energy} controls the space-time
\(\mathrm L^2\)-norm of the second term.  Hence
\begin{equation}\label{eq:pstar-cell-L2}
 \iiint_{\widetilde{\mathcal Q}_{j,\ell}}
 |p_j^*|^2
 \,\mathrm d \lambda\,\mathrm d x\,\mathrm d t
 \leq C,
\end{equation}
where \(C\) is independent of \(j\) and \(\ell\).  The adjoint form of the local boundedness estimate in
Lemma~\ref{lem:local-parabolic-estimates}, applied to a fixed finite
family of future-oriented parabolic cylinders covering
\(\mathcal Q_{j,\ell}\) and contained in
\(\widetilde{\mathcal Q}_{j,\ell}\), together with
\eqref{eq:pstar-cell-L2}, gives
\begin{equation}\label{eq:pstar-cell-Linfty}
 \|p_j^*\|_{\mathrm L^\infty(\mathcal Q_{j,\ell})}
 \leq C.
\end{equation}
It follows that \(|\chi^*|\leq C+1\) on every such cell.  This proves
\eqref{eq:corrector-infty}.

The use of \(p^*-j\) is important.  It is a homogeneous solution,
whereas the equation for \(\chi^*\) contains the distributional forcing
\(\operatorname{div}_Y(A^{\mathsf T}e_\lambda)\).  The local boundedness
argument therefore does not differentiate the merely measurable
coefficient matrix.
\end{proof}

The preceding lemma also identifies the expected behavior of the
half-space comparison function.  Indeed, at scales large compared with
the spatial period, periodic homogenization suggests that a normalized
positive adjoint solution vanishing on the flat boundary should behave
like the Euclidean height function \(\lambda\).  This behavior is
already encoded by
\[
p^*=\lambda+\chi^*,
\]
since \(\chi^*\) is bounded.  However, \(p^*\) need neither vanish on
\(\{\lambda=0\}\) nor be positive throughout the half-space.  The next
lemma modifies \(p^*\) so as to obtain these two properties while
preserving a uniformly bounded additive error from \(\lambda\).

\begin{lem}\label{lem:height-coordinate}
Assume \eqref{eq:ellipticity} and \eqref{eq:full-periodicity}, with no
periodicity assumption in the time variable.  Then there exists a weak
solution \(\Phi^*\) satisfying
\begin{equation}\label{eq:height-equation}
 \cH^*\Phi^*=0
 \quad\text{in }\mathbb R^{n+2}_+,
 \qquad
 \Phi^*=0
 \quad\text{on }\{\lambda=0\}.
\end{equation}
The function \(\Phi^*\) is one-periodic in each component of \(x\), is
strictly positive in \(\mathbb R^{n+2}_+\), and satisfies
\begin{equation}\label{eq:height-bounded-error}
 \bigl|\Phi^*(\lambda,x,t)-\lambda\bigr|
 \leq C_\Phi.
\end{equation}
Moreover, there exist constants \(c_1>0\) and \(C_1<\infty\) such
that
\begin{equation}\label{eq:height-comparison}
 c_1\lambda
 \leq
 \Phi^*(\lambda,x,t)
 \leq
 C_1\lambda,
 \qquad \lambda\geq1.
\end{equation}
The constants \(C_\Phi,c_1\), and \(C_1\) depend only on
\(n,\mu,\Lambda\).
\end{lem}

\begin{proof}
Let \(\chi^*\), \(p^*\), and \(M_\chi\) be as in
Lemma~\ref{lem:adjoint-periodic-corrector}.  We divide the proof into three steps.

\smallskip
\noindent
\emph{Step 1. Complete solutions in finite spatial slabs.}
For every positive integer \(N\), set
\[
 D_N=(0,N)\times\mathbb T_x^n.
\]
We first construct a complete solution of
\begin{equation}\label{eq:slab-height}
 \begin{cases}
  \cH^*\Phi_N^*=0
  &\text{in }D_N\times\mathbb R_t,\\
  \Phi_N^*=0
  &\text{on }\{\lambda=0\}\times\mathbb R_t,\\
  \Phi_N^*=N
  &\text{on }\{\lambda=N\}\times\mathbb R_t,
 \end{cases}
\end{equation}
where all functions are periodic in \(x\).  Recall that a complete solution
means a solution defined for every
\(t\in\mathbb R\).

As in the proof of Lemma~\ref{lem:adjoint-periodic-corrector}, reverse time and set
\[
 \widetilde A(Y,\tau)=A^{\mathsf T}(Y,-\tau),
 \qquad
 q(Y,\tau)=p^*(Y,-\tau).
\]
Then
\[
 \partial_\tau q
 -\operatorname{div}_Y
 \bigl(\widetilde A\nabla_Yq\bigr)=0
 \quad\text{in }\mathbb R_Y^{n+1}\times\mathbb R_\tau.
\]
For \(T>0\), write
\[
 \Psi_{N,T}=\lambda+u_{N,T}.
\]
We construct the correction \(u_{N,T}\) to have zero trace on
\(\{\lambda=0\}\cup\{\lambda=N\}\) and to satisfy
\[
 \partial_\tau u_{N,T}
 -\operatorname{div}_Y
 \bigl(\widetilde A\nabla_Yu_{N,T}\bigr)
 =
 \operatorname{div}_Y
 \bigl(\widetilde A e_\lambda\bigr)
 \quad\text{in }D_N\times(-T,\infty),
\]
with initial value \(u_{N,T}(\cdot,-T)=0\).  Since the right-hand side belongs to the dual of the Dirichlet energy
space, the standard nonautonomous variational theory again gives a unique
energy solution \(u_{N,T}\), first on every finite time interval and
hence on \((-T,\infty)\).  Indeed, let
\[
 V_N=
 \bigl\{
 v\in\mathrm H^1\bigl((0,N)\times\mathbb T_x^n\bigr):
 \operatorname{Tr}_{\lambda=0}v
 =
 \operatorname{Tr}_{\lambda=N}v
 =0
 \bigr\}.
\]
For almost every \(\tau\), the form
\[
 a_\tau(v,\varphi)
 =
 \int_{D_N}
 \widetilde A(Y,\tau)\nabla_Yv\cdot\nabla_Y\varphi\,\mathrm dY
\]
is measurable in \(\tau\), bounded, and coercive on \(V_N\), while
\[
 F_\tau(\varphi)
 =
 -\int_{D_N}
 \widetilde A(Y,\tau)e_\lambda\cdot\nabla_Y\varphi\,\mathrm dY
\]
belongs to \(V_N'\), uniformly in \(\tau\).  The nonautonomous
variational theorem therefore gives a unique energy solution on every
finite interval \((-T,\hat T)\).  Uniqueness on overlapping intervals yields
the asserted solution on \((-T,\infty)\).

Having constructed \(u_{N,T}\), \(\Psi_{N,T}\) is the unique
energy solution of
\[
 \partial_\tau\Psi_{N,T}
 -\operatorname{div}_Y
 \bigl(\widetilde A\nabla_Y\Psi_{N,T}\bigr)=0
 \quad\text{in }D_N\times(-T,\infty),
\]
with boundary values \(0\) at \(\lambda=0\), \(N\) at \(\lambda=N\),
and initial value \(\Psi_{N,T}(\cdot,-T)=\lambda\).  Since \(D_N=(0,N)\times\mathbb T_x^n\), its spatial boundary consists
only of the two faces \(\{\lambda=0\}\) and \(\{\lambda=N\}\) and there is
no boundary in the periodic \(x\)-variables.   The maximum principle gives
\begin{equation}\label{eq:slab-positive-approximation}
 0\leq\Psi_{N,T}\leq N.
\end{equation}

The functions \(q-M_\chi\) and \(q+M_\chi\) solve the same forward
homogeneous equation.  Since
\[
 |q(Y,\tau)-\lambda|\leq M_\chi,
\]
they bracket the prescribed datum \(\lambda\) at the initial time and
also bracket the boundary values \(0\) and \(N\) on the faces
\(\lambda=0\) and \(\lambda=N\), respectively.  The comparison
principle therefore yields
\begin{equation}\label{eq:slab-comparison-approximation}
 q-M_\chi
 \leq\Psi_{N,T}
 \leq q+M_\chi
 \quad\text{in }D_N\times(-T,\infty).
\end{equation}

For fixed \(N\), let \(S>T\).  The difference
\[
 z_{S,T}=\Psi_{N,S}-\Psi_{N,T}
\]
solves the homogeneous equation on \(D_N\times(-T,\infty)\) and has zero
trace on the two transverse faces.  Since \(z_{S,T}\) has zero trace on
\(\{\lambda=0\}\cup\{\lambda=N\}\), the Poincar\'e inequality in the
\(\lambda\)-variable gives
\[
 \|z_{S,T}(\cdot,\tau)\|_{\mathrm L^2(D_N)}
 \leq
 C N\|\partial_\lambda z_{S,T}(\cdot,\tau)\|_{\mathrm L^2(D_N)}
 \leq
 C N\|\nabla_Yz_{S,T}(\cdot,\tau)\|_{\mathrm L^2(D_N)}.
\]
Testing the homogeneous equation for \(z_{S,T}\) with \(z_{S,T}\) and
using ellipticity therefore yields
\[
 \frac{\mathrm d}{\mathrm d\tau}
 \|z_{S,T}(\cdot,\tau)\|_{\mathrm L^2(D_N)}^2
 +
 cN^{-2}
 \|z_{S,T}(\cdot,\tau)\|_{\mathrm L^2(D_N)}^2
 \leq0.
\]
Consequently, Gr\"onwall's inequality yields
\[
 \|z_{S,T}(\cdot,\tau)\|_{\mathrm L^2(D_N)}
 \leq
 e^{-c_N(\tau+T)}
 \|z_{S,T}(\cdot,-T)\|_{\mathrm L^2(D_N)},
 \qquad \tau\geq-T,
\]
where \(c_N>0\) depends only on \(N\) and the ellipticity constants.  The comparison estimate
\eqref{eq:slab-comparison-approximation} bounds the last norm by a
constant depending only on \(N\) and the structural parameters.
Consequently,
\[
 \|\Psi_{N,S}(\cdot,\tau)-\Psi_{N,T}(\cdot,\tau)\|_{\mathrm L^2(D_N)}
 \leq
 C_Ne^{-c_N(\tau+T)},
 \qquad \tau\geq-T.
\]
If \(I=[a,b]\) is compact and \(T>-a\), integration of the homogeneous
energy inequality gives
\[
 \mu\int_a^b
 \|\nabla_Y(\Psi_{N,S}-\Psi_{N,T})(\cdot,\tau)\|_2^2
 \,\mathrm d\tau
 \leq
 \frac12\|z_{S,T}(\cdot,a)\|_2^2
 \longrightarrow0
\]
as \(S>T\to\infty\).  Thus the whole family is Cauchy in the local
energy space and uniqueness follows by applying the same contraction to the
difference of two complete solutions with uniformly bounded
\(\mathrm L^2(D_N)\)-norm.
It follows that \(\Psi_{N,T}\) converges as \(T\to\infty\) to a
complete forward solution \(\Psi_N\).  This solution is unique among
complete solutions for which \(\Psi_N-\lambda\) is uniformly bounded
in \(\mathrm L^2(D_N)\).  Define
\[
 \Phi_N^*(Y,t)=\Psi_N(Y,-t).
\]
Then \(\Phi_N^*\) solves \eqref{eq:slab-height}.  Passing to the limit
in \eqref{eq:slab-positive-approximation} and
\eqref{eq:slab-comparison-approximation} gives
\begin{equation}\label{eq:slab-positive}
 0\leq\Phi_N^*\leq N
\end{equation}
and
\begin{equation}\label{eq:slab-comparison}
 p^*-M_\chi
 \leq\Phi_N^*
 \leq p^*+M_\chi.
\end{equation}
In particular,
\begin{equation}\label{eq:slab-linear}
 |\Phi_N^*(\lambda,x,t)-\lambda|
 \leq2M_\chi,
 \qquad 0<\lambda<N,
\end{equation}
uniformly in \(N\) and \(t\).

\smallskip
\noindent
\emph{Step 2. Passage to the half-space.}
For \(R,T>1\), set
\[
 D_R=(0,R)\times\mathbb T_x^n,
 \qquad
 I_T=(-T,T),
 \qquad
 \mathcal Q_{R,T}=D_R\times I_T.
\]
Choose a transverse cutoff
\(\eta_R\in C^\infty([0,\infty))\) satisfying
\[
 \eta_R=1\quad\text{on }[0,R],
 \qquad
 \eta_R=0\quad\text{on }[R+1,\infty),
 \qquad
 |\eta_R'|\leq C,
\]
and a temporal cutoff
\(\zeta_T\in C_0^\infty(I_{T+1})\) satisfying
\(\zeta_T=1\) on \(I_T\) and \(|\zeta_T'|\leq C\).  Testing the equation for \(\Phi_N^*\) with
\(\eta_R^2\zeta_T^2\Phi_N^*\), justified by Steklov averaging in time,
and then using ellipticity and Young's
inequality, gives
\[
 \int_{I_T}\int_{D_R}
 |\nabla_Y\Phi_N^*|^2\,\mathrm dY\,\mathrm d t
 \leq
 C\int_{I_{T+1}}\int_{D_{R+1}}
 |\Phi_N^*|^2\,\mathrm dY\,\mathrm d t
 \leq C_{R,T},
 \qquad N>R+1.
\]
The periodicity in \(x\) eliminates boundary terms in the tangential
variables, while the zero trace of \(\Phi_N^*\) on \(\{\lambda=0\}\)
makes the test function admissible at the flat boundary.  The last
inequality follows from \eqref{eq:slab-linear}, which also directly
controls the zeroth-order part of the \(\mathrm H^1(D_R)\)-norm.
Consequently,
\begin{equation}\label{eq:height-local-H1}
 \sup_{N>R+1}
 \|\Phi_N^*\|_{\mathrm L^2(I_T;\mathrm H^1(D_R))}
 \leq C_{R,T}.
\end{equation}

The equation gives
\[
 \partial_t\Phi_N^*
 =
 -\operatorname{div}_Y
 \bigl(A^{\mathsf T}\nabla_Y\Phi_N^*\bigr)
\]
in the sense of distributions.  Consequently,
\begin{equation}\label{eq:height-local-time}
 \sup_{N>R+1}
 \|\partial_t\Phi_N^*\|_{
 \mathrm L^2(I_T;\mathrm H^{-1}(D_R))}
 \leq C_{R,T}.
\end{equation}
Here \(\mathrm H^{-1}(D_R)\) denotes the dual of
\(\mathrm H_0^1(D_R)\), with periodicity in \(x\).  Since
\[
 \mathrm H^1(D_R)
 \Subset\mathrm L^2(D_R)
 \hookrightarrow\mathrm H^{-1}(D_R),
\]
the Aubin-Lions lemma and a diagonal extraction provide a sequence
\(N_j\to\infty\) and a function \(\Phi^*\) such that, for every
\(R,T>1\),
\begin{align}
 \Phi_{N_j}^*
 \rightharpoonup\Phi^*&\qquad\text{weakly in }
 \mathrm L^2(I_T;\mathrm H^1(D_R)),
 \label{eq:height-weak-convergence}\\
 \Phi_{N_j}^*
\longrightarrow\Phi^*
 &\qquad\text{strongly in }
 \mathrm L^2(\mathcal Q_{R,T}).
 \label{eq:height-strong-convergence}
\end{align}
We may also assume convergence almost everywhere.

The boundedness of the trace map on
\(\mathrm L^2(I_T;\mathrm H^1(D_R))\) and
\eqref{eq:height-weak-convergence} preserve the zero trace at
\(\lambda=0\).  Passing to the limit in the weak formulation shows
that
\[
 \cH^*\Phi^*=0
 \quad\text{on }(0,\infty)\times\mathbb T_x^n\times\mathbb R_t.
\]
The functions \(\Phi_N^*\) are periodic in \(x\), so the limit has the
same periodicity.  Extending in \(x\) and unfolding periodic test
functions proves \eqref{eq:height-equation} in
\(\mathbb R^{n+2}_+\).

The almost-everywhere convergence, \eqref{eq:slab-positive}, and
\eqref{eq:slab-linear} give
\begin{equation}\label{eq:height-limit-linear}
 \Phi^*\geq0,
 \qquad
 |\Phi^*(\lambda,x,t)-\lambda|
 \leq2M_\chi
\end{equation}
almost everywhere.  The interior and flat-boundary
De Giorgi-Nash-Moser estimates give a representative which is locally
H\"older continuous in the half-space and continuous up to the flat
boundary.  Hence the inequalities in \eqref{eq:height-limit-linear} then hold
pointwise.

Set \(C_\Phi=2M_\chi\).  The bounded-error estimate already gives
\(\Phi^*>0\) when \(\lambda>C_\Phi\).  Suppose that
\(\Phi^*(\lambda_0,x_0,t_0)=0\) for some \(\lambda_0>0\).  Set
\(\lambda_*=C_\Phi+1\).  A finite correctly oriented adjoint Harnack
chain contained in \(\{\lambda>0\}\) joins a point
\[
 (\lambda_*,x_0,t_0+\Theta)
\]
to \((\lambda_0,x_0,t_0)\), where
\(\Theta=\Theta(C_\Phi,\lambda_0)>0\).  Harnack's inequality gives
\[
 \Phi^*(\lambda_*,x_0,t_0+\Theta)
 \leq C\Phi^*(\lambda_0,x_0,t_0)=0.
\]
This contradicts
\[
 \Phi^*(\lambda_*,x_0,t_0+\Theta)
 \geq\lambda_*-C_\Phi=1.
\]
Thus \(\Phi^*>0\) throughout the half-space.

\smallskip
\noindent
\emph{Step 3. Quantitative comparison with height.}
Set \(H_0=2C_\Phi+2\).  If \(\lambda\geq H_0\), then
\[
 \frac12\lambda
 \leq\Phi^*(\lambda,x,t)
 \leq\frac32\lambda.
\]
It remains to prove the lower bound on
\(1\leq\lambda\leq H_0\).  For every \((x,t)\), the bounded-error estimate gives
\begin{equation}\label{eq:height-positive-reference}
 \Phi^*(H_0,x,t)
 \geq H_0-C_\Phi
 \geq\frac{H_0}{2}.
\end{equation}
Fix \(Z=(\lambda,x,t)\) in the bounded-height slab.  Choose
\(\Theta=\Theta(C_\Phi)>0\) so that a finite chain of overlapping
adjoint parabolic cylinders joins
\[
 Z_*=(H_0,x,t+\Theta)
\]
to \(Z\) inside
\[
 \biggl\{
 (\upsilon,y,s):\frac12<\upsilon<H_0+1
 \biggr\}.
\]
The time coordinates decrease along the chain, as required for
\(\cH^*\).  Let \(N_0\) be the number of cylinders and \(C_H\) the
corresponding Harnack constant.  The number \(N_0\) is controlled by
\(C_\Phi\), while \(C_H\) depends only on \(n,\mu,\Lambda\) and the
fixed chain geometry.
Iterating Harnack's inequality and using
\eqref{eq:height-positive-reference} at time \(t+\Theta\) gives
\[
 \frac{H_0}{2}
 \leq\Phi^*(Z_*)
 \leq C_H^{N_0}\Phi^*(Z).
\]
Consequently,
\[
 \Phi^*(\lambda,x,t)
 \geq m_0
 :=C_H^{-N_0}\frac{H_0}{2}>0
\]
on \(1\leq\lambda\leq H_0\).  Since
\(\lambda\leq H_0\) there,
\[
 \Phi^*(\lambda,x,t)
 \geq\frac{m_0}{H_0}\lambda.
\]
Together with the direct estimate for \(\lambda\geq H_0\), this
proves the lower inequality in \eqref{eq:height-comparison}.  Finally,
\[
 \Phi^*(\lambda,x,t)
 \leq\lambda+C_\Phi
 \leq(1+C_\Phi)\lambda,
 \qquad \lambda\geq1.
\]
This proves the upper inequality and completes the proof.
\end{proof}

We next establish the key estimate \eqref{eq:forward-adjoint-defect}.

\begin{lem}
\label{lem:unit-height-decay}
Assume the full spatial periodicity condition
\eqref{eq:full-periodicity} and retain the geometry fixed above.  Let
\(v_2\) be defined by \eqref{eq:two-adjoint-green}.  Then, for every
\(r\geq20\),
\begin{equation}\label{eq:unit-height-decay}
 \sup_{(y,s)\in C\Delta_r^-(z_0)}
 v_2(1,y,s)
 \lesssim r^{-(n+2)}.
\end{equation}
Here \(C\Delta_r^-(z_0)\) denotes the fixed enlargement occurring in
the preceding construction.  Consequently,
\eqref{eq:forward-adjoint-defect} holds.
\end{lem}

\begin{proof}
Fix \((y,s)\in C\Delta_r^-(z_0)\) and set
\[
 Z_{y,s}=(1,y,s).
\]
Let \(c_{\mathrm{bc},0}\) and \(c_{\mathrm{bc},1}\) be the constants in
Lemma~\ref{lem:boundary-comparison}.  Let \(L_h>2\) be the structural
constant fixed in the preceding construction, and put
\[
 \zeta_{y,s}=(y,s-\tfrac12r^2),
 \qquad
 A_{r,y,s}^*=(r,y,s-r^2),
\]
\[
 \rho_h=L_hc_{\mathrm{bc},1}r,
 \qquad
 R_h=2c_{\mathrm{bc},0}c_{\mathrm{bc},1}\rho_h,
 \qquad
 \mathcal Q_{r,y,s}=T_{2R_h}(\zeta_{y,s}).
\]
Both \(Z_{y,s}\) and \(A_{r,y,s}^*\) belong to
\[
 T_{\rho_h/c_{\mathrm{bc},1}}(\zeta_{y,s})
 =T_{L_hr}(\zeta_{y,s}).
\]
This is the reduced comparison region obtained by applying
Lemma~\ref{lem:boundary-comparison} with outer scale \(R_h\), center
\(\zeta_{y,s}\), and interior scale \(\rho_h\), because
\[
 \rho_h<\frac{R_h}{c_{\mathrm{bc},0}c_{\mathrm{bc},1}}.
\]
The reference geometry is therefore explicit, and it satisfies
\begin{equation}\label{eq:unit-height-reference-geometry}
 \lambda(A_{r,y,s}^*)=r,
 \qquad
 |X(P_2)-X(A_{r,y,s}^*)|\lesssim r,
 \qquad
 t(P_2)-t(A_{r,y,s}^*)\simeq r^2.
\end{equation}
Set
\[
 \rho_0=\frac{R_h}{c_{\mathrm{bc},0}}.
\]
The fixed constants defining \(\Omega_2\) and \(P_2\) are chosen after
\(L_h\) so that, uniformly for
\((y,s)\in C\Delta_r^-(z_0)\), the fixed enlargement
\(T_{4R_h}(\zeta_{y,s})\) is contained in \(\Omega_2\), lies in the
strict causal past of \(P_2\), and remains quantitatively separated from
the artificial spatial and temporal faces.  We also require
\[
 P_2\in
 \Omega_2\cap\mathcal P_{\kappa,\rho_0}^+(\zeta_{y,s})
\]
for one fixed \(\kappa\).  These requirements are compatible with the
preceding construction because all centers \(\zeta_{y,s}\) range over a
fixed scale-\(r\) enlargement of \((x_0,t_0)\), while every length and
time separation involved is a fixed multiple of \(r\) or \(r^2\).
In particular, the two balance corkscrews
\(A_{\rho_0}^-(\zeta_{y,s})\) and
\(A_{\rho_0}^+(\zeta_{y,s})\) belong to this pole-free enlargement.
Since \(r\geq20\), both \(\rho_0\) and
\(\lambda(A_{r,y,s}^*)\) are at least one.

It follows that \(v_2\) and \(\Phi^*\) are positive adjoint solutions
in \(\mathcal Q_{r,y,s}\), and both vanish continuously on its flat
boundary portion.  The reference balance in the adjoint form of
Lemma~\ref{lem:boundary-comparison} is uniformly controlled.  Indeed,
Lemma~\ref{lem:green-boundary-time-comparison} and
Lemma~\ref{lem:height-coordinate} give
\[
 \frac{v_2(A_{\rho_0}^-(\zeta_{y,s}))}
      {v_2(A_{\rho_0}^+(\zeta_{y,s}))}
 \simeq1,
 \qquad
 \frac{\Phi^*(A_{\rho_0}^-(\zeta_{y,s}))}
      {\Phi^*(A_{\rho_0}^+(\zeta_{y,s}))}
 \simeq1.
\]
For the second comparison, we have used
\(\Phi^*(A_{\rho_0}^\pm(\zeta_{y,s}))\simeq\rho_0\), which follows from
\eqref{eq:height-comparison}.

Let
\[
 A_h=A_{\rho_h}(\zeta_{y,s})
\]
be the normalizing corkscrew furnished by
Lemma~\ref{lem:boundary-comparison}.  Applying the adjoint form of that
lemma first at \(Z_{y,s}\) and then at \(A_{r,y,s}^*\), using the same
normalizing corkscrew, gives
\begin{equation}\label{eq:v2-height-comparison}
 \frac{v_2(Z_{y,s})}{\Phi^*(Z_{y,s})}
 \lesssim
 \frac{v_2(A_h)}{\Phi^*(A_h)}
 \lesssim
 \frac{v_2(A_{r,y,s}^*)}{\Phi^*(A_{r,y,s}^*)}.
\end{equation}

The bounded-error estimate
\eqref{eq:height-bounded-error} gives
\begin{equation}\label{eq:height-unit-upper}
 \Phi^*(Z_{y,s})
 =\Phi^*(1,y,s)
 \leq1+C_\Phi
 \lesssim1.
\end{equation}
Since \(\lambda(A_{r,y,s}^*)\simeq r\) and
\(\lambda(A_{r,y,s}^*)\geq1\), the lower bound in
\eqref{eq:height-comparison} gives
\begin{equation}\label{eq:height-reference-lower}
 \Phi^*(A_{r,y,s}^*)\gtrsim r.
\end{equation}
Finally, \eqref{eq:unit-height-reference-geometry} and
Lemma~\ref{lem:green-upper} imply
\begin{equation}\label{eq:v2-reference-upper}
 v_2(A_{r,y,s}^*)
 =
 G_{\cH}^{\Omega_2}(P_2;A_{r,y,s}^*)
 \lesssim r^{-(n+1)}.
\end{equation}
Substituting \eqref{eq:height-unit-upper},
\eqref{eq:height-reference-lower}, and
\eqref{eq:v2-reference-upper} into
\eqref{eq:v2-height-comparison} yields
\[
 v_2(1,y,s)\lesssim r^{-(n+2)}.
\]
The constants are uniform in \((y,s)\), which proves
\eqref{eq:unit-height-decay}.  The restriction of \(\omega_{\cH,\Omega_1}^{P_1}\) to
\(C\Delta_r^-(z_0)\) has mass at most one.  Therefore,
\begin{align*}
 \iint_{C\Delta_r^-(z_0)}
 v_2(1,y,s)\,
 \mathrm d\omega_{\cH,\Omega_1}^{P_1}(y,s)
 &\leq
 \biggl(
 \sup_{(y,s)\in C\Delta_r^-(z_0)}v_2(1,y,s)
 \biggr)
 \omega_{\cH,\Omega_1}^{P_1}
 \bigl(C\Delta_r^-(z_0)\bigr)\lesssim r^{-(n+2)}.
\end{align*}
This is \eqref{eq:forward-adjoint-defect}.
\end{proof}

\begin{rem}
The point \(A_{r,y,s}^*\) may depend on \((y,s)\).  This causes no loss
because all geometric and analytic constants are uniform.  A single
reference point \(A_r^*\) may be used only when one comparison cylinder
has a reduced region containing
\(\{1\}\times C\Delta_r^-(z_0)\).  At unit height, the proof uses only
the upper estimate \(\Phi^*(1,y,s)\lesssim1\).  The forward parabolic
measure enters only in the final integration of the positive pointwise
bound.  No representation of the adjoint difference
\(\delta_\lambda v_2(\lambda,x,t)
=v_2(\lambda+1,x,t)-v_2(\lambda,x,t)\)
by forward parabolic measure is asserted.
\end{rem}

\begin{prop}
\label{prop:summation}
Assume \eqref{eq:full-periodicity} and let \(w\) be as in
\eqref{Greena}.  Then, for every \(r\geq20\),
\begin{equation}\label{eq:discrete-green-sum}
 \sum_k w(A_k)^2
 \leq C r^nw(A_r^*)^2.
\end{equation}
In particular,
\begin{equation}\label{eq:sampling}
 \sum_k w(A_k)^2
 \leq \frac C{r^3}
 \iiint_{T_{Cr}^-(z_0)}w^2
 \,\mathrm d \lambda\,\mathrm d x\,\mathrm d t.
\end{equation}
\end{prop}

\begin{proof}
Equations \eqref{eq:correct-comparison},
\eqref{eq:correct-sum-to-integral}, and
\eqref{eq:forward-adjoint-defect} give
\[
 \sum_k w(A_k)^2
 \lesssim r^{2(n+1)}w(A_r^*)^2 r^{-(n+2)}
 =r^nw(A_r^*)^2.
\]
This is \eqref{eq:discrete-green-sum}.  To obtain the stated bulk form,
put
\[
 A_r'=(r,x_0,t_0-3r^2).
\]
Thus \(A_r'\) precedes \(A_r^*=(r,x_0,t_0-2r^2)\) by exactly \(r^2\).
The two points are joined by a fixed adjoint Harnack chain contained in
\(T_{Cr}^-(z_0)\), after increasing the structural constant \(C\).
The chain and its fixed enlargements remain at parabolic distance at
least a fixed multiple of \(r\) from the singularity \(P\) and stay
inside the half-space.  Hence
\[
 w(A_r^*)\lesssim w(A_r').
\]
Fix \(0<\delta<1/8\) and set
\[
 Z_r'=\bigl((r,x_0),t_0-3r^2-\tfrac18\delta^2r^2\bigr),
 \qquad
 \mathcal C_r'=\mathcal Q_{\delta r}^+(Z_r').
\]
This cylinder and a fixed enlargement are contained in
\(T_{Cr}^-(z_0)\) and are pole-free.  Adjoint local boundedness at its
interior point \(A_r'\) gives the desired estimate.  Indeed,
\[
 \frac18\delta^2r^2<\left(\frac{\delta r}{2}\right)^2,
 \qquad
 A_r'\in\mathcal Q_{\delta r/2}^+(Z_r').
\]
Therefore
\[
 w(A_r^*)^2
 \lesssim w(A_r')^2
 \lesssim r^{-(n+3)}
 \iiint_{\mathcal C_r'}w^2
 \leq r^{-(n+3)}
 \iiint_{T_{Cr}^-(z_0)}w^2
 \,\mathrm d \lambda\,\mathrm d x\,\mathrm d t.
\]
Since \(n-(n+3)=-3\), the claimed estimate follows.
\end{proof}

\begin{rem}\label{rem:role-spatial-periodicity}
It is useful to identify precisely the role of full spatial periodicity
in the constructions in Lemma~\ref{lem:adjoint-periodic-corrector} and Lemma~\ref{lem:height-coordinate}.  This assumption is essential for the
compact-torus corrector argument used above, but it is not an abstractly
necessary condition for the existence of an adjoint height coordinate.
Periodicity in every component of \(Y=(\lambda,x)\) allows the corrector
equation \eqref{eq:adjoint-cell} to be posed on the compact spatial torus
\[
 \mathbb T_Y^{n+1}
 =
 \mathbb T_\lambda\times\mathbb T_x^n.
\]
For every \(h\in\mathrm H^1_{\mathrm{per}}(\mathbb T^{n+1})\) with zero
mean, Poincar\'e's inequality provides the spectral gap
\[
 \|h\|_{\mathrm L^2(\mathbb T^{n+1})}
 \leq
 C\|\nabla_Yh\|_{\mathrm L^2(\mathbb T^{n+1})}.
\]
Consequently, the difference of two trajectories for the time-reversed
corrector equation decays exponentially.  This gives the pullback limit
from the infinite past and hence the bounded complete corrector
\(\chi^*\).  The torus energy estimate is also uniform on every translated
spatial cell.  Local boundedness applied to
\[
 p^*(Y,t)=\lambda+\chi^*(Y,t)
\]
therefore gives the global estimate
\[
 \|\chi^*\|_{\mathrm L^\infty(\mathbb R^{n+2})}
 \leq M_\chi.
\]
If periodicity is assumed only in \(\lambda\), the corresponding spatial
domain is \(\mathbb T_\lambda\times\mathbb R_x^n\).  This cylinder is
noncompact and has no Poincar\'e spectral gap.  For
example, if \(\zeta\in C_0^\infty(\mathbb R^n)\) and
\[
 h_R(\lambda,x)
 =
 R^{-n/2}\zeta(x/R),
\]
then \(h_R\) is independent of \(\lambda\) and
\[
 \|h_R\|_2\simeq1,
 \qquad
 \|\nabla_Yh_R\|_2\simeq R^{-1}.
\]
Thus tangential modes of arbitrarily low frequency prevent a uniform
exponential contraction.  Neither the existence of a bounded complete
corrector nor a uniform bound for it follows from periodicity in
\(\lambda\) alone.  Once transverse periodicity is fixed, tangential
periodicity supplies the remaining compactness needed for the spatial
torus, the Poincar\'e inequality, and the global bound for \(\chi^*\).
Periodicity in \(\lambda\) also
gives
\[
 p^*(\lambda+N,x,t)=p^*(\lambda,x,t)+N,
 \qquad N\in\mathbb Z.
\]
This affine identity records the transverse periodic structure.  The slab
comparison itself uses only the bounded-error estimate
\(\|p^*-\lambda\|_\infty\leq M_\chi\), which yields
\[
 |\Phi_N^*(\lambda,x,t)-\lambda|
 \leq2M_\chi
\]
uniformly in \(N,x,t\).
\end{rem}
\begin{rem} The obstruction under transverse periodicity alone is genuine.  For
\(n=1\), let \(X=(\lambda,x)\) and consider the symmetric matrix
\[
 A(x)
 =
 \begin{pmatrix}
  1&\varepsilon\operatorname{sgn}x\\
  \varepsilon\operatorname{sgn}x&1
 \end{pmatrix},
 \qquad 0<\varepsilon<1.
\]
This coefficient is uniformly elliptic and is independent of both
\(\lambda\) and \(t\).  In particular, it is one-periodic in \(\lambda\).
Suppose that there existed an adjoint solution \(\Psi^*\) satisfying
\[
 \cH^*\Psi^*=0\quad\text{in }\mathbb R^3_+,
 \qquad
 \Psi^*=0\quad\text{on }\{\lambda=0\},
 \qquad
 |\Psi^*(\lambda,x,t)-\lambda|\leq C
 \quad\text{in }\mathbb R^3_+.
\]
For \(R>0\), define
\[
 \Psi_R^*(X,t)
 =
 \frac{1}{R}\Psi^*(RX,R^2t).
\]
Since \(A(Rx)=A(x)\), the function \(\Psi_R^*\) solves the same adjoint
equation.  Moreover,
\[
 |\Psi_R^*(X,t)-\lambda|\leq\frac{C}{R}.
\]
It follows that \(\Psi_R^*\to\lambda\) locally in
\(\mathrm L^2\).  Caccioppoli's inequality gives local weak compactness of
the gradients, so passage to the limit in the weak formulation would imply
that \(\lambda\) solves the adjoint equation.  This is impossible because
\[
 \cH^*\lambda
 =
 -\operatorname{div}_X(Ae_\lambda)
 =
 -\partial_x\bigl(\varepsilon\operatorname{sgn}x\bigr)
 =
 -2\varepsilon\delta_{\{x=0\}}.
\]
Thus transverse periodicity alone does not guarantee an adjoint height
coordinate satisfying \(\Psi^*=\lambda+O(1)\).  On the other hand, full spatial periodicity is not necessary in every
special case.  For example, if
\[
 \operatorname{div}_X(A^{\mathsf T}e_\lambda)=0
 \quad\text{in the sense of distributions},
\]
then \(\Phi^*=\lambda\) is itself an adjoint height coordinate, without
any tangential periodicity.  More generally, the full spatial periodicity
assumption may be replaced by any hypothesis that independently provides a
positive adjoint solution satisfying
\[
 \Phi^*=0\quad\text{on }\{\lambda=0\},
 \qquad
 |\Phi^*-\lambda|\leq C.
\]
Once such a function has been constructed, its use in
Lemma~\ref{lem:unit-height-decay} and in the subsequent large-scale
summation requires no further periodicity.
\end{rem}

\section{Proof of the large-scale reduction}
\label{sec:large-scale-proof}

\begin{proof}[Proof of Theorem~\ref{thm:large-scale}]
By Lemma~\ref{lem:cube-conventions}, it is enough to establish the
reverse H\"older estimate on backward parabolic cubes.  Let \(r>1\),
fix \(z_0=(x_0,t_0)\), and set \(\Delta=\Delta_r^-(z_0)\).  Let \(P\) be
an \(M\)-admissible pole for \(\Delta\), and let \(w\) be as in
\eqref{Greena}, that is,
\[
 w(Z)=G_{\cH}(P;Z).
\]
Here and below, \(20\) denotes a fixed structural large-scale threshold.
It has been enlarged, if necessary, so that all scale-\(r\) reference
corkscrew heights used in Lemma~\ref{lem:boundary-comparison} are at
least one.

Suppose first that \(r\geq20\).  Cover \(\Delta\) by unit backward
parabolic cubes \(\Delta_k=\Delta_1^-(z_k)\), where
\(z_k=(x_k,t_k)\), and choose the points
\(A_k=A_1^+(z_k)\), exactly as in
Subsection~\ref{subsec:two-comparison}.  After increasing \(M\) by a
fixed factor, \eqref{eq:admissible-pole} and the location of \(z_k\)
give
\[
 t(P)-t_k\simeq\lambda(P)^2,
 \qquad
 |x(P)-x_k|+\lambda(P)
 \lesssim\lambda(P)
 \lesssim\bigl(t(P)-t_k\bigr)^{1/2},
\]
and \(t(P)-t_k\gtrsim M^2r^2\).  For each \(k\), let
\(P_k=P_{\Delta_1(z_k)}^{\mathrm{far}}\) be the fixed
\(M\)-admissible pole supplied by
Lemma~\ref{lem:universal-subcube-pole}.  Thus \(P\) and \(P_k\) lie in
one common forward pole region at scale one, uniformly in \(r\) and
\(k\).  By Lemma~\ref{lem:cube-conventions}, the local reverse H\"older
estimate applies to \(\Delta_k\) at \(P_k\), and
Lemma~\ref{lem:change-of-pole} transfers it to \(P\).  In particular,
\(\omega^P\ll\mathrm d x\,\mathrm d t\) on each \(\Delta_k\), and hence
on \(\Delta\).  Let \(k^P\) denote the associated parabolic Poisson
kernel at \(P\), and put \(A_k^-=A_1^-(z_k)\).  The local reverse H\"older
hypothesis gives
\begin{equation}\label{eq:local-rh-unit-cubes}
 \iint_{\Delta_k}(k^P)^2\,\mathrm d x\,\mathrm d t
 \lesssim
 \frac{(\omega^P(\Delta_k))^2}{|\Delta_k|}
 \simeq
 (\omega^P(\Delta_k))^2.
\end{equation}

We next use the right-hand inequality in the
Green-function/parabolic-measure comparison
\eqref{eq:cfms-forward-raw}.  Choose a fixed finite family of symmetric
cubes \(\Theta_{k,m}=\Delta_{\rho_m}(\zeta_{k,m})\), with
\(\rho_m\simeq1\), such that their half-cubes cover \(\Delta_k\) and
their fixed enlargements lie in a fixed dilation of \(\Delta_k\).
Applying the right-hand inequality in \eqref{eq:cfms-forward-raw} to
each \(\Theta_{k,m}\), and comparing the resulting past-oriented
corkscrews with \(A_k^-\) by
Lemma~\ref{lem:green-boundary-time-comparison} and fixed, correctly
oriented Harnack chains, gives
\[
 \omega^P(\Delta_k)\lesssim w(A_k^-).
\]
Here the number of cubes, their radii, and all chain geometries are
uniform in \(k\).  Since \(w\), regarded as a function of \(Z\), is an
adjoint solution with a quantitatively separated singularity,
Lemma~\ref{lem:green-boundary-time-comparison}, together with fixed
correctly oriented Harnack chains, gives
\[
 w(A_k^-)\lesssim w(A_k).
\]
Consequently,
\begin{align}
 \iint_{\Delta}(k^P)^2\,\mathrm d x\,\mathrm d t
 &\leq
 \sum_k\iint_{\Delta_k}(k^P)^2\,
 \mathrm d x\,\mathrm d t\lesssim
 \sum_k(\omega^P(\Delta_k))^2
 \lesssim
 \sum_k w(A_k)^2.
 \label{eq:local-rh-cover}
\end{align}
Proposition~\ref{prop:summation} now yields
\begin{equation}\label{eq:large-scale-after-summation}
 \iint_{\Delta}(k^P)^2\,\mathrm d x\,\mathrm d t
 \lesssim
 r^n w(A_r^*)^2.
\end{equation}

To estimate the reference value \(w(A_r^*)\), put
\[
 \rho=\frac r2,
 \qquad
 \widetilde z=(x_0,t_0-\tfrac12r^2),
 \qquad
 \widetilde\Delta=\Delta_\rho(\widetilde z).
\]
Then \(\Delta_{\rho/2}(\widetilde z)\subset\Delta\).  Moreover,
\[
 A_\rho^-(\widetilde z)
 =(2r,x_0,t_0-\tfrac92r^2),
 \qquad
 A_r^*=(r,x_0,t_0-2r^2).
\]
Thus \(A_\rho^-(\widetilde z)\) precedes \(A_r^*\) by
\(5r^2/2\), and the two points are joined by a uniformly controlled
adjoint Harnack chain.  The constant \(M\) in the pole-admissibility
condition is chosen so that
\(P\in\mathcal P_{\kappa,\rho}^+(\widetilde z)\) and so that the chain
lies quantitatively in the causal past of \(P\).  The adjoint Harnack
inequality and Lemma~\ref{lem:green-boundary-time-comparison} give
\begin{equation}\label{eq:large-reference-comparison}
 w(A_r^*)
 \lesssim
 w(A_\rho^-(\widetilde z))
 \lesssim
 w(A_\rho^+(\widetilde z)).
\end{equation}
We now use the left-hand inequality in
\eqref{eq:cfms-forward-raw}.  It gives
\begin{align}
 w(A_\rho^+(\widetilde z))
 &\lesssim
 \rho^{-(n+1)}
 \omega^P\bigl(\Delta_{\rho/2}(\widetilde z)\bigr)\lesssim
 r^{-(n+1)}\omega^P(\Delta).
 \label{eq:large-scale-green-measure}
\end{align}
Combining
\eqref{eq:large-scale-after-summation}-%
\eqref{eq:large-scale-green-measure}, we obtain
\begin{align}
 \iint_{\Delta}(k^P)^2\,\mathrm d x\,\mathrm d t
 &\lesssim
 r^n\bigl(r^{-(n+1)}\omega^P(\Delta)\bigr)^2=r^{-(n+2)}(\omega^P(\Delta))^2\lesssim
 \frac{(\omega^P(\Delta))^2}{|\Delta|},
 \label{eq:global-rh-direct}
\end{align}
where the last step uses \(|\Delta|=|\Delta_r^-(z_0)|\simeq r^{n+2}\).  Equation~\eqref{eq:global-rh-direct} is precisely the squared
\(\mathrm{RH}_2\) estimate on \(\Delta\).

It remains to consider \(1<r<20\).  Cover \(\Delta\) by a uniformly
bounded family of backward parabolic cubes
\(\Delta_j=\Delta_{\rho_j}^-(z_j)\), with bounded overlap, whose radii
\(\rho_j\) satisfy \(\rho_j\in[1/2,1]\).  Let
\(P_j=P_{\Delta_{\rho_j}(z_j)}^{\mathrm{far}}\) be supplied by
Lemma~\ref{lem:universal-subcube-pole}.  The preceding calculation,
now using \(1<r<20\) and \(\rho_j\simeq1\), shows that \(P\) and
\(P_j\) lie in a common forward pole region at scale \(\rho_j\), with
uniform constants.  By Lemma~\ref{lem:cube-conventions}, the local reverse
H\"older estimate applies to \(\Delta_j\) at \(P_j\), and
Lemma~\ref{lem:change-of-pole} therefore gives
\begin{align}
 \iint_{\Delta}(k^P)^2\,\mathrm d x\,\mathrm d t
 &\lesssim
 \sum_j
 \frac{(\omega^P(\Delta_j))^2}{|\Delta_j|}\lesssim
 \biggl(\sum_j\omega^P(\Delta_j)\biggr)^2\lesssim
 (\omega^P(C\Delta))^2.
 \label{eq:intermediate-scale-cover}
\end{align}
Lemmas~\ref{lem:parabolic-measure-doubling}
and~\ref{lem:change-of-pole}, in their backward-cube formulations,
imply boundary doubling for the admissible pole \(P\).  Hence
\[
 \omega^P(C\Delta)\lesssim\omega^P(\Delta).
\]
Since \(|\Delta|\simeq r^{n+2}\simeq1\) for \(1<r<20\), it follows from \eqref{eq:intermediate-scale-cover} that
\[
 \iint_{\Delta}(k^P)^2\,\mathrm d x\,\mathrm d t
 \lesssim
 \frac{(\omega^P(\Delta))^2}{|\Delta|}.
\]
The range \(0<r\leq1\) is covered by the hypothesis of
Theorem~\ref{thm:large-scale}.  This completes the proof.
\end{proof}

We also record why the local absolute-continuity conclusions used above
give absolute continuity on the whole boundary for each fixed pole.
Fix \(P=(X_P,t_P)\).  The set
\(\mathbb R^n\times(-\infty,t_P)\) has a countable cover by boundary
cubes whose fixed comparison regions lie strictly below \(t_P\).  On
each such cube, boundary comparison and a finite, correctly oriented
Harnack chain compare the restriction of \(\omega_{\cH}^P\) with the
parabolic measure having an admissible local pole.  No uniform comparison
constant is needed for this qualitative assertion.  The latter measure
is absolutely continuous on the cube by the estimate just proved, and
hence so is \(\omega_{\cH}^P\).  Finally,
\eqref{eq:measure-causal-support} gives zero mass to
\(\mathbb R^n\times[t_P,\infty)\).  Thus
\(\omega_{\cH}^P\ll\mathrm d x\,\mathrm d t\) on the whole boundary.

\section{Proof of the small-scale theorem}\label{sec:parabolic-KS}

In this section we prove Theorem~\ref{thm:parabolic-KS}.  No periodicity is
used.  We assume throughout that \(A=A^{\mathsf T}\), that the boundary
trace \(A^0\) satisfies
\eqref{eq:boundary-time-modulus}-\eqref{eq:boundary-time-continuity}, and
that the transverse modulus satisfies \eqref{eq:normal-square-Dini}.  We use
the uniformly continuous representative of \(A^0\) fixed in
Subsection~\ref{subsec:small-assumptions}.

\subsection{A sufficient half-order BMO condition}

We first record that the half-order BMO condition
\eqref{eq:boundary-fiber-bmo} implies the temporal hypothesis used in the
main theorem.

\begin{lem}\label{lem:temporal-freezing}
Let \(a\in \mathrm L^\infty(\mathbb R)\) and suppose that
\(D_t^{1/2}a\in \mathrm{BMO}(\mathbb R)\).  Then \(a\) has a representative such
that
\begin{equation}\label{eq:bmo-holder}
 |a(t)-a(s)|
 \leq
 C\|D_t^{1/2}a\|_{\mathrm{BMO}(\mathbb R)}|t-s|^{1/2}.
\end{equation}
In particular, if \eqref{eq:boundary-fiber-bmo} holds, then \(A^0\) admits
the uniformly continuous representative required in
Subsection~\ref{subsec:small-assumptions}, and
\begin{equation}\label{eq:matrix-bmo-holder}
 \bigl\|A^0(\cdot,t)-A^0(\cdot,s)\bigr\|_{\mathrm L^\infty(\mathbb R^n)}
 \leq CK_0|t-s|^{1/2}.
\end{equation}
Consequently, \eqref{eq:bmo-time-modulus} holds.
\end{lem}

\begin{proof}
Put \(g=D_t^{1/2}a\), and let
\(\{\Delta_j\}_{j\in\mathbb Z}\) be a homogeneous
Littlewood-Paley decomposition in \(t\).  Convolution of a \(\mathrm{BMO}\)
function with a mean-zero Schwartz kernel gives
\[
 \|\Delta_jg\|_\infty\leq C\|g\|_{\mathrm{BMO}}.
\]
On the support of the multiplier defining \(\Delta_j\), the multiplier
\(D_t^{-1/2}\) has size \(2^{-j/2}\).  Bernstein's inequality therefore
gives
\begin{equation}\label{eq:LP-a-bounds}
 \|\Delta_ja\|_\infty
 \leq C2^{-j/2}\|g\|_{\mathrm{BMO}},
 \qquad
 \|\partial_t\Delta_ja\|_\infty
 \leq C2^{j/2}\|g\|_{\mathrm{BMO}}.
\end{equation}
Choose \(J\in\mathbb Z\) so that
\(2^J|t-s|\simeq1\).  For the low frequencies,
\begin{align*}
 \sum_{j\leq J}
 |\Delta_ja(t)-\Delta_ja(s)|
 &\leq
 C|t-s|\|g\|_{\mathrm{BMO}}\sum_{j\leq J}2^{j/2}\leq C\|g\|_{\mathrm{BMO}}|t-s|^{1/2}.
\end{align*}
For the high frequencies,
\begin{align*}
 \sum_{j>J}
 |\Delta_ja(t)-\Delta_ja(s)|
 &\leq
 C\|g\|_{\mathrm{BMO}}\sum_{j>J}2^{-j/2}\leq C\|g\|_{\mathrm{BMO}}|t-s|^{1/2}.
\end{align*}
The homogeneous decomposition determines \(a\) modulo a polynomial.
Boundedness leaves only an additive constant, which disappears in the
difference.  This proves \eqref{eq:bmo-holder}.  Applying the scalar
argument to the matrix entries and taking the essential supremum in
\(x\) proves \eqref{eq:matrix-bmo-holder}.  Changing the representative
to the \(\mathrm L_x^\infty\)-valued H\"older representative makes its value at
every \(t_0\) an \(\mathrm L_x^\infty\) limit of admissible matrices at Lebesgue
times.  Symmetry and the closed ellipticity inequalities pass to this
limit, so \(A^0(x,t_0)\) is admissible for every chosen \(t_0\).
\end{proof}

\begin{cor}\label{cor:bmo-sufficient}
The conclusions of Theorems~\ref{thm:parabolic-KS} and~\ref{thm:global},
and of Corollary~\ref{cor:global-D2}, remain valid if
\eqref{eq:boundary-time-modulus}-\eqref{eq:boundary-time-continuity} are
replaced by \eqref{eq:boundary-fiber-bmo}.  In this case the constants may
depend on \(K_0\).
\end{cor}

\begin{proof}
This follows from Lemma~\ref{lem:temporal-freezing} and
\eqref{eq:bmo-time-modulus}.
\end{proof}

\subsection{The two perturbation inputs}

The first perturbation input is the following consequence of the \(\lambda\)-independent theory
of~\cite{AEN}.  The second input is
Proposition~\ref{prop:parabolic-Carleson-perturbation}\textup{(i)}.

\begin{lem}\label{lem:lambda-independent-B2-openness}
There exists
\(\varepsilon_\infty=\varepsilon_\infty(n,\mu,\Lambda)>0\) with the
following property.  Let \(B=B(x)=B(x)^{\mathsf T}\) and let
\(B'=B'(x,t)\) be real and independent of \(\lambda\), with the fixed
ellipticity bounds.  If
\begin{equation}\label{eq:Linfty-openness}
 \|B'-B\|_\infty\leq\varepsilon_\infty,
\end{equation}
then \(\omega_{\cH_{B'}}\in B_2(\mathrm d x\,\mathrm d t)\), with a
characteristic depending only on \(n,\mu,\Lambda\).
\end{lem}

\begin{proof}
Set \(B'^\dagger(x,t)=B'(x,-t)^{\mathsf T}\).  Since
\(B=B^{\mathsf T}\) is independent of time,
\(\|B'^\dagger-B\|_\infty=\|B'-B\|_\infty\).
Theorem~2.32\textup{(vi)} of~\cite{AEN} gives compatible well-posedness of the
\(\mathrm L^2\) regularity problem for the coefficient \(B\), with
uniform Rellich bounds.  The \(\mathrm L^\infty\)-openness theorem
\cite[Theorem~2.19]{AEN}, applied at index \(s=0\), therefore gives the
same compatible well-posedness conclusion for \(B'^\dagger\) if \(\varepsilon_\infty\) is
sufficiently small.  Under the canonical time-reversal identification,
this is compatible well-posedness of the adjoint \(\mathrm L^2\) regularity problem
for \(\cH_{B'}^*\).  Corollary~2.38 of~\cite{AEN} then gives
compatible well-posedness of \((D)_2\) for \(\cH_{B'}\).  In particular,
the corresponding solution satisfies
\(u(\lambda,\cdot,\cdot)\to f\) in \(\mathrm L^2\) as
\(\lambda\downarrow0\), and for data in the energy trace class it agrees with the canonical
energy solution.  For real scalar
equations, local boundedness and local H\"older continuity on Whitney
regions convert the modified non-tangential maximal estimate and
Whitney-average convergence in that result into the maximal estimate and
non-tangential convergence in Definition~\ref{defn:Dq}, after a fixed
enlargement of the aperture.
Theorem~\ref{thm:Bp-Dq} now yields
\(\omega_{\cH_{B'}}\in B_2(\mathrm d x\,\mathrm d t)\).  All constants
are uniform over the stated class.
\end{proof}

\subsection{Proof of the reverse H\"older estimate}

\begin{proof}[Proof of Theorem~\ref{thm:parabolic-KS}]
Fix \(\kappa>16\).  Increase \(M\), if necessary, so that
\(c_aM^2>\max\{64,\kappa^2\}\), and put
\(\kappa_P=(C_a+1)c_a^{-1/2}\).  Then every \(M\)-admissible pole \(P\) for
\(\Delta_r=\Delta_r(x_0,t_0)\) satisfies
\[
 P\in\mathcal P_{\kappa_P,r}^+(x_0,t_0),
 \qquad
 t(P)>t_0+\kappa^2r^2.
\]
Fix \(\Delta_r\) and such a pole \(P\).  Choose a fixed geometric number
\(L>\max\{M,2\kappa\}\), also larger than all corkscrew constants used
below.  Choose
\(\vartheta_r\in C_0^\infty(\mathbb R)\) so that
\[
 0\leq\vartheta_r\leq1,\qquad
 \vartheta_r=1\quad\text{if }|t-t_0|\leq(Lr)^2,
 \qquad
 \vartheta_r=0\quad\text{if }|t-t_0|\geq(2Lr)^2.
\]
Define the time-independent coefficient
\begin{equation}\label{eq:time-frozen-B}
 B(x):=A^0(x,t_0)
\end{equation}
and the globally defined, \(\lambda\)-independent coefficient
\begin{equation}\label{eq:time-localized-reference}
 B_r(x,t)
 :=
 (1-\vartheta_r(t))B(x)+\vartheta_r(t)A^0(x,t).
\end{equation}
Both are symmetric and have the ellipticity bounds of \(A\).  Outside the support of \(\vartheta_r\), \(B_r(x,t)=B(x)\).  On the
support of \(\vartheta_r\), \(|t-t_0|^{1/2}\leq2Lr\).  Hence
\begin{equation}\label{eq:Br-close-B}
 \|B_r-B\|_\infty\leq\varpi(2Lr).
\end{equation}
Lemma~\ref{lem:lambda-independent-B2-openness} shows that
\begin{equation}\label{eq:first-smallness}
 \varpi(2Lr)\leq\varepsilon_\infty
\end{equation}
implies \(\omega_{\cH_{B_r}}\in B_2(\mathrm d x\,\mathrm d t)\), with a
uniform reverse H\"older constant.

We next restore the transverse dependence.  Let
\(\chi_r\in C^\infty([0,\infty))\) satisfy
\[
 0\leq\chi_r\leq1,\qquad
 \chi_r(\lambda)=1\quad(0\leq\lambda\leq Lr),
 \qquad
 \chi_r(\lambda)=0\quad(\lambda\geq2Lr),
\]
and put
\begin{equation}\label{eq:localized-coefficient-new}
 \widetilde A_r(\lambda,x,t)
 :=
 (1-\vartheta_r(t))B(x)
 +\vartheta_r(t)
 \left[(1-\chi_r(\lambda))A^0(x,t)
       +\chi_r(\lambda)A(\lambda,x,t)\right].
\end{equation}
This is a convex combination of \(B,A^0\), and \(A\), and hence remains
symmetric and uniformly elliptic.  Moreover,
\begin{equation}\label{eq:local-agreement-new}
 \widetilde A_r=A
 \quad\text{if}\quad
 0<\lambda<Lr,\qquad |t-t_0|<(Lr)^2,
\end{equation}
whereas
\begin{equation}\label{eq:normal-discrepancy-exact}
 \widetilde A_r(\lambda,x,t)-B_r(x,t)
 =
 \vartheta_r(t)\chi_r(\lambda)
 \bigl(A(\lambda,x,t)-A^0(x,t)\bigr).
\end{equation}

Let \(E_r^\sharp\) be the Whitney essential supremum in
\eqref{eq:carleson-discrepancy}.  If
\((\lambda',y,s)\in W(\lambda,x,t)\), then \(\lambda'\leq C\lambda\).  Moreover,
\eqref{eq:normal-discrepancy-exact} vanishes unless
\(\lambda'<2Lr\).  Hence
\begin{equation}\label{eq:Whitney-Dini-bound}
 E_r^\sharp(\lambda,x,t)
 \leq
 \mathbf 1_{\{\lambda<C_Lr\}}\eta(C\lambda).
\end{equation}
Here \(C_L\) depends only on \(L\) and the fixed Whitney aperture.
The two coefficients are evaluated at the same \((y,s)\), so no
regularity in \(x\) is used.  From
\eqref{eq:carleson-discrepancy},
\begin{align}
 \|\widetilde A_r-B_r\|_*^2
 &\leq
 C\int_0^{C_Lr}\eta(C\lambda)^2
 \,\frac{\mathrm d \lambda}{\lambda}\leq
 C_L\int_0^{C_Lr}\eta(s)^2\,\frac{\mathrm d s}{s}
 =
 C_L\mathfrak D(C_Lr)^2.
 \label{eq:Dini-Carleson-bound}
\end{align}

Whenever \(\varpi(2Lr)\leq\varepsilon_\infty\), the first perturbation
gives a uniform \(B_2\)-characteristic for \(\omega_{\cH_{B_r}}\).  We
fix the corresponding Carleson threshold \(\varepsilon_*\) in
Proposition~\ref{prop:parabolic-Carleson-perturbation}\textup{(i)}.

Choose \(r_0\in(0,(2C_L)^{-1})\) so small that
\begin{equation}\label{eq:choice-r0-new}
 \varpi(2Lr_0)\leq\varepsilon_\infty,
 \qquad
 C_L^{1/2}\mathfrak D(C_Lr_0)\leq\varepsilon_*.
\end{equation}
Such an \(r_0\) exists by
\eqref{eq:boundary-time-continuity} and \eqref{eq:Dini-tail-zero}.  For
every \(0<r\leq r_0\), Lemma~\ref{lem:lambda-independent-B2-openness},
\eqref{eq:Dini-Carleson-bound}-\eqref{eq:choice-r0-new}, and
Proposition~\ref{prop:parabolic-Carleson-perturbation}\textup{(i)} give
\(\omega_{\cH_{\widetilde A_r}}\in B_2(\mathrm d x\,\mathrm d t)\),
with a uniform reverse H\"older constant.  In particular, for
\(P_\Delta=A_{4r}^+(x_0,t_0)\),
\[
 \left(
 \frac{1}{|\Delta_r|}
 \iint_{\Delta_r}\bigl(k_{\widetilde A_r}^{P_\Delta}\bigr)^2\,
 \mathrm d x\,\mathrm d t
 \right)^{1/2}
 \leq
 \frac{C}{|\Delta_r|}
 \omega_{\widetilde A_r}^{P_\Delta}(\Delta_r).
\]
Since the fixed pole \(P\) belongs to
\(\mathcal P_{\kappa_P,r}^+(x_0,t_0)\),
Lemma~\ref{lem:change-of-pole} transfers this estimate from
\(P_\Delta\) to \(P\), with uniform constants.

It remains to return from \(\widetilde A_r\) to \(A\).  Since
\(L>\kappa\), \eqref{eq:local-agreement-new} shows that the two matrices
agree throughout \(\mathcal U_{\kappa r}(x_0,t_0)\).  Moreover, the
choice of \(M\) gives
\[
 P\in\mathcal P_{\kappa_P,r}^+(x_0,t_0),
 \qquad
 t(P)>t_0+\kappa^2r^2.
\]
We may therefore apply
Corollary~\ref{cor:local-change-operator} with \(A_1=A\),
\(A_2=\widetilde A_r\), and the same pole \(P_1=P_2=P\).  Thus, for
every Borel set \(E\subset\Delta_r\),
\begin{equation}\label{eq:local-change-operator}
 \frac{\omega_A^P(E)}{\omega_A^P(\Delta_r)}
 \simeq
 \frac{\omega_{\widetilde A_r}^P(E)}
      {\omega_{\widetilde A_r}^P(\Delta_r)}.
\end{equation}
The preceding change-of-pole argument gives the normalized
\(\mathrm{RH}_2\) estimate for \(\omega_{\widetilde A_r}^P\).  The
comparison \eqref{eq:local-change-operator} first yields absolute
continuity of \(\omega_A^P\) on \(\Delta_r\).  Lebesgue differentiation
then compares the normalized densities almost everywhere and transfers
\(\mathrm{RH}_2\) to \(A\), uniformly for \(0<r\leq r_0\).

Finally, let \(r_0<r\leq1\).  Choose
\(\rho\) with \(r_0/4\leq\rho\leq r_0/2\).  A standard equal-radius
covering provides boundary parabolic cubes
\(\Delta_j=\Delta_\rho(z_j)\), \(1\leq j\leq N\), such that
\[
 \Delta_r\subset\bigcup_{j=1}^N\Delta_j\subset C\Delta_r,
 \qquad
 \sum_{j=1}^N\mathbf 1_{\Delta_j}\leq C,
 \qquad
 N\simeq\left(\frac r\rho\right)^{n+2}
 \leq Cr_0^{-(n+2)}.
\]
For each \(j\), let
\(P_j=P_{\Delta_j}^{\mathrm{far}}\) be supplied by
Lemma~\ref{lem:universal-subcube-pole}.  The estimate proved above
applies to \(\omega_A^{P_j}\).  Since the centers
\(z_j=(x_j,t_j)\) lie in a fixed enlargement of \(\Delta_r\), the
admissibility of \(P\), after the fixed choice of \(M\), gives
\[
 t(P)-t_j\simeq\lambda(P)^2,
 \qquad
 |x(P)-x_j|+\lambda(P)
 \lesssim\bigl(t(P)-t_j\bigr)^{1/2}.
\]
Thus the original pole \(P\) and the poles \(P_j\) lie in a common forward pole region
\(\mathcal P_{\kappa',\rho}^+(z_j)\), with fixed \(\kappa'\), and are
quantitatively separated from the scale-\(\rho\) comparison region.
Lemma~\ref{lem:change-of-pole} therefore transfers the normalized
\(\mathrm{RH}_2\) estimate from \(P_j\) to \(P\), uniformly in \(j\).
Since
\(|\Delta_j|\simeq|\Delta_r|/N\), bounded overlap and boundary doubling
give
\begin{align*}
 \iint_{\Delta_r}(k_A^P)^2\,\mathrm d x\,\mathrm d t
 \leq
 C\sum_{j=1}^N
 \frac{(\omega_A^P(\Delta_j))^2}{|\Delta_j|}&\leq
 C\frac{N}{|\Delta_r|}
 \left(\sum_{j=1}^N\omega_A^P(\Delta_j)\right)^2\\
 &\leq
 C\frac{N}{|\Delta_r|}
 (\omega_A^P(C\Delta_r))^2
 \leq
 C\frac{N}{|\Delta_r|}
 (\omega_A^P(\Delta_r))^2.
\end{align*}
Consequently,
\[
 \left(
 \frac1{|\Delta_r|}
 \iint_{\Delta_r}(k_A^P)^2\,\mathrm d x\,\mathrm d t
 \right)^{1/2}
 \leq
 CN^{1/2}\frac{\omega_A^P(\Delta_r)}{|\Delta_r|}
 \leq
 Cr_0^{-(n+2)/2}
 \frac{\omega_A^P(\Delta_r)}{|\Delta_r|}.
\]
This proves \eqref{eq:KS-small-rh}.  The qualitative absolute continuity
on the whole boundary for each fixed pole follows from the countable-cover
argument recorded after the proof of Theorem~\ref{thm:large-scale}, using
cubes of radius at most \(r_0\); no uniform Harnack-chain constant is needed
for that conclusion.
\end{proof}

\subsection{The half-order BMO condition and the first-order framework}

The condition \eqref{eq:boundary-fiber-bmo} is neither required in
Theorem~\ref{thm:parabolic-KS} nor assumed in~\cite{AEN}, which allows
bounded measurable time-dependent coefficients.  It is nevertheless a
natural sufficient condition for the temporal modulus because half-order
time differentiation is intrinsic to the first-order framework
of~\cite{AEN}.  The Rellich calculation in
\cite[Sections~12.5-12.7 and Proposition~2.34]{AEN} gives the
mixed-order equivalence
\[
 \|h_r\|_2^2+
 \|h_r\|_{\dot{\mathrm H}^{-1/4}_{\partial_t-\Delta_x}}^2
 \simeq
 \|h\|_2^2+\|h\|_{\dot{\mathrm H}^{-1/2}_{P}}^2
 \simeq
 \|h_\perp\|_2^2+
 \|h_\perp\|_{\dot{\mathrm H}^{-1/4}_{\partial_t-\Delta_x}}^2,
\]
where \(h_\perp\) and \(h_r\) are the conormal and tangential parabolic
parts of the boundary conormal gradient \(h\), respectively, and
\(\dot{\mathrm H}^{-1/2}_{P}\) is the homogeneous Sobolev scale
associated with the parabolic Dirac operator \(P\).  The negative-order terms reflect the commutator
\[
 [A^0,H_tD_t^{1/2}],
 \qquad
 \|[A^0,H_tD_t^{1/2}]g\|_2
 \lesssim K_0\|g\|_2,
\]
whose bound follows from Murray's commutator theorem~\cite{Murray}.
Consequently, the mixed-order equivalence does not by itself imply a
pure global \(\mathrm L^2\) Rellich estimate.

Under the additional sufficient hypothesis
\eqref{eq:boundary-fiber-bmo}, the weighted maximal estimate in
\cite[Theorem~2.40]{AEN} is unweighted
at heights \(\lambda\leq1\), and so supplies an alternative small-height
description for the \(\lambda\)-independent boundary coefficient
\(A^0\).  For the subsequent transverse perturbation, however, the
scale-dependent coefficient \(B_r\) constructed above is the appropriate
reference: Lemma~\ref{lem:lambda-independent-B2-openness} gives a
uniform global \(B_2\)-estimate for its parabolic measure, and
Proposition~\ref{prop:parabolic-Carleson-perturbation}\textup{(i)} then
applies directly, without requiring an additional localized perturbation
theorem.

\begin{rem}
The proof of Theorem~\ref{thm:parabolic-KS} uses only
\eqref{eq:boundary-time-modulus}-\eqref{eq:boundary-time-continuity}.
Lemma~\ref{lem:temporal-freezing} shows that the condition
\eqref{eq:boundary-fiber-bmo} is sufficient.  A joint parabolic
\(\mathrm{BMO}\) condition on \(D_t^{1/2}A^0\) does not, by itself, imply
the required uniform \(\mathrm L_x^\infty\)-valued time modulus.
\end{rem}

\begin{rem}
The ordinary Dini condition
\[
 \int_0^1\eta(\rho)\,\frac{\mathrm d \rho}{\rho}<\infty
\]
implies \eqref{eq:normal-square-Dini}, since
\(\eta\leq2\|A\|_\infty\).  If transverse regularity is instead stated as
\[
 |\partial_\lambda A(\lambda,x,t)|
 \leq\frac{\omega(\lambda)}{\lambda},\quad\mbox{then}\quad
 \eta(\rho)
 \leq
 \int_0^\rho\omega(s)\,\frac{\mathrm d s}{s}.
\]
In that formulation one must verify the tail-square condition
\[
 \int_0^1
 \left(\int_0^\rho\omega(s)\,\frac{\mathrm d s}{s}\right)^2
 \,\frac{\mathrm d \rho}{\rho}<\infty.
\]
Mere integrability of \(\omega(s)/s\) does not automatically imply this
stronger tail-square condition.
\end{rem}

\begin{rem}
Neither transverse, tangential, nor time periodicity is used in
Theorem~\ref{thm:parabolic-KS}.  Periodicity enters only in the separate
large-scale theorem.
\end{rem}

\subsection{Proof of the global theorem}

\begin{proof}
Theorem~\ref{thm:parabolic-KS} gives the reverse H\"older estimate for
\(0<r\leq1\).  Theorem~\ref{thm:large-scale}, using the spatially periodic adjoint height
coordinate, propagates it to every \(r>1\).
\end{proof}

\begin{proof}[Proof of Corollary~\ref{cor:global-D2}]
Fix \(\Delta_0=\Delta_{r_0}(x_0,t_0)\) and the canonical pole
\(P_0=A_{4r_0}^+(x_0,t_0)\).  Let
\(P=P_{\Delta_0}^{\mathrm{far}}\) be supplied by
Lemma~\ref{lem:universal-subcube-pole}.  Thus \(P\) is an
\(M\)-admissible pole for every subcube of \(\Delta_0\), with uniform
admissibility constants.  The coordinate formulas for \(P\) and \(P_0\)
also place both points in a common forward pole region
\(\mathcal P_{\kappa_0,r_0}^+(x_0,t_0)\), with fixed \(\kappa_0\).
Theorem~\ref{thm:global} gives the \(\mathrm{RH}_2\) estimate for
\(\omega_{\cH}^P\) on every subcube of \(\Delta_0\).
Lemma~\ref{lem:change-of-pole}, applied at scale \(r_0\), compares the
normalized measures for \(P\) and \(P_0\) on every Borel subset of
\(\Delta_0\).  Lebesgue differentiation therefore compares their
normalized Poisson kernels almost everywhere on \(\Delta_0\), so the
reverse H\"older estimate transfers to \(\omega_{\cH}^{P_0}\).  Hence
\(\omega_{\cH}\in B_2(\mathrm d x\,\mathrm d t)\).  Theorem~\ref{thm:Bp-Dq}
with \(p=q=2\) gives solvability of \(D_2\).  For uniqueness, set
\[
 \widetilde A(Y,\tau)=A^{\mathsf T}(Y,-\tau).
\]
Time reversal identifies forward solutions for the coefficient
\(\widetilde A\) with adjoint solutions for \(A\).  The matrix
\(\widetilde A\) satisfies the same ellipticity, symmetry, full spatial
periodicity, transverse square-Dini, and uniform temporal continuity
hypotheses as \(A\).  In fact, the temporal modulus of its boundary trace is
again \(\varpi\).  Theorem~\ref{thm:global}
therefore gives the corresponding \(B_2\) estimate for the adjoint
parabolic measure.  The forward-adjoint uniqueness criterion recalled
after Theorem~\ref{thm:Bp-Dq} now gives uniqueness among weak solutions with
non-tangential maximal function in \(\mathrm L^2\).
\end{proof}

\section{Further results and remarks}\label{sec:completion}

We conclude with four extensions and consequences of the preceding
arguments: a nonsymmetric \(A_\infty\) theorem, a transfer result for
time-independent Lipschitz graph domains, an abstract criterion for the
existence of an adjoint height coordinate, and scale-uniform boundary
estimates for periodic homogenization families.

\subsection{Nonsymmetric coefficients and \texorpdfstring{\(A_\infty\)}{A-infinity}}
\label{subsec:nonsymmetric-Ainfty}

Symmetry is used in Theorem~\ref{thm:parabolic-KS} to obtain the
\(B_2\)-estimate.  If only \(A_\infty\) is sought, the local argument can
instead be based on \cite[Theorem~1.5]{AENDirichlet} and the
finite-Carleson perturbation theory for the parabolic Dirichlet problem.
The adjoint height coordinate then propagates the resulting
\(B_p\)-estimate to every scale.  We use the finite-Carleson statement in
Proposition~\ref{prop:parabolic-Carleson-perturbation}\textup{(ii)}.

The finite-norm conclusion is important here.  A small discrepancy
preserves a prescribed class \(B_p\), whereas a finite discrepancy
preserves \(A_\infty\), with an exponent that may change.  Let
\(A^0=A^0(x,t)\) be real, bounded, and uniformly elliptic, and set
\begin{equation}\label{eq:nonsymmetric-transverse-modulus}
 \eta_0(\rho)
 :=
 \operatorname*{ess\,sup}_{\substack{x\in\R^n,\ t\in\R\\
 0<\lambda\leq\rho}}
 \bigl|A(\lambda,x,t)-A^0(x,t)\bigr|,
 \qquad 0<\rho\leq1.
\end{equation}

\begin{lem}
\label{lem:nonsymmetric-local-Ainfty}
Assume that \(A\) and \(A^0\) are real, bounded, and uniformly elliptic,
that \(A^0=A^0(x,t)\), and that
\begin{equation}\label{eq:nonsymmetric-square-Dini-finite}
 \mathfrak D_0
 :=
 \left(
 \int_0^1\eta_0(\rho)^2\,
 \frac{\mathrm d \rho}{\rho}
 \right)^{1/2}<\infty.
\end{equation}
Then there are \(p_0>1\), \(M>4\), and \(C<\infty\) such that, for
every \((x_0,t_0)\in\R^n\times\R\), every \(0<r\leq1\), and every
\(M\)-admissible pole \(P\) for \(\Delta_r(x_0,t_0)\), one has
\(\omega_{\cH}^P\ll\mathrm d x\,\mathrm d t\) on
\(\Delta_r(x_0,t_0)\), and
\begin{equation}\label{eq:nonsymmetric-local-Bp}
 \left(
 \frac1{|\Delta_r|}
 \iint_{\Delta_r(x_0,t_0)}(k^P)^{p_0}\,
 \mathrm d x\,\mathrm d t
 \right)^{1/p_0}
 \leq
 C\,\frac{\omega_{\cH}^P(\Delta_r(x_0,t_0))}{|\Delta_r|}.
\end{equation}
The constants depend only on \(n,\mu,\Lambda,\mathfrak D_0\) and the
fixed geometric parameters.  No symmetry, temporal regularity, or
periodicity is assumed.
\end{lem}

\begin{proof}
Consider the transversely independent reference operator
\[
 \cH_0
 =
 \partial_t-\operatorname{div}_{\lambda,x}
 \bigl(A^0(x,t)\nabla_{\lambda,x}\bigr).
\]
By \cite[Theorem~1.5]{AENDirichlet},
\(\omega_{\cH_0}\in A_\infty(\mathrm d x\,\mathrm d t)\), with
constants depending only on \(n,\mu,\Lambda\).  This result allows
arbitrary real, nonsymmetric \(A^0\) measurable in \((x,t)\).  Fix
\(\kappa>16\), choose \(M>4\) so large that
\[
 c_aM^2>\max\{64,\kappa^2\},
\]
put \(\kappa_P=(C_a+1)c_a^{-1/2}\), and choose
\(L>\max\{M,2\kappa\}\).  Every \(M\)-admissible pole \(P\) for
\(\Delta_r(x_0,t_0)\) then satisfies
\[
 P\in\mathcal P_{\kappa_P,r}^+(x_0,t_0),
 \qquad
 t(P)>t_0+\kappa^2r^2.
\]
Indeed,
\[
 |x(P)-x_0|+\lambda(P)
 \leq(C_a+1)\lambda(P)
 \leq\kappa_P\bigl(t(P)-t_0\bigr)^{1/2},
\]
and
\[
 t(P)-t_0
 \geq c_a\lambda(P)^2
 \geq c_aM^2r^2
 >\kappa^2r^2.
\]
For \(0<r\leq1\), choose
\(\chi_r\in C^\infty([0,\infty))\) such that
\[
 0\leq\chi_r\leq1,
 \qquad
 \chi_r=1\ \text{on }[0,Lr],
 \qquad
 \chi_r=0\ \text{on }[2Lr,\infty),
\]
and define
\begin{equation}\label{eq:nonsymmetric-localized-coefficient}
 \widetilde A_r
 =
 A^0+\chi_r(A-A^0).
\end{equation}
Thus \(\widetilde A_r\) is real and uniformly elliptic, agrees with
\(A\) for \(0<\lambda<Lr\), and agrees with \(A^0\) for
\(\lambda>2Lr\).

Let \(E_r^\sharp\) be the Whitney essential supremum of
\(\widetilde A_r-A^0\).  Since the transverse coordinate in a Whitney
region is comparable to \(\lambda\),
\[
 E_r^\sharp(\lambda,x,t)
 \leq
 C\mathbf 1_{\{\lambda<C_Lr\}}
 \begin{cases}
  \eta_0(C\lambda),&C\lambda\leq1,\\
  2\Lambda,&C\lambda>1.
 \end{cases}
\]
where \(C_L\) depends only on \(L\) and the fixed Whitney aperture.
Both matrices are evaluated at the same tangential point in this
estimate.  Hence no regularity in \(x\) or \(t\) is used.  It follows
from \eqref{eq:carleson-discrepancy} that
\begin{equation}\label{eq:nonsymmetric-Carleson-bound}
 \|\widetilde A_r-A^0\|_*^2
 \leq
 C_L\bigl(\mathfrak D_0^2+\Lambda^2\bigr),
 \qquad 0<r\leq1.
\end{equation}
The bound is independent of \(r\).
Proposition~\ref{prop:parabolic-Carleson-perturbation}\textup{(ii)} therefore gives
\(\omega_{\cH_{\widetilde A_r}}\in A_\infty\), uniformly in \(r\).
By the quantitative reverse H\"older characterization and
self-improvement recorded in
Proposition~\ref{prop:Ainfty-union-Bp}, there exist an exponent
\(p_0>1\) and a constant, both independent of \(r\), such that
\[
 \omega_{\cH_{\widetilde A_r}}
 \in B_{p_0}(\mathrm d x\,\mathrm d t),
 \qquad 0<r\leq1.
\]

Let \(P\) be the \(M\)-admissible pole in the statement.  By the
preceding parameter choices,
\[
 P\in\mathcal P_{\kappa_P,r}^+(x_0,t_0),
 \qquad
 t(P)>t_0+\kappa^2r^2.
\]
Lemma~\ref{lem:change-of-pole} transfers the canonical
\(B_{p_0}\)-estimate for \(\cH_{\widetilde A_r}\) to
\(\omega_{\cH_{\widetilde A_r}}^P\).  Since \(L>\kappa\), the
localization cylinder in Corollary~\ref{cor:local-change-operator} is contained in
\(\{0<\lambda<Lr\}\), where \(A=\widetilde A_r\).  Applying that
corollary with the same geometric pole \(P\) for both operators and then
using Lebesgue differentiation gives, with the coefficient matrix
indicated in the subscript,
\[
 \frac{k_A^P}{\omega_{\cH}^P(\Delta_r)}
 \simeq
 \frac{k_{\widetilde A_r}^P}
      {\omega_{\cH_{\widetilde A_r}}^P(\Delta_r)}
 \quad\text{almost everywhere on }\Delta_r.
\]
The uniform \(B_{p_0}\)-estimate for the right-hand side therefore gives
\eqref{eq:nonsymmetric-local-Bp}.
\end{proof}

We next isolate the large-scale part of the argument.  It applies to
every exponent \(p>1\) and does not require symmetry.

\begin{lem}
\label{lem:height-global-Bp}
Let \(1<p<\infty\), and assume that \(A\) is real, bounded, and
uniformly elliptic.  Suppose that there is a positive adjoint solution
\(\Phi^*\), locally H\"older continuous up to the flat boundary, such
that
\begin{equation}\label{eq:abstract-height-coordinate}
 \cH^*\Phi^*=0\quad\text{in }\R^{n+2}_+,
 \qquad
 \Phi^*=0\quad\text{on }\{\lambda=0\},
\end{equation}
and
\begin{equation}\label{eq:abstract-height-comparison}
 c\lambda\leq\Phi^*(\lambda,x,t)\leq C\lambda,
 \qquad \lambda\geq1.
\end{equation}
Assume that, for some \(M>4\), the estimate
\begin{equation}\label{eq:abstract-local-Bp}
 \left(
 \frac1{|\Delta_r|}
 \iint_{\Delta_r}(k^P)^p\,
 \mathrm d x\,\mathrm d t
 \right)^{1/p}
 \leq
 C_p\,\frac{\omega_{\cH}^P(\Delta_r)}{|\Delta_r|}
\end{equation}
holds whenever \(0<r\leq1\) and \(P\) is an \(M\)-admissible pole.
After increasing \(M\) by a structural factor, the same estimate holds
for every \(r>0\), with a constant depending additionally on the
constants in \eqref{eq:abstract-height-comparison}.
\end{lem}

\begin{proof}
By Lemma~\ref{lem:cube-conventions}, it is enough to consider backward
cubes.  Let \(\Delta=\Delta_r^-(z_0)\), and let \(P\) be an
\(M\)-admissible pole.  Suppose first that \(r\geq R_0\), where \(R_0\)
is a sufficiently large structural constant.  Cover \(\Delta\) by unit
backward cubes \(\Delta_k=\Delta_1^-(z_k)\), where
\(z_k=(x_k,t_k)\), whose fixed enlargements have bounded
overlap.  Their number \(N\) satisfies
\begin{equation}\label{eq:nonsymmetric-number-unit-cubes}
 N\lesssim r^{n+2}.
\end{equation}
After the fixed adjustment of \(M\) in the statement,
\eqref{eq:admissible-pole} and the location of \(z_k\) give
\[
 t(P)-t_k\simeq\lambda(P)^2,
 \qquad
 |x(P)-x_k|+\lambda(P)
 \lesssim\lambda(P)
 \lesssim\bigl(t(P)-t_k\bigr)^{1/2},
\]
and \(t(P)-t_k\gtrsim M^2r^2\).  Let
\(P_k=P_{\Delta_1(z_k)}^{\mathrm{far}}\) be supplied by
Lemma~\ref{lem:universal-subcube-pole}.  Thus \(P\) and \(P_k\) lie in
a common forward pole region at scale one.  By
Lemma~\ref{lem:cube-conventions}, the local estimate applies to
\(\Delta_k\) at \(P_k\), and Lemma~\ref{lem:change-of-pole} gives it
for \(P\), uniformly in \(r\) and \(k\).  This estimate and
Lemma~\ref{lem:cfms}, together with
Lemma~\ref{lem:cube-conventions}, boundary doubling, and
Lemma~\ref{lem:green-boundary-time-comparison} for the chosen corkscrew
normalizations, give
\begin{equation}\label{eq:nonsymmetric-local-unit-Bp}
 \iint_{\Delta_k}(k^P)^p\,
 \mathrm d x\,\mathrm d t
 \lesssim
 (\omega_{\cH}^P(\Delta_k))^p
 \lesssim w(A_k)^p,
\end{equation}
where \(w(Z)=G_{\cH}(P;Z)\) and \(A_k\) is a future-oriented unit
corkscrew for \(\Delta_k\).  Put
\[
 \widehat z=(x_0,t_0-r^2/2),
 \qquad
 \rho=r/4,
 \qquad
 A_r^\circ=A_\rho^+(\widehat z).
\]
Then \(\lambda(A_r^\circ)=r\) and
\(\Delta_{\rho/2}(\widehat z)\subset\Delta\).  After increasing \(M\)
by a fixed amount, the left-hand inequality in
\eqref{eq:cfms-forward-raw} gives
\begin{equation}\label{eq:nonsymmetric-reference-green-measure}
 w(A_r^\circ)
 \lesssim
 r^{-(n+1)}\omega_{\cH}^P(\Delta).
\end{equation}
All the points \(A_k\) and \(A_r^\circ\) lie in
\(T_{L_0r}(x_0,t_0-r^2/4)\) for a fixed \(L_0\).  Choosing a sufficiently
large fixed multiple of this box produces one flat-boundary comparison
cylinder whose reduced region contains all these points and whose fixed
enlargement lies quantitatively in the causal past of \(P\).  This is
the same comparison-cylinder construction used before
Lemma~\ref{lem:two-comparison}.  The structural constants in
Lemma~\ref{lem:boundary-comparison} determine the required multiple.
In this cylinder, \(w\) and \(\Phi^*\) are positive adjoint solutions
vanishing on the same flat boundary portion.  Their reference balances
are uniform by Lemma~\ref{lem:green-boundary-time-comparison} and
\eqref{eq:abstract-height-comparison}.  Normalized adjoint boundary
comparison therefore gives
\[
 \frac{w(A_k)}{\Phi^*(A_k)}
 \lesssim
 \frac{w(A_r^\circ)}{\Phi^*(A_r^\circ)}.
\]
Our normalization gives \(\lambda(A_k)=4\), while
\(\lambda(A_r^\circ)=r\).  Hence
\eqref{eq:abstract-height-comparison} gives
\(\Phi^*(A_k)\simeq1\) and \(\Phi^*(A_r^\circ)\simeq r\), so
\begin{equation}\label{eq:nonsymmetric-unit-large-green}
 w(A_k)\lesssim r^{-1}w(A_r^\circ).
\end{equation}

Combining \eqref{eq:nonsymmetric-number-unit-cubes},
\eqref{eq:nonsymmetric-local-unit-Bp},
\eqref{eq:nonsymmetric-reference-green-measure}, and
\eqref{eq:nonsymmetric-unit-large-green}, we obtain
\begin{align}
 \iint_\Delta(k^P)^p\,\mathrm d x\,\mathrm d t
 &\lesssim
 \sum_k w(A_k)^p
 \lesssim
 r^{n+2-p}w(A_r^\circ)^p\notag\\
 &\lesssim
 r^{n+2-p-p(n+1)}(\omega_{\cH}^P(\Delta))^p\notag\\
 &=
 r^{(n+2)(1-p)}(\omega_{\cH}^P(\Delta))^p
 \simeq
 |\Delta|^{1-p}(\omega_{\cH}^P(\Delta))^p.
 \label{eq:nonsymmetric-global-Bp-estimate}
\end{align}
This is the \(p\)-th power of \eqref{eq:abstract-local-Bp}.  If
\(1<r<R_0\), cover \(\Delta\) by a structural number of backward
cubes \(\Delta_j=\Delta_{\rho_j}^-(z_j)\), with
\(\rho_j\in[1/2,1]\).  For each \(j\), let
\(P_j=P_{\Delta_{\rho_j}(z_j)}^{\mathrm{far}}\).  As above, after the
fixed adjustment of \(M\), the poles \(P\) and \(P_j\) lie in one
common forward pole region at scale \(\rho_j\).  By
Lemma~\ref{lem:cube-conventions}, the local estimate applies to
\(\Delta_j\) at \(P_j\), and Lemma~\ref{lem:change-of-pole} gives it for
\(P\), uniformly in \(j\).  Bounded overlap and the inequality
\(\sum_j a_j^p\leq(\sum_j a_j)^p\) give
\[
 \iint_\Delta(k^P)^p\,\mathrm d x\,\mathrm d t
 \lesssim
 (\omega_{\cH}^P(C\Delta))^p.
\]
Boundary doubling and change of pole control the right-hand side by
\(C\omega_{\cH}^P(\Delta)^p\).  Since \(1<r<R_0\), this is again
\eqref{eq:nonsymmetric-global-Bp-estimate}.  The range \(0<r\leq1\)
is the hypothesis.
\end{proof}

\begin{thm}
\label{thm:nonsymmetric-global-Ainfty}
Let \(A\) be real, bounded, measurable, and uniformly elliptic, but not
necessarily symmetric.  Assume the full spatial periodicity
\eqref{eq:full-periodicity}, and suppose that \(A\) has a real, bounded,
uniformly elliptic boundary trace
\(A^0=A^0(x,t)\) satisfying
\eqref{eq:nonsymmetric-square-Dini-finite}.  Then
\begin{equation}\label{eq:nonsymmetric-global-Ainfty}
 \omega_{\cH}\in A_\infty(\mathrm d x\,\mathrm d t).
\end{equation}
More precisely, there are \(p_0>1\), \(M>4\), and \(C<\infty\) such
that, for every \((x_0,t_0)\in\R^n\times\R\), every \(r>0\), and every
\(M\)-admissible pole \(P\) for \(\Delta_r(x_0,t_0)\),
\begin{equation}\label{eq:nonsymmetric-global-Bp}
 \left(
 \frac1{|\Delta_r|}
 \iint_{\Delta_r(x_0,t_0)}(k^P)^{p_0}\,
 \mathrm d x\,\mathrm d t
 \right)^{1/p_0}
 \leq
 C\,\frac{\omega_{\cH}^P(\Delta_r(x_0,t_0))}{|\Delta_r|}.
\end{equation}
The constants depend only on \(n,\mu,\Lambda,\mathfrak D_0\) and the
fixed geometric parameters.  No symmetry, periodicity in time, or
temporal regularity is required.
\end{thm}

\begin{proof}
Lemma~\ref{lem:nonsymmetric-local-Ainfty} proves
\eqref{eq:nonsymmetric-global-Bp} for \(0<r\leq1\).  The construction
in Section~\ref{sec:height-coordinate} does not use symmetry.  Its cell
problem is formulated with \(A^{\mathsf T}\), so
Lemma~\ref{lem:height-coordinate} applies directly to \(A\) and gives
the adjoint height coordinate required in
Lemma~\ref{lem:height-global-Bp}.  That lemma yields
\eqref{eq:nonsymmetric-global-Bp} at every scale for \(M\)-admissible
poles.  To recover the canonical poles in
Definition~\ref{defn:parabolic-Bp}, fix
\(\Delta_0=\Delta_{r_0}(x_0,t_0)\) and let
\(P_0=A_{4r_0}^+(x_0,t_0)\).  Let
\(P=P_{\Delta_0}^{\mathrm{far}}\) be supplied by
Lemma~\ref{lem:universal-subcube-pole}.  Thus \(P\) is an
\(M\)-admissible pole for every subcube of \(\Delta_0\), with uniform
constants.  The coordinate formulas for \(P\) and \(P_0\) place both in
a common forward pole region at scale \(r_0\).  The estimate just proved therefore holds for
\(\omega_{\cH}^P\) on every \(\Delta\subset\Delta_0\).  Applied at scale
\(r_0\), Lemma~\ref{lem:change-of-pole} and Lebesgue differentiation give
\[
 \frac{k^{P_0}}{\omega_{\cH}^{P_0}(\Delta_0)}
 \simeq
 \frac{k^P}{\omega_{\cH}^P(\Delta_0)}
 \quad\text{almost everywhere on }\Delta_0.
\]
Consequently, for every \(\Delta\subset\Delta_0\),
\begin{align*}
 \left(
 \frac1{|\Delta|}\iint_\Delta(k^{P_0})^{p_0}
 \,\mathrm d x\,\mathrm d t
 \right)^{1/p_0}
 &\lesssim
 \frac{\omega_{\cH}^{P_0}(\Delta_0)}
      {\omega_{\cH}^{P}(\Delta_0)}
 \left(
 \frac1{|\Delta|}\iint_\Delta(k^P)^{p_0}
 \,\mathrm d x\,\mathrm d t
 \right)^{1/p_0}\\
 &\lesssim
 \frac{\omega_{\cH}^{P_0}(\Delta_0)}
      {\omega_{\cH}^{P}(\Delta_0)}
 \frac{\omega_{\cH}^{P}(\Delta)}{|\Delta|}
 \lesssim
 \frac{\omega_{\cH}^{P_0}(\Delta)}{|\Delta|}.
\end{align*}
Thus
\(\omega_{\cH}\in B_{p_0}(\mathrm d x\,\mathrm d t)\), and
Proposition~\ref{prop:Ainfty-union-Bp} gives
\eqref{eq:nonsymmetric-global-Ainfty}.
\end{proof}

\begin{rem}
The transpose in the construction of \(\Phi^*\) is essential, but
symmetry is not.  The square-Dini convergence to \(A^0\) gives local
\(A_\infty\), while full spatial periodicity is used only to construct
\(\Phi^*\) and globalize the estimate.  The localization in
\eqref{eq:nonsymmetric-localized-coefficient} is needed because a
nonconstant transversely periodic discrepancy generally has infinite
Carleson norm on the full half-space.
\end{rem}

\begin{rem}
By Theorem~\ref{thm:Bp-Dq},
Theorem~\ref{thm:nonsymmetric-global-Ainfty} yields solvability of
\(D_q\) for \(q=p_0'=p_0/(p_0-1)\).  No uniqueness conclusion is
asserted.  A standard sufficient route to uniqueness is to establish
the corresponding estimate for the adjoint parabolic measure.
\end{rem}

\begin{rem}\label{rem:ellipticresult}
The preceding argument also has an elliptic consequence.  Consider
\[
 \mathcal L
 =
 -\operatorname{div}_{\lambda,x}
 \bigl(A(\lambda,x)\nabla_{\lambda,x}\bigr)
 \quad\text{in }\mathbb R^{n+1}_+,
\]
where \(A\) is real, bounded, uniformly elliptic, not necessarily
symmetric, and periodic in every spatial variable
\(X=(\lambda,x)\).  Suppose that \(A\) has a transversely independent
boundary trace \(A^0=A^0(x)\) and that
\[
 \int_0^1
 \left(
 \operatorname*{ess\,sup}_{0<\lambda\leq\rho}
 \bigl\|A(\lambda,\cdot)-A^0\bigr\|_{\mathrm L^\infty(\mathbb R^n)}
 \right)^2
 \frac{\mathrm d\rho}{\rho}
 <\infty.
\]
Then our argument gives
\[
 \omega_{\mathcal L}\in A_\infty(\mathrm d x).
\]
Indeed, the result of
Hofmann, Kenig, Mayboroda, and Pipher
\cite{HofmannKenigMayborodaPipher} gives the local
\(A_\infty\)-estimate for the transversely independent reference
operator with coefficient \(A^0(x)\).  Finite-Carleson perturbation
transfers this estimate to \(A\) below the period scale.  Full spatial
periodicity then produces a positive adjoint elliptic height coordinate,
and the preceding large-scale argument propagates the estimate to every
scale.  The result in
\cite{HofmannKenigMayborodaPipher} allows nonsymmetric coefficients but
assumes independence of the transverse variable, whereas
\cite{KenigShen} treats periodic transverse dependence under a symmetry
hypothesis.  Recent related results include the work of David, Gloria,
Mayboroda, and Qi~\cite{DavidGloriaMayborodaQi}, which treats symmetric
coefficients whose increasingly fine periodic structure is imposed
locally on Whitney boxes, and the work of Shen and
Zhuge~\cite{ShenZhuge}, which establishes large-scale harmonic-measure
and non-tangential maximal-function estimates for symmetric periodic
homogenization families in bounded Lipschitz and \(C^1\) domains.  These
settings are distinct from the fixed-period, nonsymmetric conclusion
considered here.  Thus, to the best of our knowledge, the nonsymmetric
transversely dependent conclusion above appears to be new even in the
elliptic setting.  We emphasize that the conclusion is
\(A_\infty\), or equivalently \(B_p\) for some \(p>1\), rather than the
specific \(B_2\)-estimate obtained in the symmetric case.
\end{rem}

\subsection{Time-independent Lipschitz graph domains}
\label{subsec:Lipschitz-graph-domains}

We next explain how the half-space results transfer to time-independent
Lipschitz graph domains and identify the compatibility needed to
preserve full spatial periodicity.  Let
\[
 \Omega_\psi
 =
 \{(x,z,t)\in\R^n\times\R\times\R:
 z>\psi(x)\},
 \qquad
 \psi\in W^{1,\infty}(\R^n),
\]
and consider
\[
 \cH_B U
 =
 \partial_tU-\operatorname{div}_{x,z}
 \bigl(B(x,z,t)\nabla_{x,z}U\bigr).
\]
Here \(B\) is real, bounded, measurable, and uniformly elliptic, with ellipticity
constants \(\mu\) and \(\Lambda\).
Set
\[
 \rho(\lambda,x,t)
 =
 \bigl(x,\lambda+\psi(x),t\bigr),
 \qquad
 \widetilde B(\lambda,x,t)
 =
 B\bigl(x,\lambda+\psi(x),t\bigr),
\]
and let \(u=U\circ\rho\).  With the gradients ordered as
\(\nabla_{x,z}=(\nabla_x,\partial_z)^{\mathsf T}\) and
\(\nabla_{\lambda,x}=(\partial_\lambda,\nabla_x)^{\mathsf T}\),
\[
 \bigl(\nabla_{x,z}U\bigr)\circ\rho
 =
 C_\psi\nabla_{\lambda,x}u,
 \qquad
 C_\psi(x)
 =
 \begin{pmatrix}
  -\nabla\psi(x)&I_n\\
  1&0
 \end{pmatrix}.
\]
The spatial Jacobian of \(\rho\) has absolute determinant one.  Since
\(\rho\) is independent of \(t\), the weak pull-back equation is
\begin{equation}\label{eq:flattened-coefficient}
 \cH_A u
 =
 \partial_tu-\operatorname{div}_{\lambda,x}
 \bigl(A\nabla_{\lambda,x}u\bigr)=0,
 \qquad
 A=C_\psi^{\mathsf T}\widetilde B C_\psi.
\end{equation}
Thus no drift or lower-order term is introduced.  For a time-dependent
graph, the formal change of variables contains the additional term
\(-\partial_t\psi\,\partial_\lambda u\), provided the graph has enough
temporal regularity for this calculation.

Put \(L=\|\nabla\psi\|_{\mathrm L^\infty(\R^n)}\).  Since
\(|C_\psi\xi|\simeq_{n,L}|\xi|\), reality, boundedness, and uniform
ellipticity pass from \(B\) to \(A\), with constants depending on
\(n,L,\mu,\Lambda\).  Moreover, \(C_\psi\) is invertible and hence
\[
 A=A^{\mathsf T}
 \quad\Longleftrightarrow\quad
 B=B^{\mathsf T}.
\]

Suppose that there is a bounded matrix \(B^\partial=B^\partial(x,t)\)
such that
\[
 \lim_{\rho\downarrow0}
 \operatorname*{ess\,sup}_{\substack{x\in\R^n,\ t\in\R\\
 0<\lambda\leq\rho}}
 \bigl|\widetilde B(\lambda,x,t)-B^\partial(x,t)\bigr|=0.
\]
Thus \(B^\partial\) is the uniform essential boundary trace of
\(\widetilde B\), in the sense of
\eqref{eq:uniform-essential-trace}.  The corresponding trace of the
flattened coefficient is
\[
 A^0(x,t)
 =
 C_\psi(x)^{\mathsf T}B^\partial(x,t)C_\psi(x).
\]
For a uniformly continuous representative of \(B^\partial\), set
\[
 \varpi_B^\partial(\rho)
 =
 \sup_{\substack{s,t\in\R\\|t-s|^{1/2}\leq\rho}}
 \bigl\|B^\partial(\cdot,t)-B^\partial(\cdot,s)\bigr\|_{
 \mathrm L^\infty(\R^n)}
\]
and
\[
 \eta_B^\psi(\rho)
 =
 \operatorname*{ess\,sup}_{\substack{x\in\R^n,\ t\in\R\\
 0<\lambda\leq\rho}}
 \bigl|
 \widetilde B(\lambda,x,t)-B^\partial(x,t)
 \bigr|.
\]
The moduli of the flattened coefficient satisfy
\[
 \varpi(\rho)\simeq_{n,L}\varpi_B^\partial(\rho),
 \qquad
 \eta(\rho)\simeq_{n,L}\eta_B^\psi(\rho).
\]
Thus the temporal hypothesis in Theorem~\ref{thm:parabolic-KS} is
equivalent to
\(\varpi_B^\partial(\rho)\to0\) as \(\rho\downarrow0\), and the
transverse square-Dini condition is equivalent to
\begin{equation}\label{eq:graph-square-Dini}
 \int_0^1
 \bigl(\eta_B^\psi(\rho)\bigr)^2\,
 \frac{\mathrm d \rho}{\rho}<\infty.
\end{equation}
The condition
\[
 \operatorname*{ess\,sup}_{x\in\R^n}
 \bigl\|D_t^{1/2}B^\partial(x,\cdot)\bigr\|_{
 \mathrm{BMO}(\R_t)}<\infty
\]
is sufficient for the required temporal continuity.  Indeed,
\(D_t^{1/2}A^0=C_\psi^{\mathsf T}(D_t^{1/2}B^\partial)C_\psi\), and
Lemma~\ref{lem:temporal-freezing} applies.  This condition is automatic
when \(B\) is independent of \(t\).

For the global results, assume that \(B\) admits a real, bounded,
uniformly elliptic extension to all \(z\in\R\).  If this extension is
one-periodic in \(z\), then the flattened coefficient is one-periodic in
\(\lambda\):
\[
 B(x,z+1,t)=B(x,z,t)
 \quad\Longrightarrow\quad
 A(\lambda+1,x,t)=A(\lambda,x,t).
\]
Full spatial periodicity of \(A\) is equivalent to
\begin{equation}\label{eq:graph-periodicity-compatibility}
 \begin{aligned}
 &C_\psi(x+k)^{\mathsf T}
 B\bigl(x+k,\psi(x+k)+\lambda+j,t\bigr)C_\psi(x+k)\\
 &\hspace{2cm}=
 C_\psi(x)^{\mathsf T}
 B\bigl(x,\psi(x)+\lambda,t\bigr)C_\psi(x)
 \end{aligned}
\end{equation}
for every \(k\in\mathbb Z^n\) and \(j\in\mathbb Z\), and almost every
\((\lambda,x,t)\).  A convenient sufficient condition is
\[
 \psi(x+k)=\psi(x)+\ell\cdot k,
 \qquad
 B(x+k,z+\ell\cdot k+j,t)=B(x,z,t),
\]
for some \(\ell\in\R^n\), every \(k\in\mathbb Z^n\), and every
\(j\in\mathbb Z\).  The first identity implies
\(\nabla\psi(x+k)=\nabla\psi(x)\), and hence
\(C_\psi(x+k)=C_\psi(x)\), for almost every \(x\).  These two conditions
include periodic graphs when \(\ell=0\) and affine-periodic graphs
\(\psi(x)=\ell\cdot x+\varphi(x)\) when \(\varphi\) is periodic.

The compatibility condition is not automatic.  For \(B=I\),
\[
 A(\lambda,x,t)
 =
 \begin{pmatrix}
  1+|\nabla\psi(x)|^2&-\nabla\psi(x)^{\mathsf T}\\
  -\nabla\psi(x)&I_n
 \end{pmatrix},
\]
which is tangentially periodic if and only if \(\nabla\psi\) is
one-periodic in every coordinate.

Finally, with \(J_\psi=(1+|\nabla\psi|^2)^{1/2}\),
\[
 \mathrm d \sigma(x)\,\mathrm d t
 =J_\psi(x)\,\mathrm d x\,\mathrm d t.
\]
Let \(K_B^{\rho(P)}\) denote the Poisson kernel in \(\Omega_\psi\), with
respect to \(\mathrm d \sigma\,\mathrm d t\).  The change of variables
then gives
\[
 k_A^P(x,t)
 =K_B^{\rho(P)}\bigl(x,\psi(x),t\bigr)J_\psi(x).
\]
The boundary map is parabolically bi-Lipschitz and preserves
corkscrews, Harnack chains, and pole admissibility up to constants
depending on \(n\) and \(L\).  We therefore obtain the following
consequence.

\begin{cor}\label{cor:Lipschitz-graph-transfer}
Let \(\Omega_\psi\), \(B\), \(B^\partial\), and \(L\) be as above.
Assume that \(B\) admits a real, bounded, uniformly elliptic extension to
all \(z\in\R\), that \eqref{eq:graph-square-Dini} holds, and that the
compatibility condition \eqref{eq:graph-periodicity-compatibility} is
satisfied.
\begin{enumerate}[label=\textup{(\roman*)},leftmargin=2.5em]
\item If \(B=B^{\mathsf T}\) and
\[
 \varpi_B^\partial(\rho)\longrightarrow0
 \qquad\text{as }\rho\downarrow0,
\]
then the parabolic measure associated with \(\cH_B\) belongs to
\(B_2(\mathrm d\sigma\,\mathrm d t)\).  Consequently, the
\(\mathrm L^2\)-Dirichlet problem is uniquely solvable among weak
solutions whose graph non-tangential maximal function belongs to
\(\mathrm L^2(\mathrm d\sigma\,\mathrm d t)\).
\item Without assuming symmetry or temporal regularity beyond
measurability, the parabolic measure associated with \(\cH_B\) belongs
to \(A_\infty(\mathrm d\sigma\,\mathrm d t)\).  More precisely, it
belongs to \(B_{p_0}(\mathrm d\sigma\,\mathrm d t)\) for some \(p_0>1\),
and the Dirichlet problem is solvable in
\(\mathrm L^{p_0'}(\mathrm d\sigma\,\mathrm d t)\).
\end{enumerate}
In \textup{(i)}, the constants depend only on
\(n,L,\mu,\Lambda\), the fixed geometric parameters,
\(\varpi_B^\partial\), and the square-Dini norm in
\eqref{eq:graph-square-Dini}.  In \textup{(ii)}, \(p_0\) and the
constants have the same dependence, except that no temporal modulus
enters.
\end{cor}

\subsection{An abstract criterion for the height coordinate}
\label{subsec:minimal-height-hypothesis}

Full spatial periodicity is a convenient sufficient condition for the
height-coordinate construction, but it is not minimal.  Under transverse
periodicity, the relevant property has equivalent half-space and
whole-space formulations.

\begin{prop}\label{prop:minimal-height-hypothesis}
Assume \eqref{eq:ellipticity} and \eqref{eq:periodicity}, and extend
\(A\) one-periodically in \(\lambda\) to \(\R^{n+2}\).  The
following statements are equivalent.  All solutions in
\textup{(i)}-\textup{(iii)} are understood to be real-valued.
\begin{enumerate}[label=\textup{(\roman*)},leftmargin=2.5em]
\item There is a positive weak local-energy solution \(\Phi^*\) of
\(\cH^*\Phi^*=0\) in \(\R^{n+2}_+\) such that, for every \(R,T>0\),
\[
 \Phi^*\in
 \mathrm L^2\bigl((-T,T);
 \mathrm H^1((0,R)\times B(0,R))\bigr),
 \quad
 \partial_t\Phi^*\in
 \mathrm L^2\bigl((-T,T);
 \mathrm H^{-1}((0,R)\times B(0,R))\bigr),
\]
and whose Sobolev trace satisfies
\begin{equation}\label{eq:abstract-half-space-height}
 \operatorname{Tr}_0\Phi^*=0
 \quad\text{in }
 \mathrm L^2\bigl((-T,T);\mathrm H^{1/2}(B(0,R))\bigr),
 \qquad
 \|\Phi^*-\lambda\|_{\mathrm L^\infty(\R^{n+2}_+)}<\infty.
\end{equation}

\item There is a complete whole-space adjoint solution
\[
 p^*\in
 \mathrm L^2_{\mathrm{loc}}
 \bigl(\R_t;
 \mathrm H^1_{\mathrm{loc}}(\R_X^{n+1})\bigr)
\]
such that
\begin{equation}\label{eq:abstract-bounded-height}
 \cH^*p^*=0\quad\text{in }\R^{n+2},
 \qquad
 \|p^*-\lambda\|_{\mathrm L^\infty(\R^{n+2})}<\infty.
\end{equation}

\item There is a bounded complete weak solution \(\chi^*\), one-periodic
in \(\lambda\) in the sense that
\[
 \chi^*(\lambda+1,x,t)=\chi^*(\lambda,x,t)
 \quad\text{for almost every }(\lambda,x,t),
\]
such that
\[
 \chi^*\in
 \mathrm L^2_{\mathrm{loc}}
 \bigl(\R_t;\mathrm H^1_{\mathrm{loc}}(\R_X^{n+1})\bigr),
 \qquad
 \partial_t\chi^*\in
 \mathrm L^2_{\mathrm{loc}}
 \bigl(\R_t;\mathrm H^{-1}_{\mathrm{loc}}(\R_X^{n+1})\bigr),
\]
and
\begin{equation}\label{eq:abstract-bounded-corrector}
 -\partial_t\chi^*
 -\operatorname{div}_X\!\left(
 A^{\mathsf T}(X,t)
 \bigl(e_\lambda+\nabla_X\chi^*\bigr)
 \right)=0
 \quad\text{in }\R^{n+2}.
\end{equation}
\end{enumerate}
Here ``complete'' means defined for every \(t\in\R\), with no
initial or terminal time.

\noindent These equivalent conditions supply the
height-function input required in the large-scale reduction.  Every
height coordinate \(\Phi^*\) satisfying \textup{(i)} also satisfies
\begin{equation}\label{eq:abstract-height-comparison-minimal}
 c\lambda\leq\Phi^*(\lambda,x,t)\leq C\lambda,
 \qquad \lambda\geq1,
\end{equation}
where the constants depend only on \(n,\mu,\Lambda\) and quantitatively
on
\[
 K_\Phi
 =\|\Phi^*-\lambda\|_{\mathrm L^\infty(\R^{n+2}_+)}.
\]
\end{prop}

\begin{proof}
Assume \textup{(iii)} and set \(p^*=\lambda+\chi^*\).  Then
\eqref{eq:abstract-bounded-corrector} gives \textup{(ii)}.  Conversely, assume \textup{(ii)}, put \(b^*=p^*-\lambda\), and let
\(K=\|b^*\|_{\mathrm L^\infty(\R^{n+2})}\).  For \(L\in\mathbb N\), define
\[
 q_L^*(\lambda,x,t)
 =\frac1L\sum_{j=0}^{L-1}
 \bigl(p^*(\lambda+j,x,t)-j\bigr).
\]
Transverse periodicity of \(A\) implies that every summand, and hence \(q_L^*\),
solves the adjoint equation.  Moreover,
\[
 \|q_L^*-\lambda\|_\infty\leq K
\]
and
\[
 q_L^*(\lambda+1,x,t)-q_L^*(\lambda,x,t)-1
 =\frac{b^*(\lambda+L,x,t)-b^*(\lambda,x,t)}{L}.
\]
On every compact parabolic cylinder, Caccioppoli's inequality and the
equation give uniform bounds for \(q_L^*\) in
\(\mathrm L^2_t\mathrm H^1_X\) and for \(\partial_tq_L^*\) in
\(\mathrm L^2_t\mathrm H^{-1}_X\).  The local Aubin-Lions lemma and a
diagonal extraction give an adjoint solution \(q^*\) such that, along a
subsequence, \(q_L^*\) converges weakly to \(q^*\) in the local energy space and
strongly in \(\mathrm L^2_{\mathrm{loc}}\).  After passing to a further
subsequence, the convergence also holds almost everywhere.  Hence
\(\|q^*-\lambda\|_\infty\leq K\).  The right-hand side in the preceding
identity is bounded by \(2K/L\), and therefore
\[
 q^*(\lambda+1,x,t)=q^*(\lambda,x,t)+1
\]
almost everywhere.  Thus \(\chi^*=q^*-\lambda\) is bounded,
one-periodic in \(\lambda\), and satisfies
\eqref{eq:abstract-bounded-corrector}.  This proves \textup{(iii)}.

Assume next \textup{(i)} and set, for \(m\in\mathbb N\),
\[
 p_m^*(\lambda,x,t)
 =\Phi^*(\lambda+m,x,t)-m,
 \qquad \lambda>-m.
\]
The functions \(p_m^*\) solve the same adjoint equation and satisfy
\[
 \|p_m^*-\lambda\|_\infty
 \leq\|\Phi^*-\lambda\|_\infty.
\]
The same local compactness argument on compact subsets of
\(\R^{n+2}\) gives a diagonal limit \(p^*\).  Since the boundary
\(\{\lambda=-m\}\) leaves every compact set as \(m\to\infty\), the
limit is a complete whole-space adjoint solution.  Almost-everywhere
convergence along a further subsequence preserves the displayed
\(\mathrm L^\infty\)-bound, so the limit satisfies \textup{(ii)}.

It remains to prove that \textup{(ii)} implies \textup{(i)}.  Put
\(K=\|p^*-\lambda\|_\infty\), reverse time, and write
\[
 \widetilde A(X,\tau)=A^{\mathsf T}(X,-\tau),
 \qquad
 q(X,\tau)=p^*(X,-\tau).
\]
Then \(q\) solves the corresponding forward equation.  For
\(N,R,T>0\), set
\[
 D_{N,R}=(0,N)\times B(0,R),
\]
and let \(\Psi_{N,R,T}\) be the forward energy solution in
\(D_{N,R}\times(-T,\infty)\), with boundary value \(\lambda\) on
\(\partial D_{N,R}\times(-T,\infty)\) and initial value \(\lambda\) at
\(\tau=-T\).  Existence follows from standard variational theory after
subtracting \(\lambda\).  The inequalities
\[
 0\leq\lambda\leq N,
 \qquad
 q-K\leq\lambda\leq q+K
\]
hold on the lateral and initial boundaries in the trace sense.  The weak
comparison principle gives
\[
 0\leq\Psi_{N,R,T}\leq N,
 \qquad
 q-K\leq\Psi_{N,R,T}\leq q+K.
\]
In particular,
\[
 |\Psi_{N,R,T}-\lambda|\leq2K.
\]
For fixed \(N\), interior Caccioppoli estimates, and flat-boundary
Caccioppoli estimates applied to \(\Psi_{N,R,T}-\lambda\), together
with the equation, give bounds independent of \(R\) and \(T\) on bounded
cylinders, including cylinders meeting \(\lambda=0\) or \(\lambda=N\).
In particular, the approximants are uniformly bounded there in
\(\mathrm L^2(I;\mathrm H^1)\).  Letting \(T,R\to\infty\) through
diagonal subsequences and then reversing time produces a complete
adjoint solution \(\Phi_N^*\) in
\((0,N)\times\R^n\times\R\).  On every bounded cylinder touching a
transverse face, the approximants converge weakly in
\(\mathrm L^2(I;\mathrm H^1)\).  Continuity of the Sobolev trace operator
therefore carries their prescribed transverse traces to the limit.
Consequently,
\[
 \operatorname{Tr}_0\Phi_N^*=0,
 \qquad
 \operatorname{Tr}_N\Phi_N^*=N,
 \qquad
 0\leq\Phi_N^*\leq N,
 \qquad
 |\Phi_N^*-\lambda|\leq2K.
\]
A further diagonal extraction as \(N\to\infty\) gives a nonnegative
adjoint solution \(\Phi^*\) in the half-space such that
\[
 \operatorname{Tr}_0\Phi^*=0,
 \qquad
 |\Phi^*-\lambda|\leq2K.
\]
Since \(\Phi^*\geq\lambda-2K\), it is positive at every sufficiently
large transverse height.  Given an arbitrary interior point \(Z\), join
it by a correctly oriented adjoint Harnack chain to a point of
sufficiently large height occurring later in time.  Such a
chain is available because \(\Phi^*\) is defined for all \(t\in\R\).
The adjoint Harnack inequality bounds the value at the later point by a
constant times \(\Phi^*(Z)\), and hence \(\Phi^*(Z)>0\).  The
flat-boundary De Giorgi-Moser estimate, applied using the zero Sobolev
trace, gives a locally H\"older representative up to
\(\{\lambda=0\}\) that vanishes there continuously.  This proves
\textup{(i)}.

Finally, let \(\Phi^*\) be any height coordinate satisfying
\textup{(i)}, set
\[
 K=\|\Phi^*-\lambda\|_{\mathrm L^\infty(\R^{n+2}_+)},
 \qquad H_0=2K+2,
\]
and choose its locally H\"older representative.  The bounded-error
estimate gives
\[
 \Phi^*(\lambda,x,t)\leq(1+K)\lambda,
 \qquad \lambda\geq1,
\]
and
\[
 \Phi^*(\lambda,x,t)\geq\frac12\lambda,
 \qquad \lambda\geq H_0.
\]
If \(1\leq\lambda\leq H_0\), connect
\(Z=(\lambda,x,t)\) by an adjoint Harnack chain to
\[
 \widehat Z=(H_0,x,t+\theta_0H_0^2),
\]
where the fixed constant \(\theta_0>0\) is large enough to give the
correct adjoint orientation.  The number of cylinders in the chain is
bounded in terms of \(H_0\), while each cylinder has fixed relative
interior clearance.  Thus its Harnack constant depends only on
\(n,\mu,\Lambda\) and \(H_0\).  Since
\(\Phi^*(\widehat Z)\geq H_0-K\geq H_0/2\), adjoint Harnack gives a
lower bound \(\Phi^*(Z)\geq c_0\), where \(c_0>0\) depends only on
\(n,\mu,\Lambda\) and \(H_0\).  Since \(\lambda\leq H_0\), it follows
that \(\Phi^*(Z)\geq(c_0/H_0)\lambda\).  This proves
\eqref{eq:abstract-height-comparison-minimal} and completes the proof.
\end{proof}

\begin{rem}
Full spatial periodicity verifies the equivalent conditions in
Proposition~\ref{prop:minimal-height-hypothesis}.  Indeed,
Lemma~\ref{lem:adjoint-periodic-corrector} constructs the complete
corrector on the compact spatial torus.  Poincar\'e's inequality gives
the energy bound, and local parabolic boundedness gives the required
\(\mathrm L^\infty\)-bound.  No symmetry, time periodicity, or temporal
regularity is needed.
\end{rem}
\begin{rem}
If \(A(X,t+1)=A(X,t)\), uniqueness of the complete mean-zero corrector
implies \(\chi^*(X,t+1)=\chi^*(X,t)\).  The slab solutions and the
resulting height coordinate inherit the same periodicity.  Thus time
periodicity is inherited when present, but is not used in the
construction.
\end{rem}
\begin{rem}
There are other elementary ways to verify the height condition.  If
\begin{equation}\label{eq:height-divergence-free}
 \operatorname{div}_X(A^{\mathsf T}e_\lambda)=0
 \quad\text{in the sense of distributions},
\end{equation}
then \(\Phi^*=\lambda\).  This holds, in particular, when \(A=A(t)\).
More generally, independence of \(\lambda\) alone is insufficient,
since for \(A=A(x,t)\),
\[
 \cH^*\lambda
 =-\operatorname{div}_X(A^{\mathsf T}e_\lambda)
\]
need not vanish.  Hence a uniform local \(B_p\)-estimate globalizes
without spatial periodicity whenever
\eqref{eq:height-divergence-free} holds.  For symmetric \(A(t)\), this
applies to the local \(B_2\)-estimate of
Theorem~\ref{thm:parabolic-KS} whenever the temporal hypothesis of that
theorem holds.  For general real \(A(t)\), it applies to
the local \(A_\infty\)-estimate of \cite{AENDirichlet}.
\end{rem}
\begin{rem}
Transverse periodicity alone gives the commuting difference
\(u(\lambda+1,x,t)-u(\lambda,x,t)\), but it does not by itself produce
a bounded corrector.  Remark~\ref{rem:normal-periodicity-obstruction}
identifies the resulting loss in the direct representation and
trace-duality approaches.  The present criterion shows precisely what
must replace full spatial periodicity: an independently constructed
bounded complete corrector, or equivalently the adjoint height coordinate
in \eqref{eq:abstract-half-space-height}.
\end{rem}

\subsection{Uniform estimates and homogenization}
\label{subsec:uniform-homogenization}

We conclude by recording some consequences of the preceding results for
periodic homogenization families.  The main observation is that the
estimates proved in the paper are stable under the natural parabolic rescaling
and are therefore uniform in the microscopic scale.  They provide the
uniform boundary control needed in a homogenization argument of the type
developed in \cite[Section~8]{LitsgardNystrom}.  For brevity, we do not
give complete proofs of all the statements below, but restrict ourselves
to outlining the main ideas.

For the symmetric \(B_2\) statements below, we assume that \(A\)
satisfies the hypotheses of Theorem~\ref{thm:global}; for the
nonsymmetric \(A_\infty\) statements, we use the hypotheses of
Theorem~\ref{thm:nonsymmetric-global-Ainfty}.

We first retain the general time dependence considered throughout the
paper and introduce
\[
 A_\varepsilon(X,t)
 =
 A\bigl({X}/{\varepsilon},t\bigr),
 \qquad
 \cH_\varepsilon
 =
 \partial_t-\operatorname{div}_X
 \bigl(A_\varepsilon\nabla_X\bigr),
 \qquad
 0<\varepsilon\leq1.
\]
Thus all spatial variables, including the transverse variable
\(\lambda\), are rescaled.  Under the parabolic change of variables
\(X=\varepsilon Y\), \(t=\varepsilon^2s\), the normalized coefficient is
\[
 B_\varepsilon(Y,s)=A(Y,\varepsilon^2s).
\]
Its spatial period is one, and its ellipticity constants agree with
those of \(A\).

This change of variables is used only to normalize the microscopic
spatial scale when applying the boundary estimates and it should not be
interpreted as the passage to the homogenization limit.  Indeed, a
rescaling centered at an arbitrary time \(t_0\) gives
\[
 B_{\varepsilon,t_0}(Y,s)
 =
 A(Y,t_0+\varepsilon^2s).
\]
This normalization is performed separately at each \(t_0\) and does
not average the dependence on time.  Since the family
\(A(X/\varepsilon,t)\) contains no rapidly oscillating time variable,
its homogenized coefficient generally retains its dependence on \(t\).

The transverse modulus of \(B_\varepsilon\) satisfies
\(\eta_{B_\varepsilon}(\rho)=\eta(\rho)\).  Moreover, its boundary trace is
\[B_\varepsilon^0(y,s)=A^0(y,\varepsilon^2s),
\]
and hence its temporal modulus is
\[
 \varpi_{B_\varepsilon}(\rho)
 =
 \varpi(\varepsilon\rho)
 \leq
 \varpi(\rho),
 \qquad
 0<\varepsilon\leq1.
\]
Consequently, Theorem~\ref{thm:global} and
Corollary~\ref{cor:global-D2}, followed by scaling back to the original
variables, give estimates with constants independent of
\(\varepsilon\).  In particular, the parabolic measure
\(\omega_\varepsilon^P\) associated with \(\cH_\varepsilon\) satisfies
\[
 \left(
 \frac{1}{|\Delta|}
 \iint_{\Delta}(k_\varepsilon^P)^2
 \,\mathrm d x\,\mathrm d t
 \right)^{1/2}
 \leq
 \frac{C}{|\Delta|}
 \iint_{\Delta}k_\varepsilon^P
 \,\mathrm d x\,\mathrm d t
\]
for every surface cube \(\Delta\) and every corresponding
\(M\)-admissible pole \(P\), where \(C\) is independent of
\(\varepsilon\).  The Dirichlet problem for \(\cH_\varepsilon\) is
therefore uniquely solvable in \(\mathrm L^2\), with the uniform estimate
\[
 \|N_\ast u_\varepsilon\|_{\mathrm L^2(\partial\R^{n+2}_+)}
 \leq
 C\|f\|_{\mathrm L^2(\partial\R^{n+2}_+)},
\]
where \(C\) is independent of \(\varepsilon\).

The same scaling also describes the adjoint height coordinate.  Write
\(Y=(\lambda',y)\), let \(\Phi_{B_\varepsilon}^*\) denote the height
coordinate associated with \(B_\varepsilon\), and set
\[
 \Phi_\varepsilon^*(X,t)
 =
 \varepsilon
 \Phi_{B_\varepsilon}^*
 \bigl({X}/{\varepsilon},
       {t}/{\varepsilon^2}\bigr).
\]
Then \(\Phi_\varepsilon^*\) is the height coordinate associated with
\(\cH_\varepsilon^*\).  Since
\[
 \bigl|\Phi_{B_\varepsilon}^*(Y,s)-\lambda'\bigr|\leq C
\]
with \(C\) independent of \(\varepsilon\), it follows that
\[
 \bigl|\Phi_\varepsilon^*(X,t)-\lambda\bigr|
 =
 \varepsilon
 \left|
 \Phi_{B_\varepsilon}^*
 \bigl({X}/{\varepsilon},
       {t}/{\varepsilon^2}\bigr)
 -
 \frac{\lambda}{\varepsilon}
 \right|
 \leq C\varepsilon.
\]
Thus the deviation from the Euclidean height function is of microscopic
size.  More precisely,
\[
 \sup_{x,t}
 \left|
 \frac{\Phi_\varepsilon^*(\lambda,x,t)}{\lambda}-1
 \right|
 \leq
 C\frac{\varepsilon}{\lambda},
 \qquad \lambda>0.
\]
Equivalently,
\[
 \sup_{\substack{\lambda\geq R\varepsilon,x,t}}
 \left|
 \frac{\Phi_\varepsilon^*(\lambda,x,t)}{\lambda}-1
 \right|
 \leq
 \frac{C}{R},
\]
which tends to zero as \(R\to\infty\), uniformly in
\(\varepsilon\).

For almost every \(t\), let
\(\chi_j(\cdot,t)\), \(1\leq j\leq n+1\), be the mean-zero periodic
solution of
\[
 -\operatorname{div}_Y
 \left(
 A(Y,t)
 \bigl(e_j+\nabla_Y\chi_j(Y,t)\bigr)
 \right)
 =0
 \quad\text{in }\mathbb T^{n+1},
\]
and define
\[
 \widehat A(t)e_j
 =
 \int_{\mathbb T^{n+1}}
 A(Y,t)
 \bigl(e_j+\nabla_Y\chi_j(Y,t)\bigr)
 \,\mathrm dY.
\]
The family \(\cH_\varepsilon\) then homogenizes locally in the interior,
in the usual weak sense described in Appendix~\ref{Homoderv}, to
\[
 \widehat\cH
 =
 \partial_t-\operatorname{div}_X
 \bigl(\widehat A(t)\nabla_X\bigr).
\]
For completeness, a derivation of this qualitative limit and of the
cell formula for \(\widehat A(t)\) is given in
Appendix~\ref{Homoderv}.

Since only the spatial variables are rapidly oscillating, the time \(t\) acts as a parameter in the cell problem.  There is therefore
no temporal cell problem, and the time derivative remains
\(\partial_t\) in the limiting operator.  This also explains why the
effective matrix generally retains its dependence on \(t\).  In
particular, neither time periodicity nor differentiation of the
correctors with respect to \(t\) is needed to identify the limit.  See
\cite[Chapters~9 and~11]{CioranescuDonato} for the classical periodic
framework and \cite[Section~2.1]{GengNiuLocallyPeriodic} for a locally
periodic formulation with macroscopic time dependence.

The uniform \(\mathrm L^2\) boundary estimate established above,
together with standard local energy compactness, provides the
scale-independent boundary control needed to pass to the limit for
Dirichlet data, as in \cite[Section~8]{LitsgardNystrom}.  We do not
pursue convergence rates here.

We next consider simultaneous rescaling of the spatial and temporal
variables.  No periodicity in time is needed for the uniform boundary
estimates in this discussion.  For \(\alpha>0\), set
\[
 A_{\varepsilon,\alpha}(X,t)
 =
 A\bigl({X}/{\varepsilon},
        {t}/{\varepsilon^\alpha}\bigr),
 \qquad
 \cH_{\varepsilon,\alpha}
 =
 \partial_t-\operatorname{div}_X
 \bigl(A_{\varepsilon,\alpha}\nabla_X\bigr),\qquad
 0<\varepsilon\leq1.
\]
The exponent \(\alpha=2\) corresponds to the parabolic scaling.  Under
the change of variables \(X=\varepsilon Y\), \(t=\varepsilon^2s\), the normalized coefficient is
\[
 B_{\varepsilon,\alpha}(Y,s)
 =
 A\bigl(Y,\varepsilon^{2-\alpha}s\bigr).
\]
Its transverse modulus is unchanged, while the temporal modulus of its
boundary trace is
\[
 \varpi_{\varepsilon,\alpha}(\rho)
 =
 \varpi\bigl(\varepsilon^{1-\alpha/2}\rho\bigr).
\]
If \(0<\alpha\leq2\), then
 \[\varpi_{\varepsilon,\alpha}(\rho)
 \leq
 \varpi(\rho).
\]
It follows from Theorem~\ref{thm:global} and
Corollary~\ref{cor:global-D2} that the reverse H\"older and
\(\mathrm L^2\)-Dirichlet estimates for
\(\cH_{\varepsilon,\alpha}\) are uniform in \(\varepsilon\).  When \(\alpha>2\), the factor
\(\varepsilon^{1-\alpha/2}\) tends to infinity as
\(\varepsilon\downarrow0\).  The theorem still applies for each fixed
\(\varepsilon\), since
\[
 \varpi_{\varepsilon,\alpha}(\rho)\longrightarrow0
 \qquad\text{as }\rho\downarrow0,
\]
but the present temporal-modulus argument does not give a
\(B_2\)-constant that is uniform in \(\varepsilon\).

There is nevertheless a uniform conclusion for every \(\alpha>0\) at
the \(A_\infty\) level.  Indeed,
Theorem~\ref{thm:nonsymmetric-global-Ainfty} does not use a temporal
modulus, and the transverse square-Dini quantity is invariant under
the normalization above.  After scaling back, that theorem therefore
yields \(p_0>1\), \(M>4\), and \(C<\infty\), independent of
\(\varepsilon\) and \(\alpha\), such that, for every surface cube
\(\Delta_r\) and every corresponding \(M\)-admissible pole \(P\),
\(\omega_{\cH_{\varepsilon,\alpha}}^P\ll\mathrm d x\,\mathrm d t\)
on \(\Delta_r\), and, writing
\[
 k_{\varepsilon,\alpha}^P
 =
 \frac{\mathrm d\omega_{\cH_{\varepsilon,\alpha}}^P}
      {\mathrm d x\,\mathrm d t},
\]
one has
\[
 \left(
 \frac1{|\Delta_r|}
 \iint_{\Delta_r}(k_{\varepsilon,\alpha}^P)^{p_0}
 \,\mathrm d x\,\mathrm d t
 \right)^{1/p_0}
 \leq
 C\frac{\omega_{\cH_{\varepsilon,\alpha}}^P(\Delta_r)}
        {|\Delta_r|}.
\]
In particular, the corresponding parabolic measures satisfy a uniform
quantitative \(A_\infty\) estimate for every \(\alpha>0\).  For
\(\alpha>2\), this conclusion should not be confused with a uniform
\(B_2\) estimate.

To identify the homogenized operators, we now assume in addition that
\(A\) is one-periodic in its time variable, i.e.,
\[
 A(Y,s+1)=A(Y,s),
 \qquad
 (Y,s)\in\mathbb R^{n+1}\times\mathbb R.
\]
Consequently,
\[
 A_{\varepsilon,\alpha}(X,t+\varepsilon^\alpha)
 =
 A_{\varepsilon,\alpha}(X,t).
\]
The homogenized matrix depends on the relative rate of the spatial and
temporal oscillations.  The following three regimes are the parabolic non-self-similar regimes
studied systematically in \cite{GengShenNonselfsimilar}.  We also refer
to that paper and the references therein for further background and
related results on parabolic homogenization with non-self-similar
scales.  Write \(\mathbb T_s=\mathbb R/\mathbb Z\).

Suppose first that \(0<\alpha<2\).  In this regime spatial
homogenization occurs first, with the fast time variable regarded as a
parameter.  For each \(s\in\mathbb T_s\) and
\(1\leq j\leq n+1\), let \(\chi_j^{<}(\cdot,s)\) be the spatially
periodic solution of
\[
 -\operatorname{div}_Y
 \left(
 A(Y,s)
 \bigl(e_j+\nabla_Y\chi_j^{<}(Y,s)\bigr)
 \right)
 =0
 \quad\text{in }\mathbb T^{n+1},
\]
normalized by
\[
 \int_{\mathbb T^{n+1}}\chi_j^{<}(Y,s)\,\mathrm dY=0.
\]
The normalization gives uniqueness, and hence the time periodicity of
\(A\) implies
\[
 \chi_j^{<}(Y,s+1)=\chi_j^{<}(Y,s).
\]
The spatially homogenized matrix is then averaged over one time period:
\[
 \widehat A_{<}e_j
 =
 \int_0^1\int_{\mathbb T^{n+1}}
 A(Y,s)
 \bigl(e_j+\nabla_Y\chi_j^{<}(Y,s)\bigr)
 \,\mathrm dY\,\mathrm ds.
\]

At the critical exponent \(\alpha=2\), the spatial and temporal
oscillations interact in a single parabolic cell problem.  For
\(1\leq j\leq n+1\), let \(\chi_j^{(2)}\) belong to the periodic energy
class
\[
 \chi_j^{(2)}\in
 \mathrm L^2\bigl(\mathbb T_s;
 \mathrm H^1_{\mathrm{per}}(\mathbb T^{n+1})\bigr),
 \qquad
 \partial_s\chi_j^{(2)}\in
 \mathrm L^2\bigl(\mathbb T_s;
 \mathrm H^{-1}_{\mathrm{per}}(\mathbb T^{n+1})\bigr),
\]
and solve, in the weak periodic sense,
\[
 \partial_s\chi_j^{(2)}
 -
 \operatorname{div}_Y
 \left(
 A(Y,s)
 \bigl(e_j+\nabla_Y\chi_j^{(2)}\bigr)
 \right)
 =0
 \quad\text{in }\mathbb T^{n+1}\times\mathbb T_s,
\]
subject to
\[
 \chi_j^{(2)}(Y+k,s+\ell)
 =
 \chi_j^{(2)}(Y,s),
 \qquad
 (k,\ell)\in\mathbb Z^{n+1}\times\mathbb Z,
\]
and the normalization
\[
 \int_{\mathbb T^{n+1}}
 \chi_j^{(2)}(Y,s)\,\mathrm dY
 =0
 \quad\text{for almost every }s.
\]
The effective matrix is then given by
\[
 \widehat A_2e_j
 =
 \int_0^1\int_{\mathbb T^{n+1}}
 A(Y,s)
 \bigl(e_j+\nabla_Y\chi_j^{(2)}(Y,s)\bigr)
 \,\mathrm dY\,\mathrm ds.
\]

Finally, suppose that \(\alpha>2\).  The temporal oscillation is then
faster than the diffusive scale, and one first forms the temporal
average
\[
 A_{\mathrm{av}}(Y)
 =
 \int_0^1 A(Y,s)\,\mathrm ds.
\]
For \(1\leq j\leq n+1\), let \(\chi_j^{>}\) be the spatially periodic
solution of
\[
 -\operatorname{div}_Y
 \left(
 A_{\mathrm{av}}(Y)
 \bigl(e_j+\nabla_Y\chi_j^{>}(Y)\bigr)
 \right)
 =0
 \quad\text{in }\mathbb T^{n+1},
\]
normalized by
\[
 \int_{\mathbb T^{n+1}}\chi_j^{>}(Y)\,\mathrm dY=0.
\]
The effective matrix is
\[
 \widehat A_{>}e_j
 =
 \int_{\mathbb T^{n+1}}
 A_{\mathrm{av}}(Y)
 \bigl(e_j+\nabla_Y\chi_j^{>}(Y)\bigr)
 \,\mathrm dY.
\]

To summarize, the limiting operator is
\[
 \partial_t-\operatorname{div}_X
 \bigl(\widehat A_\alpha\nabla_X\bigr),
 \qquad
 \widehat A_\alpha
 =
 \begin{cases}
  \widehat A_{<},&0<\alpha<2,\\
  \widehat A_2,&\alpha=2,\\
  \widehat A_{>},&\alpha>2.
 \end{cases}
\]
The three effective matrices need not coincide.  Notice also that the
family \(A(X/\varepsilon,t)\) considered first is formally the case
\(\alpha=0\), but lies outside this trichotomy: its time dependence
remains macroscopic and is not averaged.  For every fixed \(\alpha>0\), the
temporal period tends to zero as \(\varepsilon\downarrow0\), whereas
for \(\alpha=0\) the physical time dependence remains macroscopic.

In summary, the critical exponent \(\alpha=2\) governs both sides of
the discussion.  It is the exponent at which the coupled parabolic
cell problem appears, and it is the endpoint up to which the present
temporal-modulus comparison proves a uniform \(B_2\) constant under
parabolic normalization.  Under the additional time-periodicity
hypothesis, qualitative homogenization holds in all three regimes.
The uniform quantitative \(A_\infty\) control obtained from the present
argument, however, does not require time periodicity and holds for every
\(\alpha>0\).

\appendix

\section{Smooth approximation and stability}
\label{sec:smooth-coefficient-reduction}

We record a smooth-approximation principle.  Its purpose is first to
justify the auxiliary constructions for smooth coefficients and then to
remove the smoothness assumption by compactness.  Smoothness is used only qualitatively.  In
particular, none of the estimates may depend on derivatives of the
approximating matrices.

\begin{lem}
\label{lem:smooth-coefficient-reduction}
Let \(A\) be real, bounded, uniformly elliptic, and satisfy the full
spatial periodicity condition \eqref{eq:full-periodicity}.  Then there are
matrices
\[
 A_j\in
 C^\infty\bigl(\mathbb R^{n+2};
 \mathbb R^{(n+1)\times(n+1)}\bigr)
\]
such that each \(A_j\) has the same spatial periods and ellipticity bounds
as \(A\), and
\begin{equation}\label{eq:smooth-approximation-local}
 A_j\longrightarrow A
 \quad\text{in }\mathrm L^p_{\mathrm{loc}}(\mathbb R^{n+2})
 \quad\text{for every }1\leq p<\infty.
\end{equation}
After passing to a subsequence, the convergence also holds almost
everywhere.  If \(A=A^{\mathsf T}\), then the approximating matrices may
be chosen symmetric.  Suppose, in addition, that \(A\) has the boundary trace \(A^0\) and
satisfies the temporal and transverse assumptions of
Subsection~\ref{subsec:small-assumptions}.  The sequence may then be
chosen so that, with
\[
 A_j^0(x,t)=A_j(0,x,t),
\]
the temporal modulus \(\varpi_j\) of \(A_j^0\) satisfies
\begin{equation}\label{eq:smooth-time-modulus}
 \varpi_j(\rho)\leq\varpi(\rho),
 \qquad \rho>0,
\end{equation}
and the transverse modulus
\[
 \eta_j(\rho)
 =
 \operatorname*{ess\,sup}_{0<\lambda\leq\rho}
 \bigl\|
 A_j(\lambda,\cdot,\cdot)-A_j^0
 \bigr\|_{\mathrm L^\infty(\mathbb R^n\times\mathbb R)}
\]
satisfies
\begin{equation}\label{eq:smooth-transverse-majorant}
 \eta_j(\rho)\leq\widehat\eta(\rho),
 \qquad 0<\rho\leq1,
\end{equation}
where
\begin{equation}\label{eq:common-smooth-Dini-majorant}
 \widehat\eta(\rho)
 =
 \begin{cases}
  \eta(2\rho),&0<\rho\leq1/4,\\
  2\Lambda,&1/4<\rho\leq1.
 \end{cases}
\end{equation}
In particular,
\begin{equation}\label{eq:smooth-Dini-control}
 \int_0^1\widehat\eta(\rho)^2\,
 \frac{\mathrm d\rho}{\rho}
 \leq
 \mathfrak D(1)^2+4\Lambda^2\log4,
\end{equation}
and, for \(0<R\leq1/4\),
\begin{equation}\label{eq:smooth-Dini-tail}
 \int_0^R\widehat\eta(\rho)^2\,
 \frac{\mathrm d\rho}{\rho}
 =
 \mathfrak D(2R)^2.
\end{equation}

\noindent Suppose that the conclusion of Theorem~\ref{thm:global} has been proved
for smooth coefficients with constants depending only on
\(n,\mu,\Lambda\), the fixed geometric parameters, and the common moduli
\(\varpi\) and \(\widehat\eta\), but not on derivatives of the
coefficients.  Then the same conclusion holds for \(A\), after a harmless
fixed increase of the admissibility constant \(M\).

\noindent
The same approximation principle applies to the complete correctors,
finite-slab solutions, and height coordinates constructed in
Section~\ref{sec:height-coordinate}.  Consequently, those constructions
may first be carried out for smooth coefficients, provided that all
estimates used in passing to the limit depend only on
\(n,\mu,\Lambda\).
\end{lem}

\begin{proof}
We first construct the approximating matrices.  Regard \(A\) as a
measurable function on
\(\mathbb T_X^{n+1}\times\mathbb R_t\), where
\(X=(\lambda,x)\).  Let
\(0<\varepsilon_j<1/12\), \(\varepsilon_j\downarrow0\), and let
\(\rho_{\varepsilon_j}^{\lambda}\),
\(\rho_{\varepsilon_j}^{x}\), and
\(\rho_{\varepsilon_j^2}^{t}\) be nonnegative smooth mollifiers in
\(\lambda\), \(x\), and \(t\), respectively.  The spatial mollifiers are
periodic and all three mollifiers have integral one.  Set
\[
 B_j
 =
 \rho_{\varepsilon_j}^{\lambda}*_{\lambda}
 \rho_{\varepsilon_j}^{x}*_{x}
 \rho_{\varepsilon_j^2}^{t}*_{t}A.
\]
Then \(B_j\) is smooth and periodic in every spatial variable.  Convex
averaging preserves reality and the ellipticity bounds, as well as
symmetry when \(A\) is symmetric.  The approximation-of-the-identity
property gives
\[
 B_j\longrightarrow A
 \quad\text{in }\mathrm L^p_{\mathrm{loc}}(\mathbb R^{n+2})
\]
for every finite \(p\).  This approximation is sufficient for the
corrector and height-coordinate constructions.

To preserve the one-sided boundary trace and its square-Dini modulus, a
minor modification is needed.  Indeed, direct periodic convolution at
\(\lambda=0\) may mix values near \(\lambda=0^+\) with values near
\(\lambda=1^-\), which are not controlled by
\eqref{eq:normal-modulus}.  Full spatial periodicity and
\eqref{eq:uniform-essential-trace} imply that \(A^0\) is one-periodic in
\(x\).  Define
\begin{equation}\label{eq:smoothed-boundary-trace}
 A_j^0
 =
 \rho_{\varepsilon_j}^{x}*_{x}
 \rho_{\varepsilon_j^2}^{t}*_{t}A^0.
\end{equation}
Choose a smooth, one-periodic function
\(\theta_j=\theta_j(\lambda)\) such that
\[
 0\leq\theta_j\leq1,
 \qquad
 \theta_j(\lambda)=0
 \quad\text{if }
 \operatorname{dist}(\lambda,\mathbb Z)\leq2\varepsilon_j,
\]
and
\[
 \theta_j(\lambda)=1
 \quad\text{if }
 \operatorname{dist}(\lambda,\mathbb Z)\geq3\varepsilon_j.
\]
Replacing \(B_j\) by
\begin{equation}\label{eq:boundary-pinned-approximation}
 A_j(\lambda,x,t)
 =
 \theta_j(\lambda)B_j(\lambda,x,t)
 +
 \bigl(1-\theta_j(\lambda)\bigr)A_j^0(x,t)
\end{equation}
gives a smooth spatially periodic matrix satisfying
\[
 A_j(0,x,t)=A_j^0(x,t).
\]
It remains real, symmetric when \(A\) is symmetric, and has the same
ellipticity bounds.  Moreover, \(A_j-B_j\) is supported in periodic
transverse strips whose widths tend to zero.  Since the matrices are
uniformly bounded, \eqref{eq:smooth-approximation-local} follows.

Convolution in the spatial variables is contractive in
\(\mathrm L^\infty\).  Hence
\eqref{eq:smoothed-boundary-trace} and translation invariance give
\[
 \begin{aligned}
 \bigl\|
 A_j^0(\cdot,t)-A_j^0(\cdot,s)
 \bigr\|_{\mathrm L^\infty(\mathbb R^n)}
 &\leq
 \int_{\mathbb R}\rho_{\varepsilon_j^2}^{t}(\tau)
 \bigl\|
 A^0(\cdot,t-\tau)-A^0(\cdot,s-\tau)
 \bigr\|_{\mathrm L^\infty(\mathbb R^n)}
 \,\mathrm d\tau\\
 &\leq\varpi\bigl(|t-s|^{1/2}\bigr),
 \end{aligned}
\]
which proves \eqref{eq:smooth-time-modulus}.

We next verify the transverse estimate.  By construction,
\[
 A_j(\lambda,x,t)=A_j^0(x,t)
 \qquad\text{when }0\leq\lambda\leq2\varepsilon_j.
\]
Suppose that \(0<\rho\leq1/4\) and
\(2\varepsilon_j<\lambda\leq\rho\).  Whenever
\(\theta_j(\lambda)\neq0\), every transverse level sampled in the
definition of \(B_j(\lambda,x,t)\) belongs to
\[
 (\lambda-\varepsilon_j,\lambda+\varepsilon_j)
 \subset(0,2\rho).
\]
Since the same tangential and temporal mollifiers occur in
\(B_j\) and \(A_j^0\), it follows that
\[
 \bigl\|
 B_j(\lambda,\cdot,\cdot)-A_j^0
 \bigr\|_{\mathrm L^\infty(\mathbb R^n\times\mathbb R)}
 \leq\eta(2\rho).
\]
Equation \eqref{eq:boundary-pinned-approximation} therefore gives
\[
 \eta_j(\rho)\leq\eta(2\rho),
 \qquad 0<\rho\leq1/4.
\]
For \(1/4<\rho\leq1\), uniform boundedness gives
\(\eta_j(\rho)\leq2\Lambda\).  This proves
\eqref{eq:smooth-transverse-majorant}.  Equations
\eqref{eq:smooth-Dini-control} and
\eqref{eq:smooth-Dini-tail} follow by integration and the change of
variables \(s=2\rho\).

We now prove stability of parabolic measure.  Let
\[
 \cH_j
 =
 \partial_t-\operatorname{div}_{\lambda,x}
 \bigl(A_j\nabla_{\lambda,x}\bigr),
\]
and let \(\omega_j^P\) and \(\omega^P\) denote the parabolic measures
associated with \(\cH_j\) and \(\cH\), respectively.  For
\(f\in C_0(\mathbb R^n\times\mathbb R)\), let \(u_j\) be the bounded
continuous Dirichlet solution for \(\cH_j\) with boundary values \(f\).
The maximum principle gives
\[
 \|u_j\|_\infty\leq\|f\|_\infty.
\]
Caccioppoli's inequality, the local and boundary De Giorgi-Moser
estimates, and \eqref{eq:smooth-approximation-local} give weak energy
compactness and local uniform compactness.  Since
\(A_j\to A\) strongly in \(\mathrm L^2_{\mathrm{loc}}\), one may pass to
the limit in the weak formulation.  Uniqueness of the bounded continuous
Dirichlet solution identifies every subsequential limit with the
solution \(u\) for \(\cH\).  Consequently,
\begin{equation}\label{eq:parabolic-measure-weak-convergence}
 \int_{\mathbb R^n\times\mathbb R}
 f\,\mathrm d\omega_j^P
 \longrightarrow
 \int_{\mathbb R^n\times\mathbb R}
 f\,\mathrm d\omega^P
 \qquad
 \text{for every }f\in C_0(\mathbb R^n\times\mathbb R).
\end{equation}

Assume that the smooth-coefficient theorem holds with admissibility
constant \(M_0\) and a reverse H\"older constant independent of \(j\).
Set \(M=2M_0\), fix \(\Delta_r=\Delta_r(x_0,t_0)\), and let \(P\) be
\(M\)-admissible for \(\Delta_r\).  For any fixed
\(0<\vartheta<1\), the pole \(P\) is \(M_0\)-admissible for the larger
concentric cube \(\Delta_{(1+\vartheta)r}(x_0,t_0)\).  The uniform
reverse H\"older estimate and
\eqref{eq:parabolic-measure-total-mass} imply
\[
 \left|
 \iint_{\mathbb R^n\times\mathbb R}
 g\,\mathrm d\omega_j^P
 \right|
 \leq
 C|\Delta_{(1+\vartheta)r}|^{-1/2}
 \|g\|_{\mathrm L^2}
\]
whenever
\(g\in C_0^\infty(\Delta_{(1+\vartheta)r})\).  Passing to the limit in
this inequality shows that \(\omega^P\) has an
\(\mathrm L^2\)-density in
\(\Delta_{(1+\vartheta)r}\).  Since
\(\partial\Delta_r\) is contained in this larger cube and has Lebesgue
measure zero, \(\omega^P(\partial\Delta_r)=0\).  It follows from \eqref{eq:parabolic-measure-weak-convergence} that
\[
 \omega_j^P(\Delta_r)\longrightarrow\omega^P(\Delta_r).
\]

Let \(k_j^P\) be the Poisson kernel for \(\cH_j\).  The uniform estimate
on \(\Delta_r\) gives
\[
 \|k_j^P\|_{\mathrm L^2(\Delta_r)}
 \leq
 C|\Delta_r|^{-1/2}\omega_j^P(\Delta_r).
\]
Weak compactness in \(\mathrm L^2(\Delta_r)\), together with
\eqref{eq:parabolic-measure-weak-convergence}, identifies the weak limit
with the density \(k^P\) of \(\omega^P\) on \(\Delta_r\).  Weak lower
semicontinuity and convergence of the masses give
\[
 \|k^P\|_{\mathrm L^2(\Delta_r)}
 \leq
 C|\Delta_r|^{-1/2}\omega^P(\Delta_r).
\]
This is the required reverse H\"older estimate.

Finally, consider the auxiliary functions in
Section~\ref{sec:height-coordinate}.  For each smooth matrix \(A_j\),
construct the normalized complete corrector \(\chi_j^*\) and, for each
fixed positive integer \(N\), the complete slab solution
\(\Phi_{N,j}^*\).  Their energy, local boundedness, and time-derivative
estimates depend only on \(n,\mu,\Lambda\).  On every bounded time
interval they are therefore uniformly bounded in
\(\mathrm L^2_t\mathrm H^1_X\), while their time derivatives are
uniformly bounded in \(\mathrm L^2_t\mathrm H^{-1}_X\).  The local
Aubin-Lions lemma and a diagonal extraction give strong local
\(\mathrm L^2\)-convergence and weak local
\(\mathrm L^2_t\mathrm H^1_X\)-convergence.  Together with
\eqref{eq:smooth-approximation-local}, this permits passage to the weak
corrector and slab equations.  Periodicity and the zero-mean
normalization of the corrector pass to the limit.  Weak continuity of
the trace preserves the Dirichlet values on the transverse faces of the
slabs.  The maximum-principle and bounded-error estimates also pass to
the limit.

After this limit has been taken for every fixed \(N\), the
\(N\to\infty\) argument in Section~\ref{sec:height-coordinate} produces
a function \(\Phi^*\) satisfying
\[
 \cH^*\Phi^*=0,
 \qquad
 \Phi^*=0\quad\text{on }\{\lambda=0\},
 \qquad
 \Phi^*\geq0,
 \qquad
 |\Phi^*-\lambda|\leq C.
\]
The limit is not identically zero, since
\(\Phi^*(\lambda,x,t)\geq\lambda-C\).  The strong maximum principle then
gives
\[
 \Phi^*(\lambda,x,t)>0
 \qquad\text{when }\lambda>0.
\]
Thus the complete corrector, slab solutions, and height coordinate exist
for coefficients which are merely measurable in time.
\end{proof}

\begin{rem}
The preliminary modification near the periodic seam is essential for
the reduction of the full theorem.  Direct periodic convolution in
\(\lambda\) may destroy the uniform square-Dini control at
\(\lambda=0\).  For the height-coordinate construction alone, ordinary
periodic mollification is sufficient because no boundary trace modulus
is involved.  Lemma~\ref{lem:smooth-coefficient-reduction} does not, by itself, reduce
the conditional statement of Theorem~\ref{thm:large-scale} to smooth
coefficients.  Indeed, a local \(\mathrm{RH}_2\) estimate assumed for
\(A\) need not automatically hold uniformly for arbitrary
approximations of \(A\).  The lemma instead reduces
Theorem~\ref{thm:global} and the height-coordinate construction to the
smooth setting.  Once the height coordinate has been obtained for the
original coefficient, the proof of Theorem~\ref{thm:large-scale} applies
directly.
\end{rem}

\section{The effective model}
\label{Homoderv}

We give a brief derivation of the effective operator associated with
the family \(A_\varepsilon(X,t)=A(X/\varepsilon,t)\).  Here and below, by the elliptic flux associated with a function \(v\)
we mean the vector field \(A\nabla_Xv\) appearing inside the spatial
divergence.

Fix a time \(t_0\) for which \(A(\cdot,t_0)\) satisfies the stated
assumptions, and let \(\xi\in\mathbb R^{n+1}\).  To determine the
effective response of the medium to the spatial gradient \(\xi\),
consider the affine function \(\ell_\xi(X)=\xi\cdot X\).  For a constant coefficient matrix, \(\ell_\xi\) solves the corresponding
spatial equation and produces the constant flux \(A\xi\).  For the
oscillatory coefficient \(A(X/\varepsilon,t_0)\), however,
\[
 -\operatorname{div}_X
 \bigl(A(X/\varepsilon,t_0)\nabla_X\ell_\xi\bigr)
 =
 -\frac{1}{\varepsilon}
 \left[
 \operatorname{div}_Y
 \bigl(A(Y,t_0)\xi\bigr)
 \right]_{Y=X/\varepsilon}
\]
in the sense of distributions, and the right-hand side need not
vanish.  We therefore correct the affine function at the microscopic scale by
setting
\[
 \ell_{\varepsilon,\xi}^{t_0}(X)
 =
 \xi\cdot X
 +
 \varepsilon\chi_\xi(X/\varepsilon,t_0).
\]
The factor \(\varepsilon\) makes the correction lower order on the
macroscopic scale, while its gradient contributes at order one,
\[
 \nabla_X\ell_{\varepsilon,\xi}^{t_0}(X)
 =
 \xi+\nabla_Y\chi_\xi(X/\varepsilon,t_0).
\]
Consequently,
\[
 \begin{split}
 &-\operatorname{div}_X
 \left(
 A(X/\varepsilon,t_0)
 \nabla_X\ell_{\varepsilon,\xi}^{t_0}
 \right) =
 -\frac{1}{\varepsilon}
 \left[
 \operatorname{div}_Y
 \left(
 A(Y,t_0)
 \bigl(\xi+\nabla_Y\chi_\xi(Y,t_0)\bigr)
 \right)
 \right]_{Y=X/\varepsilon}.
 \end{split}
\]
This leads to the cell problem
\[
 -\operatorname{div}_Y
 \left(
 A(Y,t_0)
 \bigl(\xi+\nabla_Y\chi_\xi(Y,t_0)\bigr)
 \right)
 =0
 \quad\text{in }\mathbb T^{n+1}.
\]
With this choice, \(\ell_{\varepsilon,\xi}^{t_0}\) is an exact weak
solution of the frozen-time spatial equation.  This calculation is
used only to identify the spatial effective response; it does not
assert that \(\ell_{\varepsilon,\xi}^{t_0}\) solves the full parabolic
equation.

More precisely, for almost every \(t\), the Lax-Milgram theorem gives
a unique function
\[
 \chi_\xi(\cdot,t)
 \in
 \mathrm H^1_{\mathrm{per}}(\mathbb T^{n+1}),
 \qquad
 \int_{\mathbb T^{n+1}}\chi_\xi(Y,t)\,\mathrm dY=0,
\]
such that
\[
 \int_{\mathbb T^{n+1}}
 A(Y,t)
 \bigl(\xi+\nabla_Y\chi_\xi(Y,t)\bigr)
 \cdot\nabla_Y\vartheta(Y)\,\mathrm dY
 =0
\]
for every
\(\vartheta\in\mathrm H^1_{\mathrm{per}}(\mathbb T^{n+1})\).
Moreover,
\[
 \|\nabla_Y\chi_\xi(\cdot,t)\|_{\mathrm L^2(\mathbb T^{n+1})}
 \leq C|\xi|,
\]
uniformly in \(t\).  A standard Galerkin construction permits the
correctors to be chosen measurably in \(t\).  The dependence of \(\chi_\xi\) on \(\xi\) is linear.  Hence, if
\(\chi_j=\chi_{e_j}\), then
\[
 \chi_\xi(Y,t)
 =
 \sum_{j=1}^{n+1}\xi_j\chi_j(Y,t).
\]
The corresponding microscopic flux is
\[
 q_\xi(Y,t)
 =
 A(Y,t)
 \bigl(\xi+\nabla_Y\chi_\xi(Y,t)\bigr).
\]
The cell equation states that
\[
\operatorname{div}_Yq_\xi(\cdot,t)=0,
\]
so the flux is balanced within the periodic cell.

The corrector estimate and the boundedness of \(A\) give
\[
 q_\xi
 \in
 \mathrm L^2_{\mathrm{loc}}
 \bigl(
 \mathbb R_t;
 \mathrm L^2_{\mathrm{per}}(\mathbb T_Y^{n+1})
 \bigr).
\]
Consequently, the periodic averaging lemma, applied in the spatial
variables only, yields
\[
 q_\xi(X/\varepsilon,t)
 \rightharpoonup
 \overline q_\xi(t)
 \quad\text{weakly in }
 \mathrm L^2_{\mathrm{loc}}
 \bigl(\mathbb R_X^{n+1}\times\mathbb R_t\bigr),
\]
where
\[
 \overline q_\xi(t)
 =
 \int_{\mathbb T^{n+1}}q_\xi(Y,t)\,\mathrm dY.
\]
Since \(q_\xi\) depends linearly on \(\xi\), so does
\(\overline q_\xi(t)\).  We therefore define the candidate effective
matrix by
\[
 \widehat A(t)\xi
 :=
 \overline q_\xi(t).
\]
This averaging observation identifies the candidate response to corrected
affine functions.  The two-scale argument below separately identifies the
flux for arbitrary locally energy-bounded solution sequences, see
\cite{Allaire} for the periodic averaging and two-scale compactness
principles used here.
Equivalently,
\[
 \widehat A(t)e_j
 =
 \int_{\mathbb T^{n+1}}
 A(Y,t)
 \bigl(e_j+\nabla_Y\chi_j(Y,t)\bigr)
 \,\mathrm dY,
 \qquad 1\leq j\leq n+1.
\]
The energy estimate for the correctors shows that \(\widehat A(t)\) is
measurable and bounded.  It is also uniformly elliptic.  Indeed,
testing the cell problem with \(\chi_\xi\) gives
\[
 \begin{split}
 \widehat A(t)\xi\cdot\xi
 &=
 \int_{\mathbb T^{n+1}}
 A(Y,t)
 \bigl(\xi+\nabla_Y\chi_\xi\bigr)
 \cdot
 \bigl(\xi+\nabla_Y\chi_\xi\bigr)
 \,\mathrm dY\geq
 \mu
 \int_{\mathbb T^{n+1}}
 \bigl|\xi+\nabla_Y\chi_\xi\bigr|^2
 \,\mathrm dY
 \geq
 \mu|\xi|^2.
 \end{split}
\]
If \(A\) is symmetric, then \(\widehat A(t)\) is symmetric as well.

It is important that one averages the corrected flux rather than the
coefficient itself.  Indeed, although
\[
 \int_{\mathbb T^{n+1}}
 \nabla_Y\chi_\xi(Y,t)\,\mathrm dY=0,
\]
the correlation between \(A\) and \(\nabla_Y\chi_\xi\) generally gives
\[
 \widehat A(t)\xi
 \neq
 \left(
 \int_{\mathbb T^{n+1}}A(Y,t)\,\mathrm dY
 \right)\xi.
\]
The corrector accounts for the microscopic redistribution of the
imposed macroscopic gradient within the periodic medium.

We next verify that \(\widehat A(t)\) is indeed the coefficient matrix
of the limiting operator.  Let \(Q=U\times I\) be a bounded parabolic
cylinder compactly contained in the space-time domain, where \(U\) is a
bounded Lipschitz domain, and let
\(u_\varepsilon\) be a sequence of weak solutions of
\[
 \partial_tu_\varepsilon
 -
 \operatorname{div}_X
 \left(
 A(X/\varepsilon,t)\nabla_Xu_\varepsilon
 \right)
 =0
 \quad\text{in }Q,
\]
satisfying uniform local energy bounds.  The equation also bounds
\(\partial_tu_\varepsilon\) locally in
\(\mathrm L^2\bigl(I;\mathrm H^{-1}(U)\bigr)\).  By the Aubin-Lions lemma,
after passing to a subsequence,
\(u_\varepsilon\to u\) strongly in \(\mathrm L^2(Q)\), while spatial
two-scale compactness gives
\[
 \nabla_Xu_\varepsilon
 \stackrel{2}{\rightharpoonup}
 \nabla_Xu+\nabla_Yu_{\mathrm{corr}},
\]
where \(\stackrel{2}{\rightharpoonup}\) denotes weak spatial two-scale
convergence, and
\[
 u_{\mathrm{corr}}
 \in
 \mathrm L^2
 \left(
 Q;
 \mathrm H^1_{\mathrm{per}}(\mathbb T^{n+1})/\mathbb R
 \right).
\]
We choose the representative of
\(u_{\mathrm{corr}}=u_{\mathrm{corr}}(X,t,Y)\) such that
\[
 \int_{\mathbb T^{n+1}}
 u_{\mathrm{corr}}(X,t,Y)\,\mathrm dY=0
\]
for almost every \((X,t)\).  This function represents the first-order
microscopic correction.  Moreover, periodicity and boundedness imply
\[
 A(X/\varepsilon,t)
 \stackrel{2}{\longrightarrow}
 A(Y,t)
\]
strongly in the spatial two-scale sense in \(\mathrm L^2(Q)\).  The
weak-strong product rule for two-scale convergence therefore gives
\[
 A(X/\varepsilon,t)\nabla_Xu_\varepsilon
 \stackrel{2}{\rightharpoonup}
 A(Y,t)
 \bigl(\nabla_Xu+\nabla_Yu_{\mathrm{corr}}\bigr).
\]
To identify \(u_{\mathrm{corr}}\), test the weak equation with \(\zeta_\varepsilon(X,t)
 =
 \varepsilon\varphi(X,t)\psi(X/\varepsilon)\), where
\(\varphi\in C^\infty_0(Q)\) and
\(\psi\in C^\infty_{\mathrm{per}}(\mathbb T^{n+1})\).  Since
\(\partial_t\zeta_\varepsilon
 =
 \varepsilon\partial_t\varphi\,
 \psi(X/\varepsilon)\) and
\[
 \nabla_X\zeta_\varepsilon
 =
 \varphi\,\nabla_Y\psi(X/\varepsilon)
 +
 \varepsilon\psi(X/\varepsilon)\nabla_X\varphi,
\]
the time term and the term containing
\(\varepsilon\nabla_X\varphi\) vanish as
\(\varepsilon\to0\).  Passing to the limit yields
\[
 \iint_Q\int_{\mathbb T^{n+1}}
 \varphi(X,t)
 A(Y,t)
 \bigl(\nabla_Xu(X,t)+\nabla_Yu_{\mathrm{corr}}(X,t,Y)\bigr)
 \cdot\nabla_Y\psi(Y)
 \,\mathrm dY\,\mathrm dX\,\mathrm dt
 =0.
\]
Since \(\varphi\) is arbitrary, it follows that
\[
 -\operatorname{div}_Y
 \left(
 A(Y,t)
 \bigl(\nabla_Xu(X,t)+\nabla_Yu_{\mathrm{corr}}(X,t,Y)\bigr)
 \right)
 =0
\]
for almost every \((X,t)\).  By uniqueness of the mean-zero solution
of the cell problem,
\[
 u_{\mathrm{corr}}(X,t,Y)
 =
 \sum_{j=1}^{n+1}
 \chi_j(Y,t)\,\partial_{X_j}u(X,t).
\]
Consequently,
\[
 A(X/\varepsilon,t)\nabla_Xu_\varepsilon
 \rightharpoonup
 \widehat A(t)\nabla_Xu
 \quad\text{weakly in }\mathrm L^2_{\mathrm{loc}}.
\]
Passing to the limit in the weak formulation with test functions
independent of the fast variable gives
\[
 \partial_tu
 -
 \operatorname{div}_X
 \bigl(\widehat A(t)\nabla_Xu\bigr)
 =0.
\]
Thus every locally convergent subsequence has a limit governed by
\[
 \widehat\cH
 =
 \partial_t-\operatorname{div}_X
 \bigl(\widehat A(t)\nabla_X\bigr).
\]
Whenever the initial-boundary data converge and the corresponding limiting
initial-boundary value problem has a unique solution, this identification
yields convergence of the full family.

Only the spatial variables are rapidly oscillating in
\(A(X/\varepsilon,t)\).  The physical time \(t\) therefore acts as a
parameter in the cell problem and is not averaged out.  This explains
why no temporal cell problem appears and why the effective matrix
\(\widehat A(t)\) may retain its dependence on \(t\).  The argument does
not differentiate the correctors with respect to time and hence requires
neither time periodicity nor temporal regularity of \(A\) beyond
measurability.

\section{Geometry of the two-comparison argument}
\label{app:comparison-geometry}

The geometry used in Subsection~\ref{subsec:two-comparison} is illustrated
in Figure~\ref{fig:two-comparison-geometry}.  The figure is schematic and
is not drawn to scale.

\begin{figure}[htbp]
\centering
\begin{tikzpicture}[
 x=0.80cm,
 y=0.81cm,
 font=\footnotesize,
 point/.style={circle,fill=black,inner sep=1.2pt},
 pole/.style={circle,draw=black,fill=white,inner sep=1.55pt,line width=.6pt},
 every path/.style={line width=.55pt},
 >={Latex[length=1.8mm]}
]
\begin{scope}
 \node[anchor=west,font=\small] at (0.35,6.28)
   {\textup{(a)} Projection onto \(X=(\lambda,x)\)};
 \fill[black!3] (0.20,0) rectangle (7.55,5.75);
 \draw (0.20,0) rectangle (7.55,5.75);
 \node[anchor=north west] at (0.35,5.65) {\(D_2=\pi_X\Omega_2\)};
 \fill[black!7] (0.70,0) rectangle (6.85,4.65);
 \draw (0.70,0) rectangle (6.85,4.65);
 \node[anchor=north west] at (0.86,4.56) {\(D_1=\pi_X\Omega_1\)};
 \draw[densely dashed] (1.05,0) rectangle (6.45,4.08);
 \node[anchor=north west,fill=white,inner sep=1pt]
   at (1.17,4.00) {\(\pi_XT_{4R}(z_r^c)\)};
 \fill[black!10] (1.45,0) rectangle (6.05,3.34);
 \draw (1.45,0) rectangle (6.05,3.34);
 \node[anchor=north west,fill=black!10,inner sep=1pt]
   at (1.58,3.25) {\(\pi_X\mathcal Q_r\)};
 \fill[black!17] (2.18,0) rectangle (5.35,1.58);
 \draw (2.18,0) rectangle (5.35,1.58);
 \node[anchor=north west,fill=black!17,inner sep=1pt]
   at (2.30,1.50) {\(\pi_X\mathcal Q_r^{\mathrm{red}}\)};
 \draw[->] (0.12,0)--(7.85,0) node[below] {\(x\)};
 \draw[->] (0.20,-0.08)--(0.20,5.98) node[left] {\(\lambda\)};
 \node[anchor=north] at (1.10,-0.13) {\(\{\lambda=0\}\)};
 \node[point,label={[xshift=-1mm]above left:\(A_k\)}] at (2.72,0.27) {};
 \node[point] at (3.02,0.31) {};
 \node[point] at (3.30,0.24) {};
 \node[point,label=below right:\(A_r^*\)] at (4.35,0.91) {};
 \node[point,label=below right:\(A_r^{\mathrm{ref}}\)] at (4.83,2.43) {};
 \node[pole,label=below right:\(P_1\)] at (5.35,2.88) {};
 \node[pole] (p2) at (6.15,5.43) {};
 \node[anchor=west] at (6.31,5.43) {\(P_2\)};
\end{scope}
\begin{scope}[xshift=7.30cm]
 \node[anchor=west,font=\small] at (0.35,6.28)
   {\textup{(b)} Time ordering};
 \draw[->] (0.65,0.25)--(0.65,5.98) node[left] {\(t\)};
 \fill[black!8] (1.05,0.58) rectangle (4.35,2.18);
 \draw[densely dashed] (1.05,0.58) rectangle (4.35,2.18);
 \node[anchor=north west,font=\scriptsize] at (1.13,2.08)
   {common comparison region};
 \node[point,label=right:\(A_r^*\)] at (3.30,1.42) {};
 \node[point] at (1.68,0.91) {};
 \node[point] at (1.95,1.00) {};
 \node[point] at (2.22,0.86) {};
 \node[anchor=north] at (1.95,0.78) {\(\{A_k\}\)};
 \draw (0.53,2.18)--(4.45,2.18);
 \node[anchor=south west] at (1.05,2.22)
   {\(\sup t\bigl(T_{4R}(z_r^c)\bigr)\)};
 \foreach \yy in {3.17,4.27,5.42}
   \draw (0.53,\yy)--(0.77,\yy);
 \node[anchor=west] at (0.88,3.17) {\(t(P_1)\)};
 \node[anchor=west] at (0.88,4.27)
   {\(t(P_2)\)};
 \node[anchor=west] at (0.88,5.42) {\(t(P)\)};
 \draw[<->] (4.62,2.20)--(4.62,3.15)
   node[midway,right] {\(\geq cR^2\)};
 \draw[<->] (4.62,3.19)--(4.62,4.25)
   node[midway,right] {\(\simeq r^2\)};
\end{scope}
\end{tikzpicture}
\caption{Geometry of the two-comparison argument.  The figure is not to
scale.  Panel~\textup{(a)} shows the projection onto the transverse
variable and one tangential spatial variable.  The spatial sections
\(D_1\subset D_2\), the common comparison cylinder \(\mathcal Q_r\), and
its reduced region are nested.  Every \(A_k\) and \(A_r^*\) belongs to the
same reduced region, whereas \(A_r^{\mathrm{ref}}\) is the common
normalizing corkscrew.  The point \(P_2\) belongs to
\(D_2\setminus\overline{D_1}\).  Panel~\textup{(b)}
records the causal geometry.  The apparent overlap of \(P_1\) with
\(\pi_X\mathcal Q_r\) in Panel~\textup{(a)} is only spatial.  The fixed
enlargement \(T_{4R}(z_r^c)\)
lies quantitatively in the strict common past of the three poles, and
\(t(P_1)<t(P_2)<t(P)\).  The forward Harnack chains in
the first Green variable, which are not shown, remain separated from the
second-variable pole \(A_r^*\).}
\label{fig:two-comparison-geometry}
\end{figure}
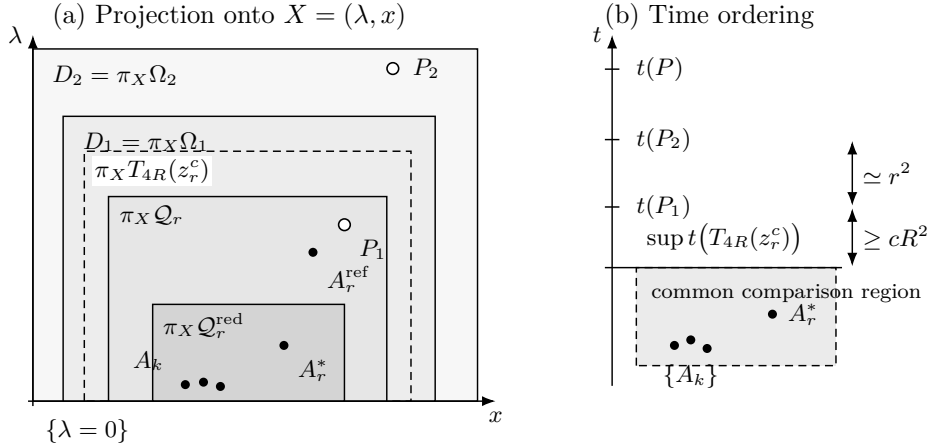


\medskip
\noindent
\textbf{Declaration on the use of generative AI.}
The author used ChatGPT, developed by OpenAI, as an interactive tool
during the preparation and revision of this manuscript.  The tool was
used for language editing, organization, LaTeX formatting, and checks of
clarity and internal consistency.  The mathematical arguments, results,
and conclusions were developed and verified by the author, who also
checked the references and takes full responsibility for the content of
the manuscript.

\end{document}